\documentclass[reqno]{amsart}

\usepackage[utf8]{inputenc}
\usepackage[T1]{fontenc}
\usepackage{amsmath}
\usepackage{amssymb}
\usepackage{amsthm}
\usepackage[english]{babel}
\usepackage{csquotes}
\usepackage[hidelinks]{hyperref}
\usepackage{enumitem}

\numberwithin{equation}{section}

\newtheorem{thm}{Theorem}[section]
\newtheorem{pro}[thm]{Proposition}
\newtheorem{lem}[thm]{Lemma}

\newcommand{\dt}{\partial_t}
\newcommand{\dx}{\partial_x}
\newcommand{\bbt}{\mathbb{T}}
\newcommand{\bbr}{\mathbb{R}}
\newcommand{\bbn}{\mathbb{N}}
\newcommand{\bbz}{\mathbb{Z}}
\newcommand{\lebl}{l}
\newcommand{\lebL}{L}
\newcommand{\lebH}{H}
\newcommand{\norm}[2]{\left\lVert #1\right\rVert_{#2}}
\newcommand{\abs}[1]{\left| #1\right|}
\newcommand{\ftxh}[1]{\hat{#1}}
\newcommand{\fttx}[1]{\mathcal{F}_{t,x}\left[ #1\right]}
\newcommand\bigast{\mathop{\mathchoice{\hbox{\huge$\ast$}}{\ast}{\scriptstyle\ast}{\scriptscriptstyle\ast}}}
\newcommand{\supp}[2][]{\text{supp}_{#1}\,#2}
\newcommand{\vari}{\mathfrak{I}}
\newcommand{\varz}{\mathfrak{V}}
\newcommand{\spx}{\textbf{X}}

\newcommand{\spec}{\mathfrak{E}}
\newcommand{\spfa}{\mathbf{F}}
\newcommand{\spfb}{\mathbf{F}}
\newcommand{\spfc}{\mathfrak{F}}
\newcommand{\spna}{\mathbf{N}}
\newcommand{\spnb}{\mathbf{N}}
\newcommand{\spnc}{\mathfrak{N}}

\begin{document}
	
\title{Well-posedness of the periodic dispersion-generalized Benjamin-Ono equation in the weakly dispersive regime}
\author{Niklas J\"ockel}
\address{Fakult\"at f\"ur Mathematik, Universit\"at Bielefeld, Postfach 100131, 33501 Bielefeld, Germany}
\email{njoeckel@math.uni-bielefeld.de}

\begin{abstract}
	We study the dispersion-generalized Benjamin-Ono equation in the periodic setting. This equation interpolates between the Benjamin-Ono equation ($\alpha=1$) and the viscous Burgers' equation ($\alpha=0$). We obtain local well-posedness in $H^s$ for $s>3/2-\alpha$ and $\alpha\in(0,1)$ by using the short-time Fourier restriction method.
\end{abstract}
\maketitle

\section{Introduction}

In this article we consider the dispersion-generalized Benjamin-Ono equation
\begin{equation}\label{eq_PDE}
	\begin{cases}
		\dt u
		+
		\dx D^\alpha_xu
		&=
		\dx (u^2)
		\\
		\hfill u(0)
		&=
		u_0
	\end{cases}
\end{equation}
posed on $\bbr\times\bbt$ with $\alpha\in(0,1)$. Here, the unknown $u$ is a real-valued function, the initial datum $u_0$ is in a Sobolev space $\lebH^s(\bbt)$ and $D_x^\alpha$ denotes the Fourier multiplier defined by $\mathcal{F}(D_x^\alpha u)(\xi)=\abs{\xi}^\alpha\mathcal{F}(u)(\xi)$.

Famous examples of equations of the form \eqref{eq_PDE} are the Korteweg-de Vries equation ($\alpha=2$) modelling unidirectional nonlinear dispersive waves as well as the Benjamin-Ono equation ($\alpha=1$) which models long internal waves in deep stratified fluids. For non-integer valued $\alpha$, equation \eqref{eq_PDE} can be seen as an interpolation between the two models and for $\alpha=1/2$ it is closely related to the Whitham equation for capillary waves, see below. In the case $\alpha=0$, equation \eqref{eq_PDE} reduces to the viscous Burgers' equation.\\

It is well-known that the Korteweg-de Vries equation as well as the Benjamin-Ono equation are integrable PDEs having infinitely many conserved quantities along their flow, see \cite{KV2019, KLV2023} and the references therein. For general $\alpha$ we still have the  conservation of the mean, the mass and the Hamiltonian. These quantities are given by
\begin{gather*}
	\int u(t,x) dx,
	\qquad
	\int \abs{u(t,x)}^2 dx,
	\qquad
	\int \frac{1}{2}\abs{D^{\alpha/2}_x u(t,x)}^2
	-
	\frac{1}{3} u^3(t,x) dx.
\end{gather*}
If posed on the real line, equation \eqref{eq_PDE} is invariant under the scaling
\begin{align*}
	u(t,x)
	\mapsto
	\lambda^{\alpha}u(\lambda^{1+\alpha}t,\lambda x)
\end{align*}
and hence the equation is scaling-critical in the homogeneous Sobolev space $\dot{\lebH}^{1/2-\alpha}(\bbr)$. In particular, the equation is $\lebL^2$-critical for $\alpha=1/2$ and it is energy-critical for $\alpha=1/3$.\\

The purpose of this article is to establish local well-posedness in $\lebH^s(\bbt)$, $s>3/2-\alpha$, for the Cauchy problem \eqref{eq_PDE} with periodic initial data and $\alpha\in(0,1)$. By local well-posedness, we mean the existence and uniqueness of solutions as well as their continuous dependence on the initial datum. Before stating our main theorem, let us briefly discuss some recent results concerning the well-posedness theory of equation \eqref{eq_PDE} for $\alpha\in[0,2]$ posed on the torus and on the real line.

\subsection{Known results and main theorem}
A classical result due to Kato \cite{Kat1975} yields the local well-posedness of the Cauchy problem \eqref{eq_PDE} in $\lebH^s(\bbt)$ and $\lebH^s(\bbr)$ for any $s>3/2$ and all $\alpha\in[0,2]$. This result is sharp for $\alpha=0$, but the proof does note take advantage of the dispersive structure of \eqref{eq_PDE} for $\alpha\neq 0$.

For the Korteweg-de Vries equation ($\alpha=2$), Bourgain \cite{Bou1993b} proved global well-posedness in $\lebL^2(\bbt)$ by the use of perturbative arguments. Local well-posedness in $\lebH^{-1/2}(\bbt)$ was obtained by Kenig, Ponce and Vega \cite{KPV1996}. This was upgraded to global well-posedness by Colliander, Keel, Staffilani, Takaoka and Tao \cite{CKS2003a}. Relying on the integrability of the equation, Kappeler and Topalov \cite{KT2006a} managed to show global well-posedness in $\lebH^{-1}(\bbt)$, which is sharp. In a recent work, Killip and Vi{\c{s}}an \cite{KV2019} proved global well-posedness in $\lebH^{-1}(\bbt)$ and in $\lebH^{-1}(\bbr)$.

Before we state some results in the range $\alpha<2$, let us recall that the behaviour of equation \eqref{eq_PDE} changes significantly in comparison to the case $\alpha=2$. Indeed, Molinet, Saut and Tzvetkov \cite{MST2001} observed that the dispersion is too weak to deal with the nonlinearity by perturbative means. More precisely, they showed that the frequency-interaction $u_{low}\dx u_{high}$ in the nonlinearity cannot be estimated appropriately in order to allow for a Picard iteration, see also \cite{KT2005, Her2008}.

In the well-posedness theory of the Benjamin-Ono equation ($\alpha=1$) a major breakthrough overcoming latter problem was obtained by Tao \cite{Tao2004}, who proved global well-posedness in $\lebH^1(\bbr)$. He applied a gauge transform which effectively cancels the worst behaving interaction in the nonlinearity. Following the same approach, Ionescu and Kenig \cite{IK2007} established global well-posedness in $\lebL^2(\bbr)$ and Molinet \cite{Mol2008} proved global well-posedness in $\lebL^2(\bbt)$. Relying on the integrability of the equation, Gérard, Kappeler and Topalov \cite{GKT2022} obtained global well-posedness in $\lebH^s(\bbt)$, $s>-1/2$, almost reaching the scaling critical regularity $-1/2$. Recently, Killip, Laurens and Vi{\c{s}}an \cite{KLV2023} obtained global well-posedness in $\lebH^s(\bbt)$ and in $\lebH^s(\bbr)$ for $s>-1/2$.

Next, we consider the case $\alpha\in(1,2)$. Here, Guo \cite{Guo2012} proved local well-posedness in $\lebH^s(\bbr)$, $s>2-\alpha$, following the approach introduced in \cite{IKT2008}, which involves function spaces with frequency-dependent time-localizations. This result was improved by Herr, Ionescu, Kenig and Koch \cite{HIKK2010} by an application of a para-differential gauge transform leading to global well-posedness in $\lebL^2(\bbr)$.
Another approach is due to Molinet and Vento \cite{MV2014} who used modified energy estimates to obtain local well-posedness in $\lebH^s(\bbt)$ for $s>1-\alpha/2$.
Later, following the approach of \cite{IKT2008}, this result was improved by Schippa \cite{Sch2020}. He obtained global well-posedness in $\lebL^2(\bbt)$ for $\alpha\in(3/2,2)$ and local well-posedness in $\lebH^s(\bbt)$ for $s>3/2-\alpha$ and $\alpha\in(1,3/2]$.

Lastly, we consider the weakly dispersive regime $\alpha\in(0,1)$. Linares, Pilod and Saut \cite{LPS2014} obtained local well-posedness in $\lebH^s(\bbr)$, $s>3/2-3\alpha/8$. This has been improved by Molinet, Pilod and Vento \cite{MPV2018} using modified energies to obtain local well-posedness in $\lebH^s(\bbr)$ for $s>3/2-5\alpha/4$.
To the best of our knowledge, there has been no result in the range $\alpha\in(0,1)$ improving on well-posedness in $\lebH^s(\bbt)$ for $s>3/2$. Note that Hur \cite{Hur2018} demonstrated norm inflation in $\lebH^s(\bbt)$, $s<-2$, for $\alpha\in[-1,2)$, yielding a strong form of ill-posedness.\\

Due to the above-mentioned result by Kato \cite{Kat1975}, there exists for any $R>0$ a timespan $T=T(R)>0$ and a continuous data-to-solution map
\begin{align*}
	S^\infty_T:B_R(0)\subset \lebH^\infty(\bbt)\rightarrow C([0,T];\lebH^\infty(\bbt))
\end{align*}
mapping smooth initial data $u_0$ to the corresponding unique smooth solution of \eqref{eq_PDE}. We will prove that the map $S^\infty_T$ extends to a continuous map $S^s_T:B_R(0)\subset \lebH^s(\bbt)\rightarrow C([0,T];\lebH^s(\bbt))$ for every $s>3/2-\alpha$ leading to:
\begin{thm}\label{thm_main}
	Let $\alpha\in(0,1)$ and $s>3/2-\alpha$. The Cauchy problem \eqref{eq_PDE} is locally well-posed in $\lebH^s(\bbt)$.
\end{thm}

\newpage
\subsection{Strategy of the proof}
Let us comment on the strategy of the proof. The general idea is to follow the approach by Ionescu, Kenig and Tataru \cite{IKT2008} involving frequency-dependent time-localized function spaces.
We begin by giving a heuristic argument indicating how these time-localizations improve the estimate of the aforementioned problematic term $u_{low}\dx u_{high}$.
Let $u_K$, respectively $u_N$, denote the localization of a function $u$ to frequencies of size $K$, respectively $N$. Fix $K\leq N$, $s\geq1/2$ and let $I$ be an interval of length $N^{-1}$. Then, we can bound the $L^1(I;\lebH^s(\bbt))$-norm of $u_K\dx u_N$ using Hölder's and Bernstein's inequality by
\begin{align*}
	\norm{u_K\dx u_N}{\lebL^1_I\lebH^s}
	\lesssim
	N^{1+s}K^{1/2}\abs{I}\norm{u_K}{\lebL^\infty_I\lebL^2}\norm{u_N}{\lebL^\infty_I\lebL^2}
	\lesssim
	\norm{u_K}{\lebL^\infty_I\lebH^s}\norm{u_N}{\lebL^\infty_I\lebH^s}.
\end{align*}
Note that the spatial regularity is the same on both sides thanks to the time-localization.

As in \cite{IKT2008}, we will introduce function spaces $\spfc^s_T$, $\spnc^s_T$ and $\spec^s_T$. The first and third space can be thought of as $C([0,T];\lebH^s)$, while the second function space is close to $\lebL^1([0,T];\lebH^s)$. The distinctive feature of the spaces $\spfc^s_T$ and $\spnc^s_T$ is that they are equipped with norms that involve frequency-dependent time-localizations.
These time-localizations allow us to prove estimates of the form \begin{align*}
	\begin{cases}
		\norm{u}{\spfc^s_T}
		&\lesssim
		\norm{u}{\spec^s_T}+\norm{\dx(u^2)}{\spnc^s_T},
		\\
		\norm{\dx(u^2)}{\spnc^s_T}
		&\lesssim
		T^d\norm{u}{\spfc^s_T}^2,
		\\
		\norm{u}{\spec^s_T}^2
		&\lesssim
		\norm{u_0}{\lebH^s}^2
		+
		F_1(T,\norm{u}{\spfc^s_T}),
	\end{cases}
\end{align*}
where $F_1$ is a function and $d\in(0,1]$ is a small but positive number, see Section \ref{s_proof} for the precise estimates. Above, the linear estimate is a consequence of Duhamel's principle, whereas the nonlinear estimate relies on the time-localization -- similar to the heuristic argument above. The energy estimates follow with considerably more effort. Here, we will use quadrilinear estimates, commutator estimates and symmetrizations.

For differences of solutions $v=u_1-u_2$ and $w=u_1+u_2$, we analogously obtain the estimates
\begin{align*}
	\begin{cases}
		\norm{v}{\spfc^{-1/2}_T}
		&\lesssim
		\norm{v}{\spec^{-1/2}_T}+\norm{\dx(vw))}{\spnc^{-1/2}_T},
		\\
		\norm{\dx (vw)}{\spnc^{-1/2}_T}
		&\lesssim
		T^d\norm{v}{\spfc^{-1/2}_T}\norm{w}{\spfc^s_T},
		\\
		\norm{v}{\spec^{-1/2}_T}^2
		&\lesssim
		\norm{v_0}{\lebH^{-1/2}}^2
		+
		F_2(T,\norm{v}{\spfc^{-1/2}_T},\norm{u_1}{\spfc^s_T},\norm{u_2}{\spfc^s_T}),
	\end{cases}
\end{align*}
as well as
\begin{align*}
	\begin{cases}
		\norm{v}{\spfc^s_T}
		&\lesssim
		\norm{v}{\spec^s_T}+\norm{\dx(vw)}{\spnc^s_T},
		\\
		\norm{\dx (vw)}{\spnc^s_T}
		&\lesssim
		T^d\norm{v}{\spfc^s_T}\norm{w}{\spfc^s_T},
		\\
		\norm{v}{\spec^s_T}^2
		&\lesssim
		\norm{v_0}{\lebH^s}^2
		+
		F_3(T,\norm{v}{\spfc^{-1/2}_T},\norm{v}{\spfc^s_T},\norm{u_2}{\spfc^s_T},\norm{u_2}{\spfc^{s+2-\alpha}}).
	\end{cases}
\end{align*}
After bootstrapping the previously given sets of estimates, we get an a priori estimate in $\lebH^s(\bbt)$ for smooth solutions of \eqref{eq_PDE} as well as estimates for the difference of two smooth solutions in $\lebH^{-1/2}(\bbt)$ and in $\lebH^s(\bbt)$. Equipped with these results, we obtain Theorem \ref{thm_main} following the classical Bona-Smith argument \cite{BS1975}.\\

Next, we want to comment on the choice of the time-localization. To prove Theorem \ref{thm_main}, we will choose a time-localization that restricts a function $u_N$ to a time interval of length $N^{-1-\epsilon}$, $\epsilon>0$. Note that the localization does not depend on $\alpha$ and that it fits to the heuristic argument given above.
In comparison to our approach, Guo \cite{Guo2012} and Schippa \cite{Sch2020}, when studying equation \eqref{eq_PDE} in the range $\alpha\in(1,2)$, used a time-localization that restricts a function $u_N$ to an interval of size $N^{\alpha-2-\epsilon}$, $\epsilon>0$. This localization appears in the heuristic argument if we use a bilinear estimate instead of Hölder's inequality. Choosing latter time-localization, we can repeat our proof with minor modifications and obtain an improved a priori estimate, but no improvement for the difference estimates, see Section \ref{ss_apr}.\\

At last, we argue that Theorem \ref{thm_main} holds for a slightly larger class of dispersive PDEs than \eqref{eq_PDE}. As in \cite{MV2014}, let us consider
\begin{equation}\label{eq_PDE_gen}
	\begin{cases}
		\dt u + L_{i\omega} u
		&=
		\dx (u^2),
		\\
		\hfill u(0)
		&=
		u_0.
	\end{cases}
\end{equation}
Here, $L_{i\omega}$ is the Fourier multiplier defined by $\mathcal{F}(L_{i\omega}u)=i\omega\mathcal{F}(u)$ and $\omega$ is an odd, real-valued function in $C^1(\bbr)\cap C^2(\bbr-\{0\})$. We need two further assumptions on $\omega$ to guarantee that Theorem \ref{thm_main} also holds for \eqref{eq_PDE_gen}: Firstly, we require that for some $\kappa>0$ we have the following control of the first two derivatives of $\omega$:
\begin{align*}
	\forall\abs{\xi}>\kappa:
	\qquad
	\abs{\omega'(\xi)}
	\sim
	\abs{\xi}^\alpha,
	\qquad
	\abs{\omega''(\xi)}
	\sim
	\abs{\xi}^{\alpha-1}.
\end{align*}
Secondly, denoting by $\abs{\xi_1^*}$ and $\abs{\xi_3^*}$ the maximum and minimum of $\abs{\xi_1}$, $\abs{\xi_2}$ and $\abs{\xi_3}$, we assume that the resonance function satisfies
\begin{align*}
	\abs{\omega(\xi_1)+\omega(\xi_2)+\omega(\xi_3)}
	\sim
	\abs{\xi_1^*}^\alpha\abs{\xi_3^*}
\end{align*}
for all $\xi_1+\xi_2+\xi_3=0$ with $\abs{\xi_1^*}>\kappa$.
Under these conditions on $\omega$, Theorem \ref{thm_main} also holds for \eqref{eq_PDE_gen}. A non-trivial example is
\begin{align*}
	\omega(\xi)
	=
	\left(\tanh(\xi)\xi(1+\tau\xi^2)\right)^{1/2},
	\qquad
	\tau>0,
\end{align*}
for which equation \eqref{eq_PDE_gen} becomes the Whitham equation for capillary waves, see Chapter $12$ in \cite{Whi1974}.
\subsection{Structure of the paper}
In Section \ref{s_preliminaries} we fix the notation and define the frequency-dependent time-localized function spaces $\spfc^s_T$, $\spnc^s_T$ and $\spec^s_T$. Moreover, we recall some properties to indicate their connection. In particular, we state the linear estimate required for the bootstrap argument.
In Section \ref{s_multi} we prove estimates exploiting the structure of the involved function spaces leading to the nonlinear estimates. We also prove a quadrilinear estimate that is relevant for the energy estimates, which are dealt with in Section \ref{s_energy}. In the last section we briefly recall the arguments due to Bona and Smith and complete the proof of Theorem \ref{thm_main}.


\section{Preliminaries}\label{s_preliminaries}
\subsection{Notation}\label{ss_notation}

For positive real numbers $a,b>0$ we write $a\lesssim b$ if $a\leq Cb$ holds for some constant $C>0$. Moreover, we write $a\gg b$ if $a\not\lesssim b$ as well as $a\sim b$ if $a\lesssim b$ and $b\lesssim a$ are satisfied.

In the following, without mentioning, we assume $t\in\bbr$, $\tau\in\bbr$, $x\in\bbt=\bbr/2\pi$ and $\xi\in\bbz$ to denote variables in time, Fourier-time, space and frequency.
Given two variables $\xi_1$ and $\xi_2$, we occasionally abbreviate their sum by $\xi_{12}=\xi_1+\xi_2$. Moreover, for variables $\xi_1,\dots,\xi_n$, $n\geq 2$, we denote by $\xi_1^*,\dots,\xi_n^*$ a reordering satisfying $\abs{\xi_1^*}\geq\dots\geq\abs{\xi_n^*}$. We set $[n]=\{1,\cdots, n\}$.

The Fourier transform on the torus, respectively on the real line, is given by
\begin{align*}
	&\hat{f}(\xi)
	=
	\mathcal{F}_xf(\xi)
	=
	\int_\bbt f(x)e^{-ix\xi}dx,
	&\mathcal{F}_tg(\tau)
	=
	\int_\bbr g(t)e^{-it\tau}d\tau.
\end{align*}
Here, the indices $x$ and $t$ indicate the transformed variable. We abbreviate $\mathcal{F}_t\mathcal{F}_x$ by $\mathcal{F}_{t,x}$.

We fix an even smooth function $\chi\in C^\infty(\bbr)$ satisfying $1_{[-5/4,5/4]}\leq \chi\leq1_{[-8/5,8/5]}$. For every $k\geq 0$, set
\begin{align*}
	\chi_k(\xi)
	=
	\chi(2^{-k}\xi)-\chi(2^{1-k}\xi),
	\qquad
	\chi_{\leq k}(\xi)
	=
	\sum_{l=0}^k\chi_l(\xi),
	\qquad
	\chi_{> k}(\xi)
	=
	\sum_{l>k}\chi_l(\xi).
\end{align*}
Based on this, we define Littlewood-Paley projectors by
\begin{align*}
	P_ku
	=
	\mathcal{F}_x^{-1}\left(\chi_k\mathcal{F}_xu\right),
	\qquad
	P_{\leq k}u
	=
	\mathcal{F}_x^{-1}\left(\chi_{\leq k}\mathcal{F}_xu\right),
	\qquad
	P_{>k}u
	=
	\mathcal{F}_x^{-1}\left(\chi_{>k}\mathcal{F}_xu\right).
\end{align*}
Similarly, we fix $\eta_0\in C^\infty(\bbr)$ satisfying $1_{[-5/4,5/4]}\leq \eta_0\leq1_{[-8/5,8/5]}$ and define
\begin{align*}
	\eta_k(\xi)
	=
	\eta_0(2^{-k}\xi)-\eta_0(2^{1-k}\xi),
	\qquad
	\eta_{\leq k}(\xi)
	=
	\sum_{l=0}^k \eta_l(\xi)
\end{align*}
for any $k\geq 1$. Mostly, we will use latter functions to restrict the modulation variable $\tau-\omega(\xi)$ to dyadic ranges, where the dispersion relation $\omega$ is given by
\begin{align*}
	\omega(\xi)
	=
	-\xi\abs{\xi}^\alpha.
\end{align*}
\subsection{Function spaces}\label{ss_function_space}

In this section we will define the function spaces $\spfc^s_T$, $\spnc^s_T$ and $\spec^s_T$. Essentially, these function spaces appeared first in \cite{IKT2008} in the context of the analysis of the KP-I equation posed on $\bbr^2$. For an adaption and thorough analysis of these spaces on $\bbt$, we refer to \cite{GO2018}.

The definition of the spaces $\spfc^s_T$ and $\spnc^s_T$ is based on the space $\spx^k$, $k\in\bbn$, given by
\begin{equation*}
	\spx^k
	:=
	\{f\in\lebL^2(\bbr;\lebl^2(\bbz\cap\supp \chi_k)): \norm{f}{\spx^k} = \sum_{l\in\bbn}2^{l/2}\norm{\eta_l(\tau-\omega(\xi))f(\tau,\xi)}{\lebL^2_\tau\lebl^2_\xi} <\infty\}.
\end{equation*}
Note that $\spx^k$ is closely related to the Bourgain space $X^{0,1/2}$ and shares many properties with it. The following proposition recalls three of these properties. The first two properties follow from direct calculations, whereas the third property can be proved similarly to Lemma 3.5 in \cite{GO2018}.
\begin{pro}
	We have the following statements:
	\begin{itemize}
		\item[(1)]	Let $k, L\in\bbn$ and write $\eta_L$ for $\eta_{\leq L}$. For all $f_k\in\spx^{k}$ we have
		\begin{align}\label{eq_est_mod_l2xk}
			\sum_{l\geq L}2^{l/2}\norm{\eta_l(\tau-\omega (\xi))\int_\bbr |f_k(\tau',\xi)|2^{-L}(1+2^{-L}|\tau-\tau'|)^{-4}d\tau'}{\lebl^2_\xi\lebL^2_\tau}
			\lesssim
			\norm{f_k}{\spx^{k}}.
		\end{align}
		\item[(2)]	Let $k,l\in\bbn$, $t_0\in\bbr$ and let $\gamma\in\mathcal{S}(\bbr)$ be a Schwartz function. For all $f_k\in\spx^k$ we have
		\begin{equation}\label{eq_est_xk_loc}
			\norm{\mathcal{F}_{t,x}\left(\gamma(2^l(t-t_0))\mathcal{F}_{t,x}^{-1}f_k\right)}{\spx^k}
			\lesssim_\gamma
			\norm{f_k}{\spx^k}.
		\end{equation}
		\item[(3)] Let $k, L\in\bbn$, $b\in(0,1/2]$, $T\in(0,1]$ and write $\eta_L$ for $\eta_{\leq L}$. Let $I$ be an interval with $0\in I$ such that $\abs{I}\sim T$ holds. For all $u$ supported in $\bbr\times\bbt$ we have
		\begin{equation}\label{eq_est_l2txk}
			\sup_{l\geq L}2^{lb}\norm{\eta_l(\tau-\omega(\xi))\fttx{1_I(t)P_ku}(\tau,\xi)}{\lebl^2_\xi\lebL^2_\tau}
			\lesssim
			T^{1/2-b}\norm{\fttx{P_ku}}{\spx^k}.
		\end{equation}
	\end{itemize}
\end{pro}
Let $\vari\geq 0$ determine the strength of the frequency-dependent time-localization. Then, we define
\begin{align*}
	\spfa^k
	&:=
	\{u_k\in C(\bbr;P_k\lebL^2(\bbt)): \norm{u_k}{\spfa^k} =\sup_{t_k}\norm{\mathcal{F}_{t,x}\left( u_k\eta_0\left(2^{\mathfrak{I} k}(\cdot-t_k)\right)\right)}{\spx^{k}}<\infty\},
	\\
	\spna^k
	&:=
	\{u_k\in C(\bbr;P_k\lebL^2(\bbt)): \norm{u_k}{\spna^k} =\sup_{t_k}\norm{(\tau-\omega(\xi)+i2^{\mathfrak{I}k})^{-1}\mathcal{F}_{t,x}\left( u_k\eta_0\left(2^{\mathfrak{I} k}(\cdot-t_k)\right)\right)}{\spx^{k}}<\infty\}.
\end{align*}
Observe that $\vari=0$ corresponds to a frequency-independent time-localization. It will become evident that choosing $\vari$ slightly larger than one is sufficient for our proof.

For $T\in(0,1]$ we define the corresponding restriction spaces to the interval $[-T,T]$ by
\begin{align*}
	\spfb_{T}^k
	&:=
	\{u_k\in C([-T,T];P_k\lebL^2(\bbt)): \norm{u_k}{\spfb_{T}^k} =\inf_{u_k=\tilde{u}_k|_{[-T,T]}}\norm{\tilde{u}_k}{\spfa_{k}}<\infty\},
	\\
	\spnb^k_T
	&:=
	\{u_k\in C([-T,T];P_k\lebL^2(\bbt)): \norm{u_k}{\spnb^k_T} =\inf_{u_k=\tilde{u}_k|_{[-T,T]}}\norm{\tilde{u}_k}{\spna_{k}}<\infty\}.
\end{align*}

When working with functions in $\spfb^k_T$ or $\spnb^k_T$, it will be sufficient to consider extensions of them by functions in $\spfa^k$ or $\spna^k$, which are supported in $[-2T,2T]$. More precisely, Lemma 3.6 in \cite{GO2018} implies:
\begin{pro}\label{lem_ext_spf_spn}
	Let $k\in\bbn$ and $T\in (0,1]$. For every $u\in\spfb_{T}^k$ and every $v\in\spnb^k_T$ there are extensions $\tilde{u}\in\spfa^k$ and $\tilde{v}\in\spna^k$ having  temporal support in $[-2T,2T]$ and satisfying the estimates
	\begin{equation*}
		\norm{\tilde{u}}{\spfa^k}
		\lesssim
		\norm{u}{\spfb_{T}^k},
		\qquad
		\norm{\tilde{u}}{\spna^k}
		\lesssim
		\norm{u}{\spnb^k_T}.
	\end{equation*}
\end{pro}
We equip the inhomogeneous Sobolev space of real-valued functions on $\bbt$ for $s\in\bbr$ with the norm
\begin{equation*}
	\norm{u}{\lebH^s}^2
	=
	\hat{u}(0)^2
	+
	\sum_{k\in\bbn} 2^{2ks}\norm{P_ku}{\lebL^2}^2.
\end{equation*}
We also write $\lebH^\infty(\bbt) := \bigcap_{s\in\bbr} \lebH^s(\bbt)$ and denote by $\lebH^\infty_c(\bbt)$ the set of $\lebH^\infty(\bbt)$-functions $u$ with mean $\hat{u}(0)=c$.
For $T\in(0,1]$ and $s\in\bbr$, define the function space $\spfc^s_T$ by 
\begin{equation*}
	\spfc^s_T
	:=
	\{u\in C([-T,T];\lebH^\infty_0(\bbt)):\norm{u}{\spfc^s_T}^2 =\sum_{k\in\bbn}2^{2ks}\norm{P_ku}{\spfb_{T}^k}^2<\infty\},
\end{equation*}
the space $\spnc^s_T$ by
\begin{equation*}
	\spnc^s_T
	:=
	\{u\in C([-T,T];\lebH^\infty_0(\bbt)):\norm{u}{\spnc^s_T}^2 =\sum_{k\in\bbn}2^{2ks}\norm{P_ku}{\spnb^k_{T}}^2<\infty\}
\end{equation*}
and the energy space $\spec^s_T$ by
\begin{equation*}
	\spec^s_T
	:=
	\{u\in C([-T,T];\lebH^\infty_0(\bbt)):\norm{u}{\spec^s_T}^2 = \sum_{k\in\bbn} 2^{2ks}\sup_{t_k\in[0,T]}\norm{P_ku(t_k)}{\lebL^2(\bbt)}^2<\infty\}.
\end{equation*}
In the subsequent proposition we recall the continuity of the norms $\spec^s_T$ and $\spnc^s_T$ with respect to $T$ and state a linear estimate which is needed for the bootstrap argument in Section \ref{s_proof}. For the proofs, we refer to Lemma 3.3, Proposition 8.1 and Proposition 4.1 in \cite{GO2018}.
\begin{pro}\label{pro_est_spf}
	Let $T\in(0,1]$ and $s\in\bbr$.
	\begin{itemize}
		\item[(1)]	For all $u\in\spfc^s_T$ we have
		\begin{equation}\label{eq_est_emb_spfc}
			\norm{u}{\lebL^\infty_T\lebH^s}
			\lesssim
			\norm{u}{\spfc^s_T}.
		\end{equation}
		\item[(2)]	
		Let $u, f\in C([-1,1];\lebH^\infty_0(\bbt))$. Then $T\mapsto\norm{u}{\spec^s_T}$ and $T\mapsto\norm{f}{\spnc^s_T}$ are non-decreasing continuous functions for $T\in(0,1]$. Moreover, we have
		\begin{align*}
			\lim_{T\rightarrow 0} 	\norm{u}{\spec^s_T}+\norm{f}{\spnc^s_T}
			=
			\norm{u_0}{\lebH^s}.
		\end{align*}
		\item[(3)]	Let $u,f\in C([-T,T];\lebH^\infty_0(\bbt))$ satisfy
		\begin{equation*}
			\partial_tu
			+
			\dx D^\alpha_x u
			=
			f\,\,\text{in }(-T,T)\times\bbt.
		\end{equation*}
		Then, we have
		\begin{equation}\label{eq_est_spf}
			\norm{u}{\spfc^s_T}
			\lesssim
			\norm{u}{\spec^s_T}+\norm{f}{\spnc^s_T}.
		\end{equation}
	\end{itemize}
\end{pro}
In particular, we conclude from the proposition that every smooth solution $u$ of \eqref{eq_PDE} with mean zero is contained in $\spfc^s_T$ and in $\spec^s_T$, whereas the nonlinearity $\dx(u^2)$ is an element of $\spnc^s_T$.
\section{Multilinear estimates}\label{s_multi}

In this section we derive the estimate of the nonlinearity for the bootstrap argument as well as a quadrilinear estimate that is used repetitively to obtain the energy estimates in the next section. As we will see, we need to exploit the frequency-dependent time-localized structure of the function spaces $\spfc_T^s$ and $\spnc_T^s$ for both estimates.\\

Before we start with the estimate of the nonlinearity, let us provide two auxiliary lemmata. The first one deals with convolution estimates for functions localized both in modulation and frequency.

\begin{lem}\label{lem_con}
	For each $i\in[4]$, choose $k_i$, $l_i\in\bbn$ and $f_i\in\lebL^2(\bbr;\lebl^2(\bbz;\bbr_{\geq 0}))$ such that
	\begin{equation}\label{def_supp}
		\supp f_i
		\subset
		D_{l_i,k_i}
		:=
		\{(\tau,\xi)\in \bbr\times\bbz : 	\tau-\omega(\xi)\in\supp{\eta_{l_i}}, \xi\in\supp{\chi_{k_i}}\}.
	\end{equation}
	Then, we have
	\begin{align}
		\overset{3}{\underset{i=1}{\bigast}} f_i(0,0)
		&\lesssim
		2^{l_3^*/2}2^{k_3^*/2}\prod_{i=1}^3\norm{f_i}{\lebL^2_\tau\lebl^2_\xi},\label{case_tri_generic}
		\\
		\overset{4}{\underset{i=1}{\bigast}} f_i(0,0)
		&\lesssim
		2^{(l_3^*+l_4^*)/2}2^{(k_3^*+k_4^*)/2}\prod_{i=1}^4\norm{f_i}{\lebL^2_\tau\lebl^2_\xi}.\label{case_quad_generic}
	\end{align}
\end{lem}
\begin{proof}
	Both estimates follow directly from Young's and Hölder's inequality in combination with property \eqref{def_supp}.
\end{proof}

Next, we collect some simple observations for later reference.
\begin{lem}\label{lem_gam_eta}
	Choose $\gamma$ with $1_{[-1/4,1/4]}\leq \gamma \leq 1_{[-1,1]}$ and recall that $1_{[-5/4,5/4]}\leq \eta_0\leq 1_{[-8/5,8/5]}$ (see Section \ref{s_preliminaries}). For $\beta\geq0$, $\vari>0$, set $\varz:=\beta\vari$. Then, we have
	\begin{align}
		k+\beta\geq k'
		&\Longrightarrow
		\forall t,m\in\bbr:
		\gamma(2^{k\vari+\varz}t+m)
		=
		\gamma(2^{k\vari+\varz}t+m)\eta_0(2^{k'\vari}t+m)\label{eq_gam_eta_2},
		\\
		k\geq k'
		&\Longrightarrow
		\abs{\{m\in\bbz:\gamma(2^{k\vari+\varz}t+m)\eta_0(2^{k'\vari}t)\neq 0\}}\leq 2^{2+\varz+\vari(k-k')}\label{eq_gam_eta_3}.
	\end{align}
\end{lem}
\subsection{Estimate of the nonlinearity}\label{ss_estimate_nonlinearity}

We continue with proving the nonlinear estimate for the bootstrap argument. More precisely, we need to bound the $\spnc^r_T$-norm of $\smash{\dx(u^2)}$ as well as the $\smash{\spnc^{-1/2}_T}$- and $\smash{\spnc^{s}_T}$-norm of $\dx(vw)$, where $u$, $u_1$ and $u_2$ are solutions of \eqref{eq_PDE} and $v=u_1-u_2$, $w=u_1+u_2$.
The proof makes use of the structure of the spaces $\spnc^\bullet_T$ and $\spfc^\bullet_T$, the strength of the time-localization $\vari$ as well as of the estimate \eqref{case_tri_generic}. Notably, as \eqref{case_tri_generic} is independent of $\alpha$, the obtained bound holds independently of the precise choice of $\alpha\in(0,1)$ and for arbitrary, but sufficiently regular functions.

\begin{lem}\label{lem_est_spn}
	Let $\vari>1$, $\alpha\in(0,1)$ and $T\in(0,1]$. For all $b\geq-1/2$ and $c>1/2$ there exists $d>0$ such that for all $f\in\spfc^b_T, g\in\spfc^c_T$ we have
	\begin{equation}\label{eq_est_spn}
		\norm{\dx (fg)}{\spnc^b_T}
		\lesssim
		T^d\norm{f}{\spfc^b_T}\norm{g}{\spfc^c_T}.
	\end{equation}
\end{lem}
\begin{proof}
	According to Proposition \ref{lem_ext_spf_spn}, there exist extensions $f_{k_1}$ and $g_{k_2}$ of $P_{k_1}f$ and $P_{k_2}g$ satisfying
	\begin{align}\label{p_est_spn_ext}
		\begin{split}
			\norm{f_{k_1}}{\spfa^{k_1}}
			&\lesssim
			\norm{P_{k_1}f}{\spfb^{k_1}_T},
			\,\,
			\supp[t]{f_{k_1}} \subset [-2T,2T],
			\\
			\norm{g_{k_2}}{\spfa^{k_2}}
			&\lesssim
			\norm{P_{k_2}g}{\spfb^{k_2}_T},
			\,\,
			\supp[t]{g_{k_2}}\subset [-2T,2T].
		\end{split}
	\end{align}
	Here, the implicit constants are independent of $k_1$, $k_2$, $f$ and $g$.
	Since $f_{k_1}g_{k_2}$ is an extension of $P_{k_1}fP_{k_2}g$, the definition of $\spnc^b_T$ yields
	\begin{align}\label{p_est_spn_red}
		\norm{\dx (fg)}{\spnc^b_T}^2
		\leq
		\sum_{k_3}\left[2^{k_3b}\sum_{k_1,k_2}\norm{P_{k_3}\dx\left(f_{k_1}g_{k_2}\right)}{\spna^{k_3}}\right]^2.
	\end{align}
	In the following, we will restrict the right-hand side of \eqref{p_est_spn_red} to the index sets $\{k_1\sim k_3\gg k_2\}$, $\{k_2\sim k_3\gg k_1\}$, $\{k_1\sim k_2\gg k_3\}$, $\{k_1\sim k_3\sim k_2\}$ and $\{k_1, k_2, k_3 \leq 100\}$ and then estimate each contribution separately. Here, we understand $k_i\sim k_j$ as $\abs{k_i-k_j}<3$ and $k_1\sim k_3\gg k_2$ as $k_1\sim k_3$, $k_3\gg k_2$ and $k_1\gg k_2$.\\
	
	\textbf{Estimating the contribution for $k_1\sim k_3\gg k_2$.}
	We choose $\gamma\in C^\infty(\bbr)$ with ${1_{[-1/4,1/4]}\leq\gamma\leq1_{[-1,1]}}$ and ${\sum_{m\in\bbz}\gamma^2(\cdot+m)=1}$. Also, we abbreviate $mod_r(\tau,\xi):=(\tau-\omega(\xi)+i2^{\vari r})^{-1}$ and set $\varz =3\vari$. Then, the definition of $\spna^{k_3}$ implies
	\begin{align*}
		\norm{P_{k_3}\dx\left(f_{k_1}g_{k_2}\right)}{\spna^{k_3}}
		&=
		\sup_{t_{k_3}}\norm{mod_{k_3}(\tau,\xi)\fttx{\eta_0(2^{\vari k_3}(t-t_{k_3}))\dx\left(f_{k_1}g_{k_2}\right)}}{\spx^{k_3}}
		\\&\leq
		2^{k_3}\sup_{t_{k_3}}\sum_{m\in\bbz}\norm{mod_{k_3}(\tau,\xi)\fttx{\left(\gamma^2(2^{\varz}\cdot-m)\eta_0(\cdot)\right)(2^{\vari k_3}(t-t_{k_3}))f_{k_1}g_{k_2}}}{\spx^{k_3}}.
	\end{align*}
	Note that $k_3+3\geq k_1, k_2$ holds. Hence by \eqref{eq_gam_eta_2}, we obtain
	\begin{align*}
		\gamma^2(2^{k_3\vari+\varz}t-m)=\gamma^2(2^{k_3\vari+\varz}t-m)\eta_0(2^{k_1\vari}t-m)\eta_0(2^{k_2\vari}t-m).
	\end{align*}
	With
	\begin{align*}
		h_{k_1}
		&:=
		\fttx{\gamma(2^{k_3\vari+\varz}(t-t_{k_3})-m)\eta_0(2^{k_1\vari}(t-t_{k_3})-m)f_{k_1}(t,x)},
		\\
		h_{k_2}
		&:=
		\fttx{\gamma(2^{k_3\vari+\varz}(t-t_{k_3})-m)\eta_0(2^{k_3\vari}(t-t_{k_3}))\eta_0(2^{k_2\vari}(t-t_{k_3})-m)g_{k_2}(t,x)},
	\end{align*}
	we can write
	\begin{align*}
		\norm{P_{k_3}\dx\left(f_{k_1}g_{k_2}\right)}{\spna^{k_3}}
		\lesssim
		2^{k_3}\sup_{t_{k_3}}\sum_{m\in\bbz}\norm{mod_{k_3}(\tau,\xi)(h_{k_1}\ast h_{k_2})}{\spx^{k_3}}.
	\end{align*}
	Observe that \eqref{eq_gam_eta_3} implies that the $\spx^{k_3}$-norm above vanishes for all but at most $2^{2+6\vari}$ many values of $m$. By the definition of $\spx^{k_3}$, we get
	\begin{align*}
		&2^{k_3}\sup_{t_{k_3}}\sum_{m\in\bbz}\norm{mod_{k_3}(\tau,\xi)\left(h_{k_1}\ast h_{k_2}\right)}{\spx^{k_3}}
		\\
		&\lesssim
		2^{k_3}\sup_{t_{k_3},m}\left(\sum_{l_3\leq\lfloor k_3\vari\rfloor }+\sum_{l_3> \lfloor k_3\vari\rfloor}\right)2^{l_3/2}\norm{\chi_{k_3}(\xi)\eta_{l_3}(\tau-\omega(\xi))mod_{k_3}(\tau,\xi)\left(h_{k_1}\ast h_{k_2}\right)}{\lebL^2_\tau\lebl^2_\xi}
		\\
		&\lesssim
		2^{k_3}\sup_{t_{k_3},m}2^{-\lfloor\vari k_3\rfloor/2}\norm{1_{D_{\leq \lfloor k_3\vari\rfloor,k_3}}\left(h_{k_1}\ast h_{k_2}\right)}{\lebL^2_\tau\lebl^2_\xi}
		+
		\sum_{l_3> \lfloor k_3\vari\rfloor}2^{-l_3/2}\norm{1_{D_{l_3,k_3}}\left(h_{k_1}\ast h_{k_2}\right)}{\lebL^2_\tau\lebl^2_\xi}.
	\end{align*}
	Above, we used the notation introduced in \eqref{def_supp} and write $D_{\leq \lfloor k_3\vari\rfloor,k_3}:=\cup_{l \leq \lfloor k_3\vari\rfloor} D_{l,k_3}$. Moreover, in the last inequality, we used 
	\begin{align*}
		\abs{\eta_{l_3}(\tau-\omega(\xi))mod_{k_3}(\tau,\xi)}\sim\abs{\eta_{l_3}(\tau-\omega(\xi))(2^{l_3}+i2^{\vari k_3})^{-1}},
	\end{align*}
	which yields the estimate
	\begin{align*}
		\abs{\eta_{l_3}(\tau-\omega(\xi))mod_{k_3}(\tau,\xi)}\lesssim 2^{-\lfloor\vari k_3\rfloor}\abs{\eta_{\leq \lfloor\vari k_3\rfloor}(\tau-\omega(\xi))}
	\end{align*}
	for $l_3\leq\lfloor\vari k_3\rfloor$. Next, we localize each $h_{k_i}$ in modulation. For $j\in[2]$ define
	\begin{align*}
		h_{l_j,k_j}(\tau,\xi):=
		\begin{cases}
			\eta_{l_j}(\tau-\omega(\xi))h_{k_j}(\tau,\xi)
			&\text{if }l_j>\lfloor k_3\vari+\varz\rfloor,
			\\
			\eta_{\leq \lfloor k_3\vari+\varz\rfloor}(\tau-\omega(\xi))h_{k_j}(\tau,\xi)
			&\text{if }l_j=\lfloor k_3\vari+\varz\rfloor.
		\end{cases}
	\end{align*}
	Now, the functions -- as well as their convolution -- are localized in modulation and frequency. Assume $l_3^*\leq l_2$ and $k_3^*=k_2$. We continue the estimation above by an application of \eqref{case_tri_generic} leading to
	\begin{align*}
		&\lesssim
		2^{k_3}\sup_{t_{k_3},m}\sum_{\substack{l_3\geq \lfloor k_3\vari\rfloor\\ l_1,l_2\geq \lfloor k_3\vari+\varz\rfloor}}2^{-l_3/2}\norm{1_{D_{l_3,k_3}}\left(h_{l_1,k_1}\ast h_{l_2,k_2}\right)}{\lebL^2_\tau\lebl^2_\xi}
		\\&\lesssim
		2^{k_2/2}2^{k_3}\sup_{t_{k_3},m}\sum_{l_3\geq \lfloor k_3\vari\rfloor}2^{-l_3/2}\left[\sum_{l_1\geq \lfloor k_3\vari+\varz\rfloor}\norm{h_{l_1,k_1}}{\lebL^2_\tau\lebl^2_\xi}\right]\left[\sum_{l_2\geq \lfloor k_3\vari+\varz\rfloor}2^{l_2/2}\norm{h_{l_2,k_2}}{\lebL^2_\tau\lebl^2_\xi}\right].
	\end{align*}
	Fix $d\in[0,1/2)$. We estimate the sum over $l_1$ by combining Hölder's inequality with estimate \eqref{eq_est_l2txk}, which is applicable due to $f_{k_1}=1_{[-2T,2T]}f_{k_1}$. Thus, we obtain
	\begin{align*}
		\sum_{l_1\geq \lfloor k_3\vari+\varz\rfloor}\norm{h_{l_1,k_1}}{\lebL^2_\tau\lebl^2_\xi}
		\lesssim
		2^{k_3\vari (d-1/2)}
		\sup_{l_1\geq \lfloor k_3\vari+\varz\rfloor}2^{l_1(1/2-d)}\norm{h_{l_1,k_1}}{\lebL^2_\tau\lebl^2_\xi}
		\lesssim
		T^d2^{k_3\vari(d-1/2)}
		\norm{h_{k_1}}{\spx^{k_1}}.
	\end{align*}
	Using \eqref{eq_est_mod_l2xk}, the sum over $l_2$ can be estimated by
	\begin{align*}
		\sum_{l_2\geq \lfloor k_3\vari+\varz\rfloor}2^{l_2/2}\norm{h_{l_2,k_2}}{\lebL^2_\tau\lebl^2_\xi}
		\lesssim
		\norm{\mathcal{F}_{t,x}\left[\gamma(2^{k_3\vari+\varz}(t-t_{k_3})-m)\eta_0(2^{k_2\vari}(t-t_{k_3})-m)g_{k_2}\right]}{\spx^{k_2}}.
	\end{align*}
	Now, we apply \eqref{eq_est_xk_loc} multiple times, loosing the additional time-localizations $\gamma$. Then, we take the supremum in $t_{k_3}$ for each factor (loosing the dependence on $m$ and $k_3$), evaluate the sum over $l_3$ and apply the bounds in \eqref{p_est_spn_ext}. It follows
	\begin{align*}
		\norm{P_{k_3}\dx\left(f_{k_1}g_{k_2}\right)}{\spna^{k_3}}
		\lesssim
		T^d2^{k_2/2}2^{k_3(1-\vari(1-d))}
		\norm{P_{k_1}f}{\spfb^{k_1}_T}
		\norm{P_{k_2}g}{\spfb^{k_2}_T}.
	\end{align*}
	For $\vari>1$ we can find $\delta, \epsilon\in(0,1/2)$ such that $1<\vari(1-\delta)-\epsilon$ holds. Then, we obtain
	\begin{align*}
		\text{RHS}\eqref{p_est_spn_red}|_{k_1\sim k_3\gg k_2}
		&\lesssim
		T^{2d}\sum_{k_3}2^{2k_3(1+b-\vari(1-\delta))}\left[\sum_{k_1}\sum_{k_2}2^{k_2/2}\norm{P_{k_1}f}{\spfb^{k_1}_T}\norm{P_{k_2}g}{\spfb^{k_2}_T}\right]^2
		\\&\lesssim
		T^{2d}\norm{g}{\spfc^{1/2}_T}^2\sum_{k_3}2^{2k_3(1+b-\vari(1-\delta)+\epsilon)}\norm{P_{k_1}f}{\spfb^{k_1}_T}^2
		\lesssim
		T^{2d}\norm{f}{\spfc^b_T}^2\norm{g}{\spfc^{1/2}_T}^2.
	\end{align*}
	Hence, $\eqref{eq_est_spn}$ follows for $d\in(0,1-1/\vari)$, whenever $k_1\sim k_3\gg k_2$ holds.\\
	
	\textbf{Estimating the contribution for $k_2\sim k_3\gg k_1$.}
	We may repeat all steps of the first case. Note that the application of \eqref{case_tri_generic} will yield the factor $2^{k_1/2}$ instead of $2^{k_2/2}$. For $b<1/2$, we obtain
	\begin{align*}
		\text{RHS}\eqref{p_est_spn_red}|_{k_2\sim k_3\gg k_1}
		&\lesssim
		T^{2d}\sum_{k_3}2^{2k_3(1+b-\vari(1-d))}\left[\sum_{k_1}\sum_{k_2}2^{k_1/2}\norm{P_{k_1}f}{\spfb^{k_1}_T}\norm{P_{k_2}g}{\spfb^{k_2}_T}\right]^2
		\\&\lesssim
		T^{2d}\norm{f}{\spfc^b_T}^2\sum_{k_3}2^{2k_3(1+b-\vari(1-d)+\epsilon+1/2-b)}\norm{P_{k_2}g}{\spfb^{k_2}_T}^2
		\lesssim
		T^{2d}\norm{f}{\spfc^b_T}^2\norm{g}{\spfc^{1/2}_T}^2.
	\end{align*}
	Similarly, the estimate follows for $b\geq 1/2$.\\
	
	\textbf{Estimating the contribution for $k_1\sim k_2\gg k_3$.}
	Without loss of generality, we can assume $k_1\geq k_2$. According to \eqref{eq_gam_eta_2}, we have
	\begin{align*}
		\gamma^2(2^{k_1\vari}t-m)=\gamma^2(2^{k_1\vari}t-m)\eta_0(2^{k_1\vari}t-m)\eta_0(2^{k_2\vari}t-m)
	\end{align*}
	and the right-hand side does not vanish for at most $2^{(k_1-k_3)\vari}$ many values of $m$. Proceeding as in the first case, we obtain the estimate
	\begin{align*}
		\norm{P_{k_3}\dx\left(f_{k_1}g_{k_2}\right)}{\spna^{k_3}}
		\lesssim
		2^{k_1\vari+k_3(1-\vari)}\sup_{t_{k_3},m}\sum_{\substack{l_3\geq \lfloor k_3\vari\rfloor\\ l_1,l_2\geq \lfloor k_1\vari\rfloor}}2^{-l_3/2}\norm{1_{D_{l_3,k_3}}(h_{l_1,k_1}\ast h_{l_2,k_2})}{\lebL^2_\tau\lebl^2_\xi}.
	\end{align*}
	Then, together with \eqref{case_tri_generic} and $l_3^*\leq l_3(1/2-\epsilon)+l_1\epsilon$, $\epsilon>0$, it follows
	\begin{align*}
		\norm{P_{k_3}\dx\left(f_{k_1}g_{k_2}\right)}{\spna^{k_3}}
		\lesssim
		T^{d}2^{k_1\vari(\epsilon+d)}2^{k_3(3/2-\vari(1+\epsilon))}\norm{P_{k_1}f}{\spfb^{k_1}_T}\norm{P_{k_2}g}{\spfb^{k_2}_T}.
	\end{align*}
	Due to $1/2+b\geq 0$, this yields
	\begin{align*}
		\text{RHS}\eqref{p_est_spn_red}|_{k_1\sim k_2\gg k_3}
		\lesssim
		T^{2d}\sum_{k_3}2^{2k_3(1-\vari(1+\epsilon))}\left[\sum_{k_1}\sum_{k_2}2^{k_1\left(\vari(\epsilon+d)+1/2+b\right)}\norm{P_{k_1}f}{\spfb^{k_1}_T}\norm{P_{k_2}g}{\spfb^{k_2}_T}\right]^2.
	\end{align*}
	Then, choosing $\epsilon$, $d>0$ such that $1<\vari(1+\epsilon)$ and $\vari(\epsilon+d)+1/2<c$ hold, we conclude
	\begin{align*}
		\text{RHS}\eqref{p_est_spn_red}|_{k_1\sim k_2\gg k_3}
		\lesssim
		T^{2d}\norm{f}{\spfc_{T}^{b}}^2\norm{g}{\spfc_{T}^c}^2
	\end{align*}
	for all $b\geq -1/2$ and $c>1/2$.\\
	
	\textbf{Estimating the contributions for $k_1\sim k_2\sim k_3$ and $k_1, k_2, k_3\leq 100$.}
	In these cases, we can repeat the proof of the first case by choosing $\varz=6\vari$, respectively $\varz=100\vari$.\\
\end{proof}

\subsection{Quadrilinear estimate}\label{ss_quadrilinear_estimate}

In this section we will use \eqref{case_quad_generic} in order to derive a quadrilinear estimate.\\

Let $\varphi:\{(\xi_1,\xi_2,\xi_3,\xi_4)\in(\bbz-\{0\})^4:\xi_{1234}=0\}\rightarrow\bbr$. Define the multilinear convolution operator induced by $\varphi$ via
\begin{align*}
	S_{\varphi}(u_1,u_2,u_3,u_4)(T)
	=
	\int_0^T\sum_{\xi_{1234}=0} \varphi(\xi_1,\xi_2,\xi_3,\xi_4)\ftxh{u}_1(\xi_1)\ftxh{u}_2(\xi_2)\ftxh{u}_3(\xi_3)\ftxh{u}_4(\xi_4).
\end{align*}
Then, we have the following estimate:
\begin{lem}\label{lem_i4}
	Let $T\in(0,1]$. For every $d\in[0,1/2)$ we have
	\begin{align*}
		\abs{S_{\varphi}(u_1,u_2,u_3,u_4)(T)
		}
		\lesssim
		T^d\sum_{k_1,k_2,k_3,k_4}\sup_{\substack{\xi_i\in\supp\chi_{k_i}\\i\in[4]}}\abs{\varphi(\xi_1,\xi_2,\xi_3,\xi_4)}2^{k_3^*/2}2^{k_4^*/2}\prod_{j=1}^4\norm{P_{k_j}u_j}{\spfb^{k_j}_T}.
	\end{align*}
\end{lem}
\begin{proof}
	We localize each $\hat{u}_i=\sum_j \chi_{k_j}\hat{u}_i$, $i\in[4]$, apply the triangle inequality and obtain 
	\begin{align*}
		\abs{S_{\varphi}(u_1,u_2,u_3,u_4)(T)
		}
		&\lesssim
		\sum_{k_1, k_2, k_3, k_4}\abs{\int_0^T\sum_{\xi_{1234}=0}
		\varphi(\xi_1,\xi_2,\xi_3,\xi_4)\prod_{j=1}^4\chi_{k_j}(\xi_j)\ftxh{u}_j(\xi_j)}
		\\&\lesssim
		\sum_{k_1, k_2, k_3, k_4}\sup_{\substack{\xi_i\in\supp\chi_{k_i}\\i\in[4]}}\abs{\varphi(\xi_1,\xi_2,\xi_3,\xi_4)}
		\int_0^T\sum_{\xi_{1234}=0}\,
		\prod_{j=1}^4\chi_{k_j}(\xi_j)\abs{\ftxh{u}_j(\xi_j)}.
	\end{align*}
	It remains to prove
	\begin{align*}
		S_{k_1,k_2,k_3,k_4}
		:=
		\int_0^T\sum_{\xi_{1234}=0}\,
		\prod_{j=1}^4\chi_{k_j}(\xi_j)\abs{\ftxh{u}_j(t,\xi_j)} dt
		\lesssim
		T^d2^{k_3^*/2}2^{k_4^*/2}\prod_{i=1}^4\norm{P_{k_i}u_i}{\spfb^{k_i}_T}.
	\end{align*}
	Let us denote by $\underline{u}_i$ the function determined by $\hat{\underline{u}}_i=\abs{\hat{u}_i}$. Due to the definition of $\spfb^{k_i}_T$ and Proposition \ref{lem_ext_spf_spn}, there exist for each $i\in[4]$ an extension $v_{k_i,i}\in\spfa^{k_i}$ of $1_{[0,T]}P_{k_i}\underline{u}_i$ such that
	\begin{align}\label{p_est_i4n_ext}
		\supp[t]{v_{k_i,i}}\subset [-2T,2T],
		\qquad
		\norm{v_{k_i,i}}{\spfa^{k_i}}
		\lesssim
		\norm{P_{k_i}\underline{u}_i}{\spfb^{k_i}_T}
		=
		\norm{P_{k_i}u_i}{\spfb^{k_i}_T}
	\end{align}
	hold. More precisely, we interpret $1_{[0,T]}P_{k_i}\underline{u}_i$ as an element of the space $\spfb^{k_i}_{[0,T]}$ (that is $\spfa^{k_i}$ restricted to $[0,T]$), apply Proposition \ref{lem_ext_spf_spn} for that space and use that any extension of $1_{[-T,T]}P_{k_i}\underline{u}_i$ is already an extension of $1_{[0,T]}P_{k_i}\underline{u}_i$.
	Next, we choose $\gamma\in C^\infty(\bbr)$ satisfying $1_{[-1/4,1/4]}\leq\gamma\leq1_{[-1,1]}$ and $\sum_{m\in\bbz}\gamma^4(\cdot+m)=1$. According to Lemma \ref{lem_gam_eta}, we have 
	\begin{align}\label{p_est_i4n_gam}
		h_{k_j,m}(t)
		:=
		\gamma(2^{k_1^*\vari}t-m)\eta_0(2^{k_j\vari}t-m)
		=
		\gamma(2^{k_1^*\vari}t-m)
	\end{align}
	for each $j\in[4]$ and conclude
	\begin{align*}
		S_{k_1,k_2,k_3,k_4}
		&=
		\int_0^T\int_\bbr
		\bigast_{i=1}^4P_{k_i}\underline{u}_i(t,x) dxdt
		=
		\int_0^T\int_\bbr
		\bigast_{i=1}^4P_{k_i}v_{k_i,i}(t,x) dxdt
		=
		\int_0^T\sum_{\xi_{1234}=0}\,
		\prod_{i=1}^4\hat{v}_{k_i,i}(t,\xi_i) dt\nonumber
		\\&=
		\int_0^T\mathcal{F}_x\left[\prod_{i=1}^4 v_{k_i,i}(t,\cdot)\right](0) dt
		=
		\int_\bbr1_{[0,T]}(t)\sum_{m\in\bbz}\mathcal{F}_x\left[
		\prod_{i=1}^4h_{k_i,m}(t)v_{k_i,i}(t,\cdot)\right](0) dt.
	\end{align*}
	Now, we split the sum over $m$ into three parts. For this, define
	\begin{align*}
		M_1
		&=
		\{m\in\bbz:1_{[0,T]}(t)\gamma(2^{k_1^*\vari}t-m)=\gamma(2^{k_1^*\vari}t-m)\text{ on }[0,T]\},
		\\M_2
		&=
		\{m\in\bbz:0\neq1_{[0,T]}(t)\gamma(2^{k_1^*\vari}t-m)\neq\gamma(2^{k_1^*\vari}t-m)\text{ on }[0,T]\},
	\end{align*}
	as well as $M_3=\complement\left(M_1\cup M_2\right)$. Clearly, we have $S_{k_1,k_2,k_3,k_4}|_{M_3}=0$. Thus, it suffices to consider $M_1$ and $M_2$.
	
	Let us start with the sum over $M_1$. In this case, we can drop the temporal indicator function due to \eqref{p_est_i4n_gam}. Moreover, since the support of $\gamma$ is of unit size, we have $\abs{M_1}\lesssim T2^{k_1^*\vari}$. We conclude
	\begin{align*}
		S_{k_1,k_2,k_3,k_4}|_{M_1}
		&=
		\int_\bbr\sum_{m\in M_1}\mathcal{F}_x\left[
		\prod_{i=1}^4h_{k_i,m}(t)v_{k_i,i}(t,\cdot)\right](0) dt
		\lesssim
		T2^{k_1^*\vari}\sup_{m\in M_1} 
		\mathcal{F}_{t,x}\left[
		\prod_{i=1}^4h_{k_i,m}v_{k_i,i}\right](0,0)
		\\&=
		T2^{k_1^*\vari}\sup_{m\in M_1} 
		\bigast_{i=1}^4\mathcal{F}_{t,x}\left[
		h_{k_i,m}v_{k_i,i}\right](0,0).
	\end{align*}
	For $i\in[4]$, define
	\begin{align*}
		f_{l_i,k_i}(\tau,\xi)
		:=
		\begin{cases}
			\eta_{l_i}(\tau-\omega(\xi))\fttx{h_{k_i,m}v_{k_i,i}}(\tau,\xi)
			&\text{if }l_i>\lfloor k_1^*\vari\rfloor,
			\\
			\eta_{\leq\lfloor k_1^*\vari\rfloor}(\tau-\omega(\xi))\fttx{h_{k_i,m}v_{k_i,i}}(\tau,\xi)
			&\text{if }l_i=\lfloor k_1^*\vari\rfloor.
		\end{cases}
	\end{align*}
	Now, we can apply \eqref{case_quad_generic} and \eqref{eq_est_mod_l2xk} to the sums over $l_i$ and arrive at
	\begin{align*}
		S_{k_1,k_2,k_3,k_4}|_{M_1}
		&=
		T2^{k_1^*\vari}\sup_{m\in M_1} \sum_{l_j\geq\lfloor k_1^*\vari\rfloor,\,j\in[4]}
		\bigast_{i=1}^4f_{l_i,k_i}(0,0)
		\\&\lesssim
		T2^{k_1^*\vari}2^{k_3^*/2}2^{k_4^*/2}\sup_{m\in M_1} \sum_{l_j\geq\lfloor k_1^*\vari\rfloor,\,j\in[4]}
		2^{l_3^*/2}2^{l_4^*/2}\prod_{i=1}^4\norm{f_{l_i,k_i}}{\lebL^2}
		\\&\lesssim
		T2^{k_3^*/2}2^{k_4^*/2}\sup_{m\in M_1}
		\prod_{i=1}^4\norm{\mathcal{F}_{t,x}\left[h_{k_i,m}v_{k_i,i}\right]}{\spx^{k_i}}.
	\end{align*}
	Lastly, we use \eqref{eq_est_xk_loc} and take the supremum in $t_{k_i}$ for each factor separately leading to a loss of the dependency on $m$. Then, by \eqref{p_est_i4n_ext}, we arrive at the desired bound:
	\begin{align*}
		S_{k_1,k_2,k_3,k_4}|_{M_1}
		&\lesssim
		T2^{k_3^*/2}2^{k_4^*/2}\sup_{m\in M_1}
		\prod_{i=1}^4\norm{\mathcal{F}_{t,x}\left[\eta_0(2^{k_1\vari}t-m)v_{k_i,i}(t,x)\right]}{\spx^{k_i}}
		\\&\lesssim
		T2^{k_3^*/2}2^{k_4^*/2}
		\prod_{i=1}^4\norm{v_{k_i,i}}{\spfa^{k_i}}
		\lesssim
		T2^{k_3^*/2}2^{k_4^*/2}
		\prod_{i=1}^4\norm{P_{k_i}u_{i}}{\spfb^{k_i}_T}.
	\end{align*}
	It remains to consider the summation over $M_2$ in $S_{k_1,k_2,k_3,k_4}$. Note that $m\in M_2$ implies that either $0$ or $T$ is in the support of $[0,T]\ni t\mapsto\gamma(2^{k_1^*\vari}t-m)$. Since the support of $\gamma$ is strictly contained in $[-1,1]$, $M_2$ has at most $4$ elements. We conclude
	\begin{align*}
		S_{k_1,k_2,k_3,k_4}|_{M_2}
		&\lesssim
		\sup_{m\in M_2}\int_\bbr1_{[0,T]}(t)\mathcal{F}_x\left[
		\prod_{i=1}^4h_{k_i,m}(t)v_{k_i,i}(t,x)\right](0) dt
		\\&=
		\sup_{m\in M_2}\mathcal{F}_{t,x}\left[1_{[0,T]}(t)
		\prod_{i=1}^4h_{k_i,m}(t)v_{k_i,i}(t,x)\right](0,0).
	\end{align*}
	Now, we localize in the modulation variable and attach the function $1_{[0,T]}$ to the factor with the highest modulation variable, which we assume to be $l_1$. More precisely, we set
	\begin{align*}
		f_{l_j,k_j}(\tau,\xi)
		:=
		\begin{cases}
			\eta_{l_1}(\tau-\omega(\xi))\fttx{1_{[0,T]}h_{k_1,m}v_{k_1,1}}(\tau,\xi)
			&\text{if }l_j>\lfloor k_1^*\vari\rfloor, j=1,
			\\
			\eta_{l_j}(\tau-\omega(\xi))\fttx{h_{k_j,m}v_{k_j,j}}(\tau,\xi)
			&\text{if }l_j>\lfloor k_1^*\vari\rfloor, j\neq 1,
			\\
			\eta_{\lfloor\leq k_1^*\vari\rfloor}(\tau-\omega(\xi))\fttx{1_{[0,T]}h_{k_j,m}v_{k_j,j}}(\tau,\xi)
			&\text{if }l_j=\lfloor k_1^*\vari\rfloor, j=1,
			\\
			\eta_{\lfloor\leq k_1^*\vari\rfloor}(\tau-\omega(\xi))\fttx{h_{k_j,m}v_{k_j,j}}(\tau,\xi)
			&\text{if }l_j=\lfloor k_1^*\vari\rfloor, j\neq1.
		\end{cases}
	\end{align*}
	Then, we proceed as before and obtain
	\begin{align*}
		S_{k_1,k_2,k_3,k_4}|_{M_2}
		&\lesssim
		\sup_{m\in M_2}\sum_{l_j\geq\lfloor k_1^*\vari\rfloor,\,j\in[4]}\bigast_{i=1}^4f_{l_i,k_i}(0,0)
		\\&\lesssim
		2^{k_3^*/2}2^{k_4^*/2}\sup_{m\in M_2}\sum_{l_j\geq\lfloor k_1^*\vari\rfloor,\,j\in[4]}2^{l_3^*/2}2^{l_4^*/2}\prod_{i=1}^4\norm{f_{l_i,k_i}}{\lebL^2}
		\\&\lesssim
		2^{-k_1^*\vari/2}2^{k_3^*/2}2^{k_4^*/2}\sup_{m\in M_2}\prod_{i=2}^4\norm{\mathcal{F}_{t,x}\left[h_{k_i,m}v_{k_i,i}\right]}{\spx^{k_i}}\sum_{l_1\geq\lfloor k_1^*\vari\rfloor}\norm{f_{l_1,k_1}}{\lebL^2}.
	\end{align*}
	To treat the norm involving the sharp time-cutoff $1_{[0,T]}$, we fix $d\in(0,1/2)$. Using Hölder's inequality as well as \eqref{eq_est_l2txk}, we obtain
	\begin{align*}
		\sum_{l_1\geq\lfloor k_1^*\vari\rfloor}\norm{f_{l_1,k_1}}{\lebL^2}
		&\lesssim
		2^{k_1^*\vari(d-1/2)}
		\sup_{l_1\geq\lfloor k_1^*\vari\rfloor}2^{l_1(1/2-d)}\norm{\eta_{l_1}(\tau-\omega(\xi))\fttx{1_{[0,T]}h_{k_1,m}v_{k_1,1}}}{\lebL^2}
		\\&\lesssim
		2^{k_1^*\vari(d-1/2)}
		T^d\norm{\mathcal{F}_{t,x}\left[h_{k_1,m}v_{k_1,1}\right]}{\spx^{k_1}}.
	\end{align*}
	As before, we apply \eqref{eq_est_xk_loc} to each factor, take the supremum in each $t_{k_i}$, $i\in[4]$, and use \eqref{p_est_i4n_ext} leading to
	\begin{align*}
		S_{k_1,k_2,k_3,k_4}|_{M_2}
		&\lesssim
		T^d2^{k_1^*(d-\vari)}2^{k_3^*/2}2^{k_4^*/2}\sup_{m\in M_2}\prod_{i=1}^4\norm{\mathcal{F}_{t,x}\left[h_{k_i,m}v_{k_i,i}\right]}{\spx^{k_i}}
		\\&\lesssim
		T^d2^{k_1^*(d-\vari)}2^{k_3^*/2}2^{k_4^*/2}\prod_{i=1}^4\norm{u_i}{\spfb^{k_i}_T}.
	\end{align*}
	Since $T\in(0,1]$ and $d<1<\vari$ hold, the claim follows.
\end{proof}
\vspace{1cm}
\section{Energy estimates}\label{s_energy}

In this section we prove the following three energy estimates:

\begin{lem}\label{lem_est_spe}
	Let $T\in(0,1]$, $\alpha\in(0,1)$ and $n\in\bbn$. Let $r\geq s>3/2-\alpha$. There exist $c,d>0$ such that:
	\begin{itemize}
		\item[(1)]	For all smooth solutions $u$ of \eqref{eq_PDE} with mean zero we have
		\begin{equation}\label{est_ert}
			\norm{u}{\spec^r_T}^2
			\lesssim
			\norm{u_0}{\lebH^r}^2
			+
			T2^{2n}\norm{u}{\spfc^r_T}^2\norm{u}{\spfc^s_T}
			+
			2^{-cn}\norm{u}{\spfc^r_T}^2\norm{u}{\spfc^s_T}
			+
			T^{d}\norm{u}{\spfc^r_T}^2\norm{u}{\spfc^s_T}^2.
		\end{equation}
		\item[(2)]	For all smooth solutions $u_1$ and $u_2$ of \eqref{eq_PDE} with mean zero, $v=u_1-u_2$, $w=u_1+u_2$, we have
		\begin{align}\label{est_ezt}
			\begin{split}
				\norm{v}{\spec^{-1/2}_T}^2
				&\lesssim
				\norm{v_0}{\lebH^{-1/2}}^2
				+
				(T2^{2n}+2^{-cn})\norm{v}{\spfc^{-1/2}_T}^2\norm{w}{\spfc^s_T}
				\\&+
				T^d\norm{v}{\spfc^{-1/2}_T}^2
				\left(
				\norm{w}{\spfc^s_T}^2+\norm{u_1}{\spfc^s_T}^2+\norm{u_2}{\spfc^s_T}^2\right).
			\end{split}
		\end{align}
		\item[(3)]	For all smooth solutions $u_1$ and $u_2$ of \eqref{eq_PDE} with mean zero, $v=u_1-u_2$, $w=u_1+u_2$, we have
		\begin{align}\label{est_est}
			\begin{split}
				\norm{v}{\spec^s_T}^2
				&\lesssim
				\norm{v_0}{\lebH^s}^2
				+
				2^{-cn}\norm{v}{\spfc^{-1/2}_T}\norm{v}{\spfc^s_T}\norm{u_2}{\spfc^{s+2-\alpha}_T}
				\\&
				+
				(T2^{2n}+2^{-cn})\left(\norm{v}{\spfc^s_T}^3
				+
				\norm{v}{\spfc^s_T}^2\norm{u_2}{\spfc^s_T}\right)
				\\&
				+
				T^d\big(\norm{v}{\spfc^s_T}^4
				+
				\norm{v}{\spfc^s_T}^3\norm{u_2}{\spfc^s_T}
				+
				\norm{v}{\spfc^s_T}^2\norm{u_2}{\spfc^s_T}\norm{w}{\spfc^s_T}
				+
				\norm{v}{\spfc^s_T}^2\norm{u_2}{\spfc^s_T}^2
				\\&
				+
				\norm{v}{\spfc^{-1/2}_T}\norm{v}{\spfc^s_T}\norm{u_2}{\spfc^s_T}\norm{u_2}{\spfc^{s+2-\alpha}_T}
				+
				\norm{v}{\spfc^{-1/2}_T}\norm{v}{\spfc^s_T}^2\norm{u_2}{\spfc^{s+2-\alpha}_T}\big).
			\end{split}
		\end{align}
	\end{itemize}
\end{lem}
All three estimates will be proved with the help of the quadrilinear estimate from Lemma \ref{lem_i4}, commutator estimates and cancellations due to symmetry. Additionally, we need to control the resonance function, which will appear in the calculations.

Let us recall that the dispersion relation $\omega$ is given by $\omega(\xi)=-\xi\abs{\xi}^\alpha$ and that it appears naturally due to $\mathcal{F}(\dx D_x^\alpha u)(\xi)=i\omega(\xi) \mathcal{F}_x(u)(\xi)$. For $\xi_1$, $\xi_2$ and $\xi_3$ satisfying $\xi_{123}=0$, we define the resonance function by
\begin{equation}\label{eq_res_fct}
	\Omega(\xi_1,\xi_2,\xi_3)
	:=
	\omega(\xi_1)
	+
	\omega(\xi_2)
	+
	\omega(\xi_3).
\end{equation}
A simple calculation leads to the asymptotic behaviour
\begin{equation}\label{eq_est_res}
	\abs{\Omega(\xi_1,\xi_2,\xi_3)}
	\sim
	\abs{\xi_1^*}^\alpha\abs{\xi_3^*}.
\end{equation}
 In the rest of this section, we write $k\lesssim l$, $k\ll l$ and $k\sim l$ instead of $2^k\lesssim 2^l$, $2^k\ll 2^l$ and $2^k\sim 2^l$. Besides the aforementioned behaviour of the resonance function, we also make use of the following estimate: \begin{lem}\label{cor_est_res_inv}
	Let $k_2\sim k_a\gg k_3$ and $k_2\gtrsim k_b$. Moreover, let $\xi_i\in\bbz$, $i\in\{a,b,2,3\}$, satisfy $\abs{\xi_i}\sim2^{k_i}$ and $\xi_{ab23}=0$. Then, we have
	\begin{equation}\label{eq_est_res_inv}
		\abs{\frac{1}{\Omega}(\xi_a,\xi_{2b},\xi_3)-\frac{1}{\Omega}(\xi_{ab},\xi_2,\xi_3)}
		\lesssim
		2^{-k_1(1+\alpha)}2^{-k_3}2^{k_b}.
	\end{equation}
\end{lem}
\begin{proof}
	Using the definition \eqref{eq_res_fct} and the fact that $\omega$ is an odd function, we obtain
	\begin{align*}
		\frac{1}{\Omega}(\xi_a,\xi_{2b},\xi_3)-\frac{1}{\Omega}(\xi_{ab},\xi_2,\xi_3)
		=
		\frac{\omega(\xi_{a})-\omega(\xi_{a3})-\omega(\xi_{ab})+\omega(\xi_{ab3})}{\Omega(\xi_2,\xi_{ab},\xi_3)\Omega(\xi_{ab},\xi_2,\xi_3)}.
	\end{align*}
	By applying \eqref{eq_est_res}, we can bound the denominator above by
	\begin{align*}
		\abs{\frac{1}{\Omega(\xi_2,\xi_{ab},\xi_3)\Omega(\xi_{ab},\xi_2,\xi_3)}}
		\lesssim
		2^{-2\alpha k_2}2^{-2k_3}.
	\end{align*}
	To bound the numerator, we consider two cases: First, we assume $k_2\gg k_b$. Then, the integers $\xi_a$, $\xi_{a3}$, $\xi_{ab}$ and $\xi_{ab3}$ all have the same sign and have modulus of size $2^{k_a}$. This allows to apply the double mean value theorem leading to the bound $2^{k_a(\alpha-1)}2^{k_3}2^{k_b}$. 
	
	Next, we consider $k_2\sim k_b$. In this case, the integers $\xi_a$, $\xi_{ab}$, $\xi_{a3}$ and $\xi_{ab3}$ have modulus of size $2^{k_a}$ due to $\xi_{ab}=-\xi_{23}$ and $\xi_{ab3}=-\xi_2$. Moreover, $\xi_a$ and $\xi_{a3}$ (resp. $\xi_{ab}$ and $\xi_{ab3}$) have the same sign. Thus, the fundamental theorem of calculus yields the bounds $\abs{\omega(\xi_{a})-\omega(\xi_{a3})}\lesssim 2^{k_3}2^{k_1\alpha}$ and $\abs{\omega(\xi_{ab})-\omega(\xi_{ab3})}\lesssim 2^{k_3}2^{k_1\alpha}$. After multiplication with $2^{k_b-k_1}\gtrsim 1$, we obtain the same bound as in the first case.
\end{proof}

Now, we can start proving Lemma \ref{lem_est_spe}. Throughout this section, we assume that the premisses of Lemma \ref{lem_est_spe} hold true. Also, without mentioning, we freely localize functions using the projector $P_k$ and denote $P_k u$ by $u_k$.
\subsection{Proof of the first energy estimate}\label{ss_for_solution}

As the title suggests, this section is devoted to the proof of \eqref{est_ert}.
We apply $P_{k_1}$ to equation \eqref{eq_PDE}, multiply it by $P_{k_1}u$ and integrate over $\bbt$. Since $\chi_{k_1}$ is even and $u$ is real, so is $P_{k_1}u$. Moreover, recall that $\dx D_x^\alpha$ has the purely imaginary symbol $i\omega$. We conclude
\begin{align*}
	\dt\norm{P_{k_1}u}{\lebL^2}^2
	=
	2\int_\bbt P_{k_1}u P_{k_1}\dt u dx
	=
	2\int_\bbt P_{k_1}u P_{k_1}\dx (u^2) dx
	=
	4\int_\bbt P_{k_1}^2 u u \dx u dx.
\end{align*}
Now, we integrate over $[0,t_{k_1}]$, use the fundamental theorem of calculus, take the supremum over $t_{k_1}$ in $[0,T]$, multiply by $2^{k_12r}$ and take the sum over $k_1$ in $\bbn$ leading to
\begin{align}\label{eq_dif_ert}
	\norm{u}{\spec_T^r}^2 -\norm{u_0}{H^r}^2
	&=
	4\sum_{k_1}2^{k_12r}\sup_{t_{k_1}}\int_0^{t_{k_1}}\sum_{\xi_{123}=0}i\xi_1\chi_{k_1}^2(\xi_1)\ftxh{u}(\xi_1)\ftxh{u}(\xi_2)\ftxh{u}(\xi_3)\nonumber
	\\&\lesssim
	\sum_{k_1\sim k_2\gg k_3}2^{k_12r}\sup_{t}\int_0^{t}\sum_{\xi_{123}=0}\underset{=:\sigma(\xi_1,\xi_2,\xi_3)}{\underbrace{i\xi_1\chi_{k_1}^2(\xi_1)\chi_{k_2}(\xi_2)\chi_{k_3}(\xi_3)}}\ftxh{u}(\xi_1)\ftxh{u}(\xi_2)\ftxh{u}(\xi_3).
\end{align}
Above, the last sum is restricted to the regime $k_1\sim k_2\gg k_3$. Clearly, this estimate is not correct at first. However, we will justify \eqref{eq_dif_ert} at the end of this section showing that it is sufficient to treat only this case.

We want to shift the derivative to the low frequency. For $\xi_{123}=0$ we write
\begin{align*}
	\sigma(\xi_1,\xi_2,\xi_3)
	+
	\sigma(\xi_2,\xi_1,\xi_3)
	=&
	-i\xi_3\chi^2_{k_1}(\xi_1)\chi_{k_2}(\xi_2)\chi_{k_3}(\xi_3)
	&=:\sigma_1(\xi_1,\xi_2,\xi_3)
	\\&
	-\xi_2\left[\chi^2_{k_1}(\xi_1)-\chi^2_{k_1}(\xi_2)\right]\chi_{k_2}(\xi_2)\chi_{k_3}(\xi_3)
	&=:\sigma_2(\xi_1,\xi_2,\xi_3)
	\\&
	-i\xi_2\chi^2_{k_1}(\xi_2)\left[\chi_{k_2}(\xi_2)-\chi_{k_2}(\xi_1)\right]\chi_{k_3}(\xi_3)
	&=:\sigma_3(\xi_1,\xi_2,\xi_3)
\end{align*}
and observe the trivial bound $\abs{\sigma_j(\xi_1,\xi_2,\xi_3)}\lesssim 2^{k_3}$ for each $j\in[3]$.
The above  implies
\begin{align}\label{eq_sig}
	2\sum_{\xi_{123}=0}\sigma(\xi_1,\xi_2,\xi_3)\hat{u}(\xi_1)\hat{u}(\xi_2)\hat{u}(\xi_3)
	=
	\sum_{j\in[3]}\sum_{\xi_{123}=0}\sigma_j(\xi_1,\xi_2,\xi_3)\hat{u}(\xi_1)\hat{u}(\xi_2)\hat{u}(\xi_3).
\end{align}
Hence, for a parameter $n\in\bbn$, we obtain
\begin{align}\label{eq_split_i}
	\begin{split}
		\eqref{eq_dif_ert}
		&=
		\sum_{k_1\gg 	k_3}\underset{=:I(k_1,k_3)}{\underbrace{2^{k_12r}\sup_{t}\sum_{\substack{k_2\\j\in[3]}}\int_0^{t}\sum_{\xi_{123}=0}\sigma_j(\xi_1,\xi_2,\xi_3)\ftxh{u}(\xi_1)\ftxh{u}(\xi_2)\ftxh{u}(\xi_3)}}
		\\
		&=
		\sum_{n\geq k_1\gg k_3} I(k_1,k_3)
		+
	\sum_{n<k_1\gg k_3} I(k_1,k_3).
	\end{split}
\end{align}
Let us continue with bounding the right-hand side of \eqref{eq_split_i} in the next two lemmata.
First, we estimate the low-frequency contribution.
\begin{lem}\label{lem_est_ert_ll}
	Let $r\geq s> 1/2$. Then, we have
	\begin{equation*}
		\sum_{n\geq k_1\gg k_3} \abs{I(k_1,k_3)}
		\lesssim
		T2^{2n}\norm{u}{\spfc^r_T}^2\norm{u}{\spfc^s_T}.
	\end{equation*}
\end{lem}
\begin{proof}
	Using the bound $\abs{\sigma_j}\lesssim 2^{k_3}$ as well as Hölder's and Jensen's inequalities, we obtain
	\begin{align*}
		\abs{I(k_1,k_3)}
		&\leq
		T\sum_{k_2}2^{k_12r}\sup_t\sum_{\xi_{123}=0}\abs{\sigma_j(\xi_1,\xi_2,\xi_3)}\abs{\ftxh{u}(\xi_1)}\abs{\ftxh{u}(\xi_2)}\abs{\ftxh{u}(\xi_3)}
		\\&\lesssim
		T\sum_{k_2}2^{k_12r}2^{k_3}\sup_t\sum_{\substack{\abs{\xi_i}\sim 2^{k_i},i\in[3]\\ \xi_{123}=0}}\abs{\ftxh{u}(\xi_1)}\abs{\ftxh{u}(\xi_2)}\abs{\ftxh{u}(\xi_3)}
		\\&\lesssim
		T\sum_{k_2}2^{k_12r}2^{k_33/2}\norm{u_{k_1}}{\lebL^\infty_T\lebL^2}\norm{u_{k_2}}{\lebL^\infty_T\lebL^2}\norm{u_{k_3}}{\lebL^\infty_T\lebL^2}.
	\end{align*}
	The claim follows by an application of the estimate \eqref{eq_est_emb_spfc} and summation in $k_1$ and $k_3$.
\end{proof}

Now, we have to estimate the second term on the right-hand side of \eqref{eq_split_i}. Recalling the definition \eqref{eq_res_fct} and the fact that $u$ is a smooth solution of \eqref{eq_PDE}, we have
\begin{align}\label{eq_ibp_uuu}
	\begin{split}
		&-\Omega(\xi_1,\xi_2,\xi_3)\hat{u}(\xi_1)\hat{u}(\xi_2)\hat{u}(\xi_3)
		\\=&\,
		\dt \left[\hat{u}(\xi_1)\hat{u}(\xi_2)\hat{u}(\xi_3)\right]
		+
		i\xi_1\widehat{uu}(\xi_1)\hat{u}(\xi_2)\hat{u}(\xi_3)
		+
		i\xi_2\hat{u}(\xi_1)\widehat{uu}(\xi_2)\hat{u}(\xi_3)
		+
		i\xi_3\hat{u}(\xi_1)\hat{u}(\xi_2)\widehat{uu}(\xi_3).
	\end{split}
\end{align}
In view of \eqref{eq_est_res}, this equation will turn out to be very helpful in the next lemma.

\begin{lem}\label{lem_est_ert_hl}
	Let $r\geq s>3/2-\alpha$. Then, there exist $c,d>0$ such that we have
	\begin{align*}
		\sum_{n<k_1\gg k_3}\abs{I(k_1,k_3)}
		\lesssim
		2^{-nc}\norm{u}{\spfc^r_T}^2\norm{u}{\spfc^s_T}
		+
		T^d\norm{u}{\spfc^r_T}^2\norm{u}{\spfc^s_T}^2.
	\end{align*}
\end{lem}
\begin{proof}
	We apply \eqref{eq_ibp_uuu} to each summand $I(k_1,k_3)$ and obtain
	\begin{align*}
		\abs{I(k_1,k_3)}
		\lesssim
		\abs{B(k_1,k_3)}
		+
		\abs{I_1 (k_1,k_2)}
		+
		\abs{I_2 (k_1,k_2)}
		+
		\abs{I_{12} (k_1,k_2)}
		+
		\abs{I_3 (k_1,k_3)},
	\end{align*}
	where the terms above are given by
	\begin{align*}
		B(k_1,k_3)
		&:=
		2^{k_12r}\sup_{t}\sum_{\substack{k_2\\ j\in[3]}}\left[\sum_{\xi_{123}=0}\sigma_j(\xi_1,\xi_2,\xi_3)\ftxh{u}(\xi_1)\ftxh{u}(\xi_2)\ftxh{u}(\xi_3)\right]_0^t,
		\\
		I_1(k_1,k_3)
		&:=
		2^{k_12r}\sup_{t}\sum_{\substack{k_2\\j\in[3]}}\int_0^{t}\sum_{\substack{\xi_{ab23}=0\\ \xi_a\sim \xi_b}} \frac{\sigma_j}{\Omega}(\xi_{ab},\xi_2,\xi_3)\xi_{ab}\ftxh{u}(\xi_a)\ftxh{u}(\xi_b)\ftxh{u}(\xi_2)\ftxh{u}(\xi_3),
		\\
		I_2(k_1,k_3)
		&:=
		2^{k_12r}\sup_{t}\sum_{\substack{k_2\\j\in[3]}}\int_0^{t}\sum_{\substack{\xi_{ab23}=0\\ \xi_a\sim \xi_b}} \frac{\sigma_j}{\Omega}(\xi_1,\xi_{ab},\xi_3)\xi_{ab}\ftxh{u}(\xi_1)\ftxh{u}(\xi_a)\ftxh{u}(\xi_b)\ftxh{u}(\xi_3),
		\\
		I_{12}(k_1,k_3)
		&:=
		2^{k_12r}\sup_{t}\sum_{\substack{k_2\\j\in[3]}}\int_0^{t}\sum_{\substack{\xi_{ab23}=0\\ \xi_a\not\sim\xi_b}} \left[\frac{\sigma_j}{\Omega}(\xi_{ab},\xi_2,\xi_3)\xi_{ab}+\frac{\sigma_j}{\Omega}(\xi_2,\xi_{ab},\xi_3)\xi_{ab}\right]\ftxh{u}(\xi_a)\ftxh{u}(\xi_b)\ftxh{u}(\xi_2)\ftxh{u}(\xi_3),
		\\
		I_3(k_1,k_3)
		&:=
		2^{k_12r}\sup_{t}\sum_{\substack{k_2\\j\in[3]}}\int_0^{t}\sum_{\xi_{12ab}=0}\frac{\sigma_j}{\Omega}(\xi_1,\xi_2,\xi_{ab})\xi_{ab}\ftxh{u}(\xi_1)\ftxh{u}(\xi_2)\ftxh{u}(\xi_a)\ftxh{u}(\xi_b).
	\end{align*}
	Here, $B$ stands for boundary term and the subscript $i$ of $I$ corresponds to the variable $\xi_i$ being split into the sum $\xi_a+\xi_b$. Without specifically mentioning it, we can assume $\xi_i\neq 0$ for all $i\in[3]$. Indeed, any summand in $I(k_1,k_3)$ with $\xi_i=0$ for some $i\in[3]$ vanishes. This follows immediately from the fact that all $\spfc^s_T$-functions have mean zero. In particular, \eqref{eq_est_res} guarantees that the resonance function $\Omega$ does not vanish in any of the expressions above.
	
	We begin by estimating $B$. Here, we necessarily have $k_2\sim k_1$ and, similarly to the proof of Lemma \ref{lem_est_ert_ll}, conclude
	\begin{align*}
		\abs{B(k_1,k_3)}
		&\lesssim
		\sum_{k_2}2^{k_1(2r-\alpha)}2^{k_3/2}
		\norm{u_{k_1}}{\lebL^2_T\lebL^\infty}\norm{u_{k_2}}{\lebL^2_T\lebL^\infty}\norm{u_{k_3}}{\lebL^2_T\lebL^\infty}.
	\end{align*}
	In order to bound the terms $I_1$, $I_2$, $I_{12}$ and $I_3$, we localize the variables $\xi_a$ and $\xi_b$ to dyadic frequency ranges, i.e. we insert the factor $1=\sum_{k_a,k_b}\chi_{k_a}(\xi_a)\chi_{k_b}(\xi_b)$.
	Using the bound $\abs{\frac{\sigma_j}{\Omega}(\xi_{ab},\xi_2,\xi_3)\xi_{ab}}\lesssim 2^{k_1(1-\alpha)}$ as well as Lemma \ref{lem_i4}, $I_1$ can be estimated by
	\begin{align*}
		\abs{I_1(k_1,k_3)}
		\lesssim
		T^d\sum_{k_a\sim k_b\gtrsim k_2}2^{k_2(2r+3/2-\alpha)}2^{k_3/2}
		\prod_{j\in\{a,b,2,3\}}\norm{u_{k_j}}{\spfb^{k_j}_T}.
	\end{align*}
	Similarly, $I_2$ can be handled with the help of $\abs{\frac{\sigma_j}{\Omega}(\xi_1,\xi_{ab},\xi_3)\xi_{ab}}\lesssim 2^{k_1(1-\alpha)}$ and Lemma \ref{lem_i4}. We obtain
	\begin{align*}
		\abs{I_2(k_1,k_3)}
		\lesssim
		T^d\sum_{k_a\sim k_b\gtrsim k_2}2^{k_2(2r+3/2-\alpha)}2^{k_3/2}
		\prod_{j\in\{1,a,b,3\}}\norm{u_{k_j}}{\spfb^{k_j}_T}.
	\end{align*}
	Next, we bound $I_{12}$. Firstly, consider the case $k_1\sim k_a\gg k_b$. A direct estimate of the modulus of the symbol $\frac{\sigma_j}{\Omega}$ yields the bound $2^{k_1(1-\alpha)}$, which is insufficient in this case. Instead, we benefit from cancellations in the symbol since the rest of the integrand -- that is the factor $\hat{u}(\xi_1)\hat{u}(\xi_2)\hat{u}(\xi_3)\hat{u}(\xi_4)$ -- is symmetric in the high-frequency variables $\xi_a$ and $\xi_2$. Indeed, this symmetry yields
	\begin{align*}
		&\sum_{\xi_{ab23}=0} \left[\frac{\sigma_j}{\Omega}(\xi_{ab},\xi_2,\xi_3)\xi_{ab}+\frac{\sigma_j}{\Omega}(\xi_2,\xi_{ab},\xi_3)\xi_{ab}\right]\ftxh{u}(\xi_2)\ftxh{u}(\xi_3)\ftxh{u}(\xi_a)\ftxh{u}(\xi_b)
		\\=&
		\sum_{\xi_{ab23}=0} \left[\frac{\sigma_j}{\Omega}(\xi_{ab},\xi_2,\xi_3)\xi_{ab}+\frac{\sigma_j}{\Omega}(\xi_a,\xi_{2b},\xi_3)\xi_{2b}\right]\ftxh{u}(\xi_2)\ftxh{u}(\xi_3)\ftxh{u}(\xi_a)\ftxh{u}(\xi_b).
	\end{align*}
	Thus, we can rewrite the new symbol as follows:
	\begin{align*}
		\left[\frac{\sigma_j}{\Omega}(\xi_{ab},\xi_2,\xi_3)\xi_{ab}+\frac{\sigma_j}{\Omega}(\xi_a,\xi_{2b},\xi_3)\xi_{2b}\right]
		&=
		\left[\xi_{ab}+\xi_{2b}\right]\frac{\sigma_j}{\Omega}(\xi_{ab},\xi_2,\xi_3)
		&=:s_1
		\\&+
		\left[\frac{1}{\Omega}(\xi_a,\xi_{2b},\xi_3)-\frac{1}{\Omega}(\xi_{ab},\xi_2,\xi_3)\right]\sigma_j(\xi_{ab},\xi_2,\xi_3)\xi_{2b}
		&=:s_2
		\\&+
		\left[\sigma_j(\xi_a,\xi_{2b},\xi_3)-\sigma_j(\xi_{ab},\xi_2,\xi_3)\right]\frac{1}{\Omega}(\xi_a,\xi_{2b},\xi_3)\xi_{2b}&=:s_3.
	\end{align*}
	Using $\abs{\sigma_j}\lesssim2^{k_3}$,  \eqref{eq_est_res} as well as the equation $\xi_{ab}+\xi_{2b}=\xi_{b}-\xi_3$, we obtain
	\begin{align*}
		\abs{s_1}
		\lesssim
		2^{-k_1\alpha}2^{\max\{k_3,k_b\}}.
	\end{align*}
	The estimate in \eqref{eq_est_res_inv} combined with $\abs{\sigma_j}\lesssim2^{k_3}$ once again yields
	\begin{align*}
		\abs{s_2}
		\lesssim
		2^{-k_1\alpha}2^{\max\{k_3,k_b\}}.
	\end{align*}
	To bound $s_3$, we observe that
	\begin{align*}
		\left[\sigma_1(\xi_a,\xi_{2b},\xi_3)-\sigma_1(\xi_{ab},\xi_2,\xi_3)\right]
		&=
		\xi_3\left[\chi^2_{k_1}(\xi_a)-\chi^2_{k_1}(\xi_{ab})\right]\chi_{k_2}(\xi_{ab})
		\\&+
		\xi_3\chi^2_{k_1}(\xi_{ab})\left[\chi_{k_2}(\xi_{ab})-\chi_{k_2}(\xi_{a})\right],
		\\
		\left[\sigma_2(\xi_a,\xi_{2b},\xi_3)-\sigma_2(\xi_{ab},\xi_2,\xi_3)\right]
		&=
		\xi_{b}\left[\chi^2_{k_1}(\xi_a)-\chi^2_{k_1}(\xi_{2b})\right]\chi_{k_2}(\xi_{2b})\chi_{k_3}(\xi_3)
		\\&+
		\xi_{2}\left[\chi^2_{k_1}(\xi_a)-\chi^2_{k_1}(\xi_{2b})\right]\left[\chi_{k_2}(\xi_{2b})-\chi_{k_2}(\xi_2)\right]\chi_{k_3}(\xi_3)
		\\&+
		\xi_{2}\left[\chi^2_{k_1}(\xi_a)-\chi^2_{k_1}(\xi_{2b})-\chi^2_{k_1}(\xi_{ab})+\chi^2_{k_1}(\xi_2)\right]\chi_{k_2}(\xi_2)\chi_{k_3}(\xi_3),
		\\
		\left[\sigma_3(\xi_a,\xi_{2b},\xi_3)-\sigma_3(\xi_{ab},\xi_2,\xi_3)\right]
		&=
		\xi_{b}\chi^2_{k_1}(\xi_{2b})\left[\chi_{k_2}(\xi_{2b})-\chi_{k_2}(\xi_a)\right]\chi_{k_3}(\xi_3)
		\\&+
		\xi_{2}\left[\chi^2_{k_1}(\xi_{2b})-\chi^2_{k_1}(\xi_2)\right]\left[\chi_{k_2}(\xi_{2b})-\chi_{k_2}(\xi_a)\right]\chi_{k_3}(\xi_3)
		\\&+
		\xi_{2}\chi^2_{k_1}(\xi_2)\chi_{k_3}(\xi_3)\left[\chi_{k_2}(\xi_{2b})-\chi_{k_2}(\xi_a)-\chi_{k_2}(\xi_2)+\chi_{k_2}(\xi_{ab})\right]
	\end{align*}
	hold. Above, each summand on the right-hand side has modulus bounded by $2^{-k_1}2^{k_3}2^{k_b}$. Using \eqref{eq_est_res}, we obtain
	\begin{align*}
		\abs{s_3}
		\lesssim
		2^{-k_1\alpha}2^{\max\{k_3,k_b\}}.
	\end{align*}
	Consequently, the symbol $s_1+s_2+s_3$ has modulus bounded by $2^{-\alpha k_1}2^{\max\{k_3,k_b\}}$ and after localizing $\xi_a$ and $\xi_b$ to dyadic frequency ranges, an application of Lemma \ref{lem_i4} yields
	\begin{align*}
		\abs{I_{12}(k_1,k_3)}
		\lesssim
		T^d\sum_{k_2, k_a, k_b}
		\,
		\prod_{j\in\{a,b,2,3\}}\norm{u_{k_j}}{\spfb^{k_j}_T}
		\times
		\begin{cases}
			2^{k_2(2r-\alpha)}2^{\max\{k_3,k_b\}}2^{k_3/2}2^{k_b/2}
			&\text{if }
			k_1\sim k_a\gg k_b,
			\\
			2^{k_2(2r-\alpha)}2^{\max\{k_3,k_b\}}2^{k_3/2}2^{k_b/2}
			&\text{if }
			k_1\sim k_b\gg k_a,
		\end{cases}
	\end{align*}
	where the second estimate follows by the same arguments.
	
	Finally, let us bound $I_3$. Making use of Lemma \ref{lem_i4} and the bound $\abs{\frac{\sigma_j}{\Omega}(\xi_1,\xi_2,\xi_{ab})\xi_{ab}}\lesssim2^{-k_1\alpha}2^{k_3}$, it follows
	\begin{align*}
		\abs{I_3(k_1,k_3)}
		\lesssim
		T^d\sum_{k_a, k_b, k_2}
		\,\prod_{j\in\{1,2,a,b\}}\norm{u_{k_j}}{\spfb^{k_j}_T}
		\times
		\begin{cases}
			2^{k_1(2r-\alpha)}2^{k_3^*/2}2^{k_4^*3/2}
			&\text{if }
			k_a \sim k_b \gg k_3,
			\\
			2^{k_1(2r-\alpha)}2^{k_a3/2}2^{k_b/2}
			&\text{if }
			k_3\sim k_a \gg k_b,
			\\
			2^{k_1(2r-\alpha)}2^{k_b3/2}2^{k_a/2}
			&\text{if }
			k_3\sim k_b \gg k_a,
		\end{cases}
	\end{align*}
	where $k_3^*$ (resp. $k_4^*$) denotes the third (resp. fourth) largest number of $k_1$, $k_2$, $k_a$ and $k_b$.
	
	Now, we sum over the bounds of $B$, $I_1$, $I_2$, $I_{12}$ and $I_3$ in $k_1$ and $k_3$ and additionally invoke \eqref{eq_est_emb_spfc} for the bound of $B$. This concludes the proof for every $c\in(0,\alpha)$.
\end{proof}

Using the estimates proved in Lemmata \ref{lem_est_ert_ll} and \ref{lem_est_ert_hl}, we obtain a bound for the modulus of \eqref{eq_dif_ert}, which is given by
\begin{align*}
	\sum_{k_1\sim k_2\gg k_3}2^{k_12r}\sup_{t}\int_0^{t}\sum_{\xi_{123}=0}i\xi_1\chi_{k_1}^2(\xi_1)\chi_{k_2}(\xi_2)\chi_{k_3}(\xi_3)\ftxh{u}(\xi_1)\ftxh{u}(\xi_2)\ftxh{u}(\xi_3).
\end{align*}
To complete the proof of estimate \eqref{est_ert}, it remains to deduce appropriate bounds for the term above for any other possible restrictions on $k_1$, $k_2$ and $k_3$.

The bound for the case $k_1\sim k_3\gg k_2$ can be derived similarly to that for $k_1\sim k_2\gg k_3$. To treat the case $k_2\sim k_3\gg k_1$, we can omit the application of \eqref{eq_sig} at the beginning of our calculations since the derivative is already on the low-frequency term. After that, we proceed as before. It remains to analyze the case $k_1\sim k_2\sim k_3$. Again, we can omit the application of \eqref{eq_sig} and argue as before with the minor difference that Corollary \ref{cor_est_res_inv} cannot be used in this case. Hence, we must prove the bound for $I_{12}$ in Lemma \ref{lem_est_ert_hl} differently. However, we can simply use the direct bound on the symbol, which is of size $2^{k_1(1-\alpha)}$.

Thus, the proof of \eqref{est_ert} is complete.
\subsection{Proof of the second energy estimate}\label{ss_for_differences_i}

In this section we will prove estimate \eqref{est_ezt}. Recall that $u_1$ and $u_2$ are smooth solutions of \eqref{eq_PDE} and that we write $v=u_1-u_2$, $w=u_1+u_2$. Observe that $v$ satisfies
\begin{equation}\label{eq_PDE_difference}
	\dt v
	+
	\dx D_x^\alpha v
	=
	\dx(vw).
\end{equation}
Hence, by similar arguments as used in Section \ref{ss_for_solution}, we obtain
\begin{align}\label{eq_h12_diff}
	\begin{split}
		&\norm{v}{\spec^{-1/2}_T}^2
		-
		\norm{v_0}{\lebH^{-1/2}}^2
		\lesssim
		\sum_{k_1, k_3}\underset{=:I\!I(k_1,k_3)}{\underbrace{2^{-k_1}
		\sup_{t}\int_0^{t}\int_\bbt (P_{k_1}^2\dx v)vP_{k_3}w}}
		\\=
		&\sum_{k_1,k_3\leq n}I\!I(k_1,k_3)
		+
		\sum_{k_1\sim k_3> n}I\!I(k_1,k_3)
		+
		\sum_{k_1\ll k_3> n}I\!I(k_1,k_3)
		+
		\sum_{k_3\ll k_1> n}I\!I(k_1,k_3).
	\end{split}
\end{align}
Moreover, we have the following analogue of \eqref{eq_ibp_uuu}:
\begin{align}\label{eq_ibp_vvw}
	\begin{split}
		&-\Omega(\xi_1,\xi_2,\xi_3)\hat{v}(\xi_1)\hat{v}(\xi_2)\hat{w}(\xi_3)
		\\=&\,
		\dt \left[\hat{v}(\xi_1)\hat{v}(\xi_2)\hat{w}(\xi_3)\right]
		+
		i\xi_1\widehat{vw}(\xi_1)\hat{v}(\xi_2)\hat{w}(\xi_3)
		+
		i\xi_2\hat{v}(\xi_1)\widehat{vw}(\xi_2)\hat{w}(\xi_3)
		+
		\sum_{j\in[2]}i\xi_3\hat{v}(\xi_1)\hat{v}(\xi_2)\widehat{u_ju_j}(\xi_3).
	\end{split}
\end{align}
Let us begin by estimating the low-frequency interaction.
\begin{lem}\label{lem_h12_LL}
	Let $s>3/2-\alpha$. Then, we have
	\begin{align*}
		\sum_{k_1,k_3\leq n}\abs{I\!I(k_1,k_3)}
		\lesssim
		T2^{2n}\norm{v}{\spfc^{-1/2}_T}^2\norm{w}{\spfc^s_T}.
	\end{align*}	
\end{lem}
\begin{proof}
	Writing $v=\sum_{k_2}v_{k_2}$, we notice that $k_2\lesssim n$ holds. Then, the claim follows from the estimate
	\begin{align*}
		\abs{I\!I(k_1,k_3)}
		\lesssim
		\sum_{k_2}2^{-k_1}\sup_{t}\int_0^{t}\int_\bbt \abs{P_{k_1}^2\dx vP_{k_2}vP_{k_3}w}
		\lesssim
		T\sum_{k_2} 2^{k_3^*/2}\norm{v_{k_1}}{\lebL^\infty_T\lebL^2}\norm{v_{k_2}}{\lebL^\infty_T\lebL^2}\norm{w_{k_3}}{\lebL^\infty_T\lebL^2}
	\end{align*}
	and the estimate \eqref{eq_est_emb_spfc}.
\end{proof}
Next, we bound those summands, in which one factor $v$ is localized to a high frequency and the other factor $v$ is localized to a low frequency.
\begin{lem}\label{lem_h12_HH}
	Let $s>3/2-\alpha$. Then, there exist $c,d>0$ such that we have
	\begin{equation*}
		\sum_{k_1\sim k_3>n}\abs{I\!I(k_1,k_3)}
		\lesssim
		2^{-nc}\norm{v}{\spfc^{-1/2}_T}^2\norm{w}{\spfc^s_T}
		+
		T^d\norm{v}{\spfc^{-1/2}_T}^2\left(\norm{u_1}{\spfc^s_T}^2+\norm{u_2}{\spfc^s_T}^2\right).
	\end{equation*}
\end{lem}
\begin{proof}
	We apply \eqref{eq_ibp_vvw} to each summand $I\!I(k_1,k_3)$ and obtain
	\begin{align*}
		\abs{I\!I(k_1,k_3)}
		\leq
		\abs{B\!B(k_1,k_3)}
		+
		\abs{I\!I_1(k_1,k_3)}
		+
		\abs{I\!I_2(k_1,k_3)}
		+
		\abs{I\!I_3(k_1,k_3)},
	\end{align*}
	where the terms on the right-hand side are defined by
	\begin{align*}
		B\!B(k_1,k_3)
		&:=
		2^{-k_1}
		\sup_{t}\sum_{k_2}
		\left[ \sum_{\xi_{123}=0}\frac{i\xi_1\chi_{k_1}^2(\xi_1)\chi_{k_2}(\xi_2)\chi_{k_3}(\xi_3)}{\Omega(\xi_1,\xi_2,\xi_3)}\hat{v}(\xi_1)\hat{v}(\xi_2)\hat{w}(\xi_3)\right]_0^t,
		\\
		I\!I_1(k_1,k_3)
		&:=
		2^{-k_1}\sup_t\sum_{k_2}\int_0^t
		\sum_{\xi_{ab23}=0}
		\frac{(-i\xi_{ab})^2\chi_{k_1}(\xi_{ab})\chi_{k_2}(\xi_2)\chi_{k_3}(\xi_3)}{\Omega(\xi_{ab},\xi_2,\xi_3)}
		\ftxh{v}(\xi_a)\ftxh{w}(\xi_b)\ftxh{v}(\xi_2)\ftxh{w}(\xi_3),
		\\
		I\!I_2(k_1,k_3)
		&:=
		2^{-k_1}\sup_t\sum_{k_2}\int_0^t
		\sum_{\xi_{1ab3}=0}
		\frac{(-i\xi_1)(-i\xi_2)\chi_{k_1}(\xi_1)\chi_{k_2}(\xi_{ab})\chi_{k_3}(\xi_3)}{\Omega(\xi_1,\xi_{ab},\xi_3)}
		\ftxh{v}(\xi_1)\ftxh{v}(\xi_a)\ftxh{w}(\xi_b)\ftxh{w}(\xi_3),
		\\
		I\!I_3(k_1,k_3)
		&:=
		2^{-k_1}\sup_t\sum_{\substack{k_2\\j\in[2]}}\int_0^t
		\sum_{\xi_{12ab}=0}
		\frac{(-i\xi_1)(-i\xi_3)\chi_{k_1}(\xi_1)\chi_{k_2}(\xi_2)\chi_{k_3}(\xi_{ab})}{\Omega(\xi_1,\xi_2,\xi_{ab})}\ftxh{v}(\xi_1)\ftxh{v}(\xi_2)\ftxh{u}_j(\xi_a)\ftxh{u}_j(\xi_b).
	\end{align*}

	The term $B\!B$ can be estimated similarly to the boundary term $B$ in Lemma \ref{lem_est_ert_hl}. Indeed, we must have $k_2\lesssim k_1$, which -- together with \eqref{eq_est_res} -- leads to
	\begin{align*}
		\abs{B\!B(k_1,k_3)}
		&\lesssim
		2^{-k_1\alpha}\sum_{k_2}2^{-k_2/2}\norm{v_{k_1}}{\lebL^\infty_T\lebL^2}\norm{v_{k_2}}{\lebL^\infty_T\lebL^2}\norm{w_{k_3}}{\lebL^\infty_T\lebL^2}.
	\end{align*}
	Let us continue with bounding the summands $I\!I_1$, $I\!I_2$ and $I\!I_3$.
	Analogously to the proof of \eqref{est_ert}, we insert the factor $1=\sum_{k_a,k_b}\chi_{k_a}(\xi_a)\chi_{k_b}(\xi_b)$ into each of the terms $I\!I_1$, $I\!I_2$ and $I\!I_3$ in order to localize the variables $\xi_a$ and $\xi_b$.
	Using the bound $\smash{\abs{\Omega^{-1}(\xi_{ab},\xi_2,\xi_3)(-i\xi_{ab})^2}\lesssim2^{k_3(2-\alpha)}2^{-k_2}}$ and applying Lemma \ref{lem_i4}, we obtain
	\begin{align*}
		\abs{I\!I_1(k_1,k_3)}
		\lesssim
		T^d\!
		\sum_{k_2,k_a,k_b}\!
		\norm{v_{k_a}}{\spfb^{k_a}_T}\!
		\norm{w_{k_b}}{\spfb^{k_b}_T}\!
		\norm{v_{k_2}}{\spfb^{k_2}_T}\!
		\norm{w_{k_3}}{\spfb^{k_3}_T}\!
		\times
		\begin{cases}
			2^{k_3(3/2-\alpha)}2^{-k_2/2}
			&\text{if }k_a\sim k_b\geq k_1,
			\\
			2^{k_3(1-\alpha)}2^{-k_2/2}2^{k_b/2}
			&\text{if }k_a\sim k_1\geq k_b,
			\\
			2^{k_3(1-\alpha)}2^{-k_2/2}2^{k_a/2}
			&\text{if }k_b\sim k_1\geq k_a.
		\end{cases}
	\end{align*}
	Similarly, from $\abs{\Omega^{-1}(\xi_1,\xi_{ab},\xi_3)(-i\xi_1)(-i\xi_2)}\lesssim 2^{k_1(1-\alpha)}$ and Lemma \ref{lem_i4}, we conclude
	\begin{align*}
		\abs{I\!I_2(k_1,k_3)}
		\lesssim
		T^d\!
		\sum_{k_2,k_a,k_b}\!
		\norm{v_{k_1}}{\spfb^{k_1}_T}\!
		\norm{v_{k_a}}{\spfb^{k_a}_T}\!
		\norm{w_{k_b}}{\spfb^{k_b}_T}\!
		\norm{w_{k_3}}{\spfb^{k_3}_T}\!
		\times
		\begin{cases}
			2^{k_1(1-\alpha)}
			&\text{if }k_a\sim k_b\geq k_1,
			\\
			2^{-k_1\alpha}2^{k_a/2}2^{k_b/2}
			&\text{else},
		\end{cases}
	\end{align*}
	whereas $\abs{\Omega^{-1}(\xi_1,\xi_2,\xi_{ab})(-i\xi_1)(-i\xi_3)}\lesssim 2^{k_1(2-\alpha)}2^{-k_2}$ and Lemma \ref{lem_i4} imply
	\begin{align*}
		\abs{I\!I_3(k_1,k_3)}
		\lesssim
		T^d\!
		\sum_{\substack{k_2,k_a,k_b \\ j\in\{1,2\}}}\!
		\norm{v_{k_1}}{\spfb^{k_1}_T}\!
		\norm{v_{k_2}}{\spfb^{k_2}_T}\!
		\norm{u_{j,k_a}}{\spfb^{k_a}_T}\!
		\norm{u_{j,k_b}}{\spfb^{k_b}_T}\!
		\times
		\begin{cases}
			2^{k_1(3/2-\alpha)}2^{-k_2/2}
			&\text{if }k_a\sim k_b\geq k_3,
			\\
			2^{k_1(1-\alpha)}2^{k_b/2}2^{-k_2/2}
			&\text{if }k_1\sim k_a\geq k_b,
			\\
			2^{k_1(1-\alpha)}2^{k_a/2}2^{-k_2/2}
			&\text{if }k_1\sim k_b\geq k_a.
		\end{cases}
	\end{align*}
	The claim follows after summation of the obtained bounds in $k_1$ and $k_3$.
\end{proof}
\begin{lem}\label{lem_h12_LH}
	Let $s>3/2-\alpha$. Then, there exist $c,d>0$ such that we have
	\begin{align*}
		\sum_{k_1\ll k_3>n}\abs{I\!I(k_1,k_3)}
		\lesssim
		2^{-nc}\norm{v}{\spfc^{-1/2}_T}^2\norm{w}{\spfc^s_T}
		+
		T\norm{v}{\spfc^{-1/2}_T}^2\left(\norm{u_1}{\spfc^s_T}^2+\norm{u_2}{\spfc^s_T}^2\right).
	\end{align*}
\end{lem}
\begin{proof}
	We can repeat the proof of the previous lemma since we still have one high- and one low-frequency factor $v$.
\end{proof}

It remains to bound those terms, where both high-frequency factors are given by $v$. Here, the first step is to shift the derivative to the low-frequency factor as in the beginning of Section \ref{ss_for_solution}. 

\begin{lem}\label{lem_h12_HL}
	Let $s>3/2-\alpha$. Then, there exist $c,d>0$ such that we have
	\begin{align*}
		\sum_{k_3\ll k_1>n}\abs{I\!I(k_1,k_3)}
		\lesssim
		2^{-nc}\norm{v}{\spfc^{-1/2}_T}^2\norm{w}{\spfc^s_T}
		+
		T^d\norm{v}{\spfc^{-1/2}_T}^2\left(\norm{u_1}{\spfc^s_T}^2+\norm{u_2}{\spfc^s_T}^2\right).
	\end{align*}
\end{lem}
\begin{proof}
	Recall that we have
	\begin{align*}
		I\!I(k_1,k_3)
		=
		2^{-k_1}
		\sup_t\int_0^t\int_\bbt (-P_{k_1}^2\dx v)vP_{k_3}w.
	\end{align*}
	We can rewrite the spatial integral of $I(k_1,k_3)$ as follows:
	\begin{align*}
		\int_\bbt (-P_{k_1}^2\dx v)vP_{k_3}w
		&=
		\int_\bbt\! P_{k_1}^2vvP_{k_3}\dx w
		+
		\int_\bbt\! P_{k_1}^2v(\dx v) P_{k_3}w
		\\&=
		\underset{=:T^1}{\underbrace{\int_\bbt\! P_{k_1}^2vvP_{k_3}\dx w
		}}
		-
		\underset{=:T^2}{\underbrace{\frac{1}{2}\int_\bbt\! P_{k_1}v P_{k_1}vP_{k_3}\dx w}}
		+
		\underset{=:T^3}{\underbrace{\int_\bbt\! P_{k_1}v\left[P_{k_1}(\dx vP_{k_3}w)-P_{k_1}\dx vP_{k_3}\right]}}.
	\end{align*}
	Then, defining
	\begin{align*}
		I\!I^j(k_1,k_3):=2^{-k_1}\sup_t\int_0^t T^j
	\end{align*}
	for $j\in[3]$, we conclude
	\begin{align*}
		\abs{I\!I(k_1,k_3)}
		\leq
		\abs{I\!I^1(k_1,k_3)}
		+
		\abs{I\!I^2(k_1,k_3)}
		+
		\abs{I\!I^3(k_1,k_3)}.
	\end{align*}
	Now, as in Lemma \ref{lem_h12_HH}, we apply equation \eqref{eq_ibp_vvw} to $I\!I^1(k_1,k_3)$ and $I\!I^2(k_1,k_3)$. It follows
	\begin{align*}
		\abs{I\!I^j(k_1,k_3)}
		\leq
		\abs{B\!B^j(k_1,k_3)}+\abs{I\!I^j_1(k_1,k_3)}+\abs{I\!I^j_2(k_1,k_3)}+\abs{I\!I^j_3(k_1,k_3)},
		\qquad
		j\in[2],
	\end{align*}
	where the terms on the right-hand side for $j=1$ are given by
	\begin{align*}
		B\!B^1(k_1,k_3)
		&:=
		2^{-k_1}
		\sup_{t}\sum_{k_2}
		\left[ \sum_{\xi_{123}=0}\frac{-i\xi_3\chi_{k_1}^2(\xi_1)\chi_{k_2}(\xi_2)\chi_{k_3}(\xi_3)}{\Omega(\xi_1,\xi_2,\xi_3)}\hat{v}(\xi_1)\hat{v}(\xi_2)\hat{w}(\xi_3)\right]_{0}^{t},
		\\
		I\!I^1_1(k_1,k_3)
		&:=
		2^{-k_1}\sup_t\sum_{k_2}\int_0^t
		\sum_{\xi_{ab23}=0}
		\frac{(-i\xi_{ab})(-i\xi_3)\chi_{k_1}^2(\xi_{ab})\chi_{k_2}(\xi_2)\chi_{k_3}(\xi_3)}{\Omega(\xi_{ab},\xi_2,\xi_3)}\hat{v}(\xi_a)\hat{w}(\xi_b)\hat{v}(\xi_2)\hat{w}(\xi_3),
		\\
		I\!I^1_2(k_1,k_3)
		&:=
		2^{-k_1}\sup_t\sum_{k_2}\int_0^t
		\sum_{\xi_{1ab3}=0}
		\frac{(-i\xi_{ab})(-i\xi_3)\chi_{k_1}^2(\xi_1)\chi_{k_2}(\xi_{ab})\chi_{k_3}(\xi_3)}{\Omega(\xi_1,\xi_{ab},\xi_3)}\hat{v}(\xi_1)\hat{v}(\xi_a)\hat{w}(\xi_b)\hat{w}(\xi_3),
		\\
		I\!I^1_3(k_1,k_3)
		&:=
		2^{-k_1}\sup_t\sum_{\substack{k_2\\ j\in[2]}}\int_0^t
		\sum_{\xi_{12ab}=0}
		\frac{(-i\xi_{ab})^2\chi_{k_1}^2(\xi_1)\chi_{k_2}(\xi_2)\chi_{k_3}(\xi_{ab})}{\Omega(\xi_1,\xi_2,\xi_{ab})}\hat{v}(\xi_1)\hat{v}(\xi_2)\ftxh{u}_j(\xi_{a})\ftxh{u}_j(\xi_{b}).
	\end{align*}
	The corresponding terms for $j=2$ follow by obvious modifications.
	It remains to bound the quantities $I\!I^1$, $I\!I^2$ and $I\!I^3$.\\
	
	\textbf{Estimating $I\!I^1$.} Note that $k_1\sim k_2$ holds.
	The boundary term can be estimated similar to the boundary term $B$ in Lemma \ref{lem_est_ert_hl} by
	\begin{align*}
		\abs{B\!B^1(k_1,k_3)}
		\lesssim
		\sum_{k_2}2^{k_2(-1-\alpha)}2^{k_3/2}\norm{v_{k_1}}{\lebL^\infty_T\lebL^2}\norm{v_{k_2}}{\lebL^\infty_T\lebL^2}\norm{w}{\lebL^\infty_T\lebL^2}.
	\end{align*}
	Now, we estimate $I\!I_3^1$. Using $\abs{\Omega(\xi_{ab},\xi_2,\xi_3)^{-1}(-i\xi_3)^2}\lesssim 2^{-k_1\alpha}2^{k_3}$ as well as Lemma \ref{lem_i4}, we obtain
	\begin{align*}
		\abs{I\!I_3^1(k_1,k_3)}
		\lesssim
		T^d\!
		\sum_{\substack{k_2,k_a,k_b\\j\in[2]}}\!
		\norm{v_{k_1}}{\spfb^{k_1}_T}\!
		\norm{v_{k_2}}{\spfb^{k_2}_T}\!
		\norm{u_{j,k_a}}{\spfb^{k_a}_T}\!
		\norm{u_{j,k_b}}{\spfb^{k_b}_T}\!
		\times
		\begin{cases}
			2^{k_1(1-\alpha)}
			&\text{if }k_a\sim k_b\gtrsim k_1,
			\\
			2^{-k_1\alpha}2^{k_a/2}2^{k_3/2}
			&\text{if }k_1\gtrsim k_a,k_b.
		\end{cases}
	\end{align*}
	Similarly, the bound $\abs{\Omega^{-1}(\xi_{ab},\xi_2,\xi_3)(-i\xi_{ab})(-i\xi_3)}\lesssim 2^{k_1(1-\alpha)}$ and Lemma \ref{lem_i4} lead to
	\begin{align*}
		\abs{I\!I^1_1(k_1,k_3)}
		\lesssim
		T^d\!
		\sum_{k_2,k_a,k_b}\!
		\norm{v_{k_a}}{\spfb^{k_a}_T}\!
		\norm{w_{k_b}}{\spfb^{k_b}_T}\!
		\norm{v_{k_2}}{\spfb^{k_2}_T}\!
		\norm{w_{k_3}}{\spfb^{k_3}_T}\!
		\times
		\begin{cases}
			2^{k_1(1/2-\alpha)}2^{k_3/2}
			&\text{if }k_a\sim k_b\gtrsim k_1,
			\\
			2^{-k_2\alpha}2^{k_a/2}2^{k_3/2}
			&\text{if }k_b\sim k_1\gtrsim k_a.
		\end{cases}
	\end{align*}
	Note that we did not cover the case $k_2\sim k_a\gg k_b$ in the previous estimate. In that case, we have
	\begin{equation}\label{eq_bnd_II_1}
		\abs{I\!I^1_1(k_1,k_3)}_{k_2\sim k_a\gg k_b}
		\lesssim
		2^{-k_1}
		\sum_{k_2\sim k_a\gg k_b}\abs{\sup_t\int_0^t\sum_{\xi_{ab23}=0}
		m(\xi_a,\xi_b,\xi_2,\xi_3)
		\hat{v}(\xi_a)\hat{w}(\xi_b)\hat{v}(\xi_2)\hat{w}(\xi_3)},
	\end{equation}
	where
	\begin{equation}\label{eq_def_m}
		m(\xi_a,\xi_b,\xi_2,\xi_3)
		:=
		\frac{(-i\xi_{ab})\chi_{k_1}^2(\xi_{ab})\chi_{k_a}(\xi_a)\chi_{k_2}(\xi_2)}{\Omega(\xi_{ab},\xi_2,\xi_3)}(-i\xi_3)\chi_{k_3}(\xi_3)\chi_{k_b}(\xi_b).
	\end{equation}
	Here, a direct estimate of the modulus of $m$ only yields the bound $2^{k_1(1-\alpha)}2^{-k_3}$, which -- after an application of Lemma \ref{lem_i4} -- would lead to a factor $2^{-k_1\alpha}2^{k_3/2}2^{k_b/2}$. As $\alpha<1$, this is insufficient.
	Thus, we need to use the symmetry of $I\!I^1_1$ in the variables $\xi_2$ and $\xi_a$. The spatial integral of \eqref{eq_bnd_II_1} can be written as
	\begin{equation}\label{eq_m}
		\begin{split}
			&2\sum_{\xi_{ab23}=0}
			m(\xi_a,\xi_b,\xi_2,\xi_3)
			\hat{v}(\xi_a)\hat{w}(\xi_b)\hat{v}(\xi_2)\hat{w}(\xi_3)
			\\=
			&\sum_{\xi_{ab23}=0}
			\left[m(\xi_a,\xi_b,\xi_2,\xi_3)
			+
			m(\xi_2,\xi_b,\xi_a,\xi_3)\right]
			\hat{v}(\xi_a)\hat{w}(\xi_b)\hat{v}(\xi_2)\hat{w}(\xi_3).
		\end{split}
	\end{equation}
	For all $\xi_a$, $\xi_b$, $\xi_2$, $\xi_3$ satisfying $\xi_{ab23}=0$ direct calculations yield
	\begin{align*}
		&m(\xi_a,\xi_b,\xi_2,\xi_3)
		\\=
		&\,\frac{i(\xi_3-\xi_b)\chi_{k_1}^2(\xi_{ab})\chi_{k_a}(\xi_a)\chi_{k_2}(\xi_2)}{\Omega(\xi_{ab},\xi_2,\xi_3)}(-i\xi_3)\chi_{k_3}(\xi_3)\chi_{k_b}(\xi_b)
		&=:m_1
		\\+
		&\,\frac{(i\xi_{2b})\left[\chi_{k_1}^2(\xi_{ab})-\chi_{k_1}^2(\xi_{2b})\right]\chi_{k_a}(\xi_a)\chi_{k_2}(\xi_2)}{\Omega(\xi_{ab},\xi_2,\xi_3)}(-i\xi_3)\chi_{k_3}(\xi_3)\chi_{k_b}(\xi_b)
		&=:m_2
		\\+
		&\,\frac{(i\xi_{2b})\chi_{k_1}^2(\xi_{2b})\left[\chi_{k_a}(\xi_a)-\chi_{k_a}(\xi_2)\right]\chi_{k_2}(\xi_2)}{\Omega(\xi_{ab},\xi_2,\xi_3)}(-i\xi_3)\chi_{k_3}(\xi_3)\chi_{k_b}(\xi_b)
		&=:m_3
		\\+
		&\,\frac{(i\xi_{2b})\chi_{k_1}^2(\xi_{2b})\chi_{k_a}(\xi_2)\left[\chi_{k_2}(\xi_2)-\chi_{k_2}(\xi_a)\right]}{\Omega(\xi_{ab},\xi_2,\xi_3)}(-i\xi_3)\chi_{k_3}(\xi_3)\chi_{k_b}(\xi_b)
		&=:m_4
		\\+
		&\,(i\xi_{2b})\chi_{k_1}^2(\xi_{2b})\chi_{k_a}(\xi_2)\chi_{k_2}(\xi_a)\left[\Omega^{-1}(\xi_{ab},\xi_2,\xi_3)-\Omega^{-1}(\xi_{2b},\xi_a,\xi_3)\right](-i\xi_3)\chi_{k_3}(\xi_3)\chi_{k_b}(\xi_b)
		&=:m_5
		\\+
		&\,\frac{(i\xi_{2b})\chi_{k_1}^2(\xi_{2b})\chi_{k_a}(\xi_2)\chi_{k_2}(\xi_a)}{\Omega(\xi_{2b},\xi_a,\xi_3)}(-i\xi_3)\chi_{k_3}(\xi_3)\chi_{k_b}(\xi_b)
		&=:m_6.
	\end{align*}
	Note that $m_6=-m(\xi_2,\xi_b,\xi_a,\xi_3)$ holds. Hence, with \eqref{eq_m}, we obtain
	\begin{equation}\label{eq_bnd_II_2}
		\eqref{eq_bnd_II_1}
		\lesssim
		2^{-k_1}\sum_{\substack{k_2\sim k_a\gg k_b\\j\in[5]}}
		\abs{\sup_t\int_0^t
		\sum_{\xi_{ab23}=0}
		m_j
		\hat{v}(\xi_a)\hat{w}(\xi_b)\hat{v}(\xi_2)\hat{w}(\xi_3)}.
	\end{equation}
	The next step consists of proving
	\begin{align*}
		\abs{m_j}
		\lesssim
		2^{\max\{k_3,k_b\}}2^{-k_2\alpha},
		\qquad
		j\in[5].
	\end{align*}
	For $m_1$ the desired bound follows from \eqref{eq_est_res}, whereas for $m_5$ it is a consequence of Corollary \ref{cor_est_res_inv}. For $m_2$, $m_3$ and $m_4$ the claim follows from estimates of the form
	\begin{align*}
		\abs{\chi_{k_1}^2(\xi_{ab})-\chi_{k_1}^2(\xi_{2b})}
		=
		\abs{\chi_{k_1}^2(\xi_{ab})-\chi_{k_1}^2(-\xi_{2b})}
		\lesssim
		\abs{\xi_{ab}+\xi_{2b}}\norm{(\chi_{k_1}^2)'}{\lebL^\infty}
		\lesssim
		\abs{\xi_b-\xi_3}2^{-k_1}.
	\end{align*}
	Thus, applying Lemma \ref{lem_i4} to the right-hand side of \eqref{eq_bnd_II_2}, we conclude
	\begin{align*}
		\abs{I\!I^1_1(k_1,k_3)}_{k_2\sim k_a\gg k_b}
		\lesssim
		T^d\!
		\sum_{k_2\sim k_a\gg k_b}\!
		2^{-k_1(1+\alpha)}2^{\max\{k_3,k_b\}}2^{k_3/2}2^{k_b/2}
		\norm{v_{k_a}}{\spfb^{k_a}_T}\!
		\norm{w_{k_b}}{\spfb^{k_b}_T}\!
		\norm{v_{k_2}}{\spfb^{k_2}_T}\!
		\norm{w_{k_3}}{\spfb^{k_3}_T}.
	\end{align*}
	The term $I\!I^1_2(k_1,k_3)$ can be dealt with by the same arguments.\\
	
	\textbf{Estimating $I\!I^2$.} To handle the term $I\!I^2$, we can repeat the calculations made for $I\!I^1$ replacing $\chi_{k_1}^2$ and $\chi_{k_2}$ by $\chi_{k_1}$ and omitting the sum over $k_2$. These changes only have an impact on the implicit constants.\\
	
	\textbf{Estimating $I\!I^3$.} Now, we consider
	\begin{align*}
		I\!I^3(k_1,k_3)
		&=
		2^{-k_1}\sup_t\int_0^t
		\int_\bbt P_{k_1}v\left[P_{k_1}(\dx vP_{k_3}w)-P_{k_1}(\dx v)P_{k_3}w\right]
		\\&=
		2^{-k_1}\sup_t\int_0^t
		\sum_{\xi_{123}=0}\underset{=:\nu(\xi_1,\xi_2,\xi_3)}{\underbrace{\xi_2\chi_{k_1}(\xi_1)\left[\chi_{k_1}(\xi_{23})-\chi_{k_1}(\xi_2)\right]\chi_{k_3}(\xi_3)}}\hat{v}(\xi_1)\hat{v}(\xi_2)\hat{w}(\xi_3).
	\end{align*}
	Applying \eqref{eq_ibp_vvw}, we obtain
	\begin{align*}
		\abs{I\!I^3(k_1,k_3)}
		\leq
		\abs{B\!B^3(k_1,k_3)}
		+
		\abs{I\!I^3_1(k_1,k_3)}
		+
		\abs{I\!I^3_2(k_1,k_3)}
		+
		\abs{I\!I^3_{12}(k_1,k_3)}
		+
		\abs{I\!I^3_3(k_1,k_3)},
	\end{align*}
	where the terms on the right-hand side are given by
	\begin{align*}
		B\!B^3(k_1,k_3)
		&:=
		2^{-k_1}
		\sup_{t}
		\left[ \sum_{\xi_{123}=0}\frac{\nu}{\Omega}(\xi_1,\xi_2,\xi_3)\hat{v}(\xi_1)\hat{v}(\xi_2)\hat{w}(\xi_3)\right]_{0}^{t},
		\\
		I\!I^3_{1}(k_1,k_3)
		&:=
		2^{-k_1}\sup_t\int_0^t\sum_{\substack{\xi_{ab23}=0\\ \xi_a\not\sim\xi_2}}
		\xi_{ab}\frac{\nu}{\Omega}(\xi_{ab},\xi_2,\xi_3)
		\hat{v}(\xi_a)\hat{v}(\xi_2)\hat{w}(\xi_3)\hat{w}(\xi_b),
		\\
		I\!I^3_{2}(k_1,k_3)
		&:=
		2^{-k_1}\sup_t\int_0^t\sum_{\substack{\xi_{1ab3}=0\\ \xi_1\not\sim\xi_a}}
		\xi_{ab}\frac{\nu}{\Omega}(\xi_1,\xi_{ab},\xi_3)
		\hat{v}(\xi_1)\hat{v}(\xi_a)\hat{w}(\xi_3)\hat{w}(\xi_b),
		\\
		I\!I^3_{12}(k_1,k_3)
		&:=
		2^{-k_1}\sup_t\int_0^t\sum_{\substack{\xi_{ab23}=0\\ \xi_a\sim\xi_2}}
		\left[\xi_{ab}\frac{\nu}{\Omega}(\xi_{ab},\xi_2,\xi_3)+\xi_{2b}\frac{\nu}{\Omega}(\xi_a,\xi_{2b},\xi_3)\right]
		\hat{v}(\xi_a)\hat{v}(\xi_2)\hat{w}(\xi_3)\hat{w}(\xi_b),
		\\
		I\!I_3^3(k_1,k_3)
		&:=
		2^{-k_1}\sup_t\sum_{j\in[2]} \int_0^t
		\sum_{\xi_{12ab}=0}
		(-i\xi_{ab})\frac{\nu}{\Omega}(\xi_1,\xi_2,\xi_{ab})
		\hat{v}(\xi_1)\hat{v}(\xi_2)\ftxh{u}_j(\xi_{a})\ftxh{u}_j(\xi_{b}).
	\end{align*}
	Here, $I\!I^3_{12}$ contains those summands, which are excluded in the summation of $I\!I^3_1$ and $I\!I^3_2$. When estimating $I\!I^3_{12}$, we will emphasize the advantage of this definition.
	Also, note that the variable $\xi_2$ is already of size $2^{k_1}$ due to the localization in $\nu(\xi_1,\xi_2,\xi_3)$. Nonetheless, we localize each $\xi_2$ to dyadic frequency ranges to improve the notation.
	
	Similar to the estimation for the boundary term $B$ in Lemma \ref{lem_est_ert_hl}, we bound $B\!B^3$ by
	\begin{align*}
		\abs{B\!B^3(k_1,k_3)}
		\lesssim
		\sum_{k_2} 2^{-k_1(1+\alpha)}2^{k_3/2}\norm{v_{k_1}}{\lebL^\infty_T\lebL^2}\norm{v_{k_2}}{\lebL^\infty_T\lebL^2}\norm{w_{k_3}}{\lebL^\infty_T\lebL^2}.
	\end{align*}
	Using Lemma \ref{lem_i4} and the bounds $\abs{\frac{\nu}{\Omega}(\xi_1,\xi_2,\xi_{ab})(-i\xi_{ab})}\lesssim 2^{-k_1\alpha}2^{k_3}$, ${\abs{\frac{\nu}{\Omega}(\xi_{ab},\xi_2,\xi_3)(\xi_{ab})}\lesssim 2^{-k_1\alpha}2^{k_3}}$ and $\abs{\frac{\nu}{\Omega}(\xi_1,\xi_{ab},\xi_3)(\xi_{ab})}\lesssim 2^{-k_1\alpha}2^{k_3}$, it follows
	\begin{align*}
		\abs{I\!I_3^3(k_1,k_3)}
		\lesssim
		T^d\!\!
		\sum_{\substack{k_2 ,k_a, k_b\\j\in[2]}}\!
		\norm{v_{k_1}}{\spfb^{k_1}_T}\!
		\norm{v_{k_2}}{\spfb^{k_2}_T}\!
		\norm{u_{j,k_a}}{\spfb^{k_a}_T}\!
		\norm{u_{j,k_b}}{\spfb^{k_b}_T}\!
		\times
		\begin{cases}
			2^{k_1(1-\alpha)}
			&\text{if }k_a\sim k_b\gtrsim k_2,
			\\
			2^{-k_1(1+\alpha)}2^{k_a3/2}2^{k_b/2}
			&\text{if }k_3\sim k_a\gtrsim k_b,
			\\
			2^{-k_1(1+\alpha)}2^{k_b3/2}2^{k_a/2}
			&\text{if }k_3\sim k_b\gtrsim k_a,
		\end{cases}
	\end{align*}
	as well as
	\begin{align*}
		\abs{I\!I_1^3(k_1,k_3)}
		\lesssim
		T^d\!\!
		\sum_{k_2,k_a,k_b}\!
		\norm{v_{k_a}}{\spfb^{k_a}_T}\!
		\norm{w_{k_b}}{\spfb^{k_b}_T}\!
		\norm{v_{k_2}}{\spfb^{k_2}_T}\!
		\norm{w_{k_3}}{\spfb^{k_3}_T}\!
		\times
		\begin{cases}
			2^{k_1(1/2-\alpha)}2^{k_3/2}
			&\text{if }k_a\sim k_b\gg k_2,
			\\
			2^{-k_1\alpha}2^{k_3/2}2^{k_a/2}
			&\text{if }k_2\sim k_b\gg k_a,
		\end{cases}
	\end{align*}
	and
	\begin{align*}
		\abs{I\!I_2^3(k_1,k_3)}
		\lesssim
		T^d\!\!
		\sum_{k_2,k_a,k_b}\!
		\norm{v_{k_a}}{\spfb^{k_a}_T}\!
		\norm{w_{k_b}}{\spfb^{k_b}_T}\!
		\norm{v_{k_2}}{\spfb^{k_2}_T}\!
		\norm{w_{k_3}}{\spfb^{k_3}_T}\!
		\times
		\begin{cases}
			2^{k_1(1/2-\alpha)}2^{k_3/2}
			&\text{if }\xi_a\sim\xi_b\gg\xi_1,
			\\
			2^{-k_1\alpha}2^{k_3/2}2^{k_a/2}
			&\text{if }\xi_1\sim\xi_b\gg\xi_a.
		\end{cases}
	\end{align*}
	It remains to bound $I\!I_{12}^3$. A direct estimate leads to the insufficient factor $2^{-k_1\alpha}2^{k_3/2}2^{k_b/2}$ forcing us to use some cancellation in the symbol. We have
	\begin{align*}
		\xi_{ab}\frac{\nu}{\Omega}(\xi_{ab},\xi_2,\xi_3)+\xi_{2b}\frac{\nu}{\Omega}(\xi_a,\xi_{2b},\xi_3)
		&=
		\left[\xi_{ab}+\xi_{2b}\right]\frac{\nu}{\Omega}(\xi_{ab},\xi_2,\xi_3)
		&=:A_1
		\\&+
		\left[\nu(\xi_a,\xi_{2b},\xi_3)-\nu(\xi_{ab},\xi_2,\xi_3)\right]\xi_{2b}\Omega^{-1}(\xi_{ab},\xi_2,\xi_3)
		&=:A_2
		\\&+
		\left[\Omega^{-1}(\xi_a,\xi_{2b},\xi_3)-\Omega^{-1}(\xi_{ab},\xi_2,\xi_3)\right]\xi_{2b}\nu(\xi_a,\xi_{2b},\xi_3)
		&=:A_3.
	\end{align*}
	To bound $A_1$, we use $\abs{\nu}\lesssim 2^{k_3}$, \eqref{eq_est_res} as well as $\xi_{ab}+\xi_{2b}=\xi_b-\xi_3$, which holds due to $\xi_{ab23}=0$. It follows
	\begin{align*}
		\abs{A_1}
		\lesssim
		\abs{\xi_b-\xi_3}\abs{\frac{\nu}{\Omega}(\xi_{ab},\xi_2,\xi_3)}
		\lesssim
		2^{-k_2\alpha}2^{\max\{k_3,k_b\}}.
	\end{align*}
	We observe
	\begin{align*}
		\nu(\xi_{a},\xi_{2b},\xi_3)-\nu(\xi_{ab},\xi_{2},\xi_3)
		&=
		\xi_{b}\chi_{k_1}(\xi_{a})\left[\chi_{k_1}(\xi_{23b})-\chi_{k_1}(\xi_{2b})\right]\chi_{k_3}(\xi_3)
		\\&+
		\xi_{2}\left[\chi_{k_1}(\xi_{a})-\chi_{k_1}(\xi_{ab})\right]\left[\chi_{k_1}(\xi_{23b})-\chi_{k_1}(\xi_{2b})\right]\chi_{k_3}(\xi_3)
		\\&+
		\xi_{2}\chi_{k_1}(\xi_{ab})\left[\chi_{k_1}(\xi_{23b})-\chi_{k_1}(\xi_{2b})-\chi_{k_1}(\xi_{23})+\chi_{k_1}(\xi_{2})\right]\chi_{k_3}(\xi_3).
	\end{align*}
	Trivially, each term on the right-hand side has modulus bounded by $2^{-k_1}2^{k_b}2^{k_3}$. Hence, it follows
	\begin{align*}
		\abs{A_2}
		\lesssim
		2^{k_b}2^{-k_1}2^{k_3}2^{k_2}2^{-k_2\alpha}2^{-k_3}
		=
		2^{-k_2\alpha}2^{k_b}.
	\end{align*}
	Concerning $A_3$, we can use Lemma \ref{cor_est_res_inv} and immediately arrive at
	\begin{align*}
		\abs{A_3}
		\lesssim
		2^{-k_2(1+\alpha)}2^{-k_3}2^{k_b}2^{k_2}2^{k_3}
		=
		2^{-k_2\alpha}2^{k_b}.
	\end{align*}
	Thus, we can bound the modulus of the symbol appearing in $I\!I^3_{12}$ by $2^{-k_2\alpha}2^{k_b}$. Lemma \ref{lem_i4} yields
	\begin{align*}
		\abs{I\!I_{12}^3(k_1,k_3)}
		\lesssim
		T^d
		\sum_{k_a,k_b,k_2}
		2^{-k_2(1+\alpha)}2^{\max\{k_3,k_b\}}2^{k_3/2}2^{k_b/2}
		\norm{v_{k_a}}{\spfb^{k_a}_T}
		\norm{w_{k_b}}{\spfb^{k_b}_T}
		\norm{v_{k_2}}{\spfb^{k_2}_T}
		\norm{w_{k_3}}{\spfb^{k_3}_T}.
	\end{align*}
	This concludes the proof after summation of the obtained bounds in $k_1$ and $k_3$.
\end{proof}

Finally, observe that the claimed estimate \eqref{est_ezt} is a consequence of \eqref{eq_h12_diff} and the Lemmata \ref{lem_h12_LL}, \ref{lem_h12_HH}, \ref{lem_h12_LH} and \ref{lem_h12_HL}.

\subsection{Proof of the third energy estimate}\label{ss_for_differences_ii}

We end Section \ref{s_energy} with the proof of estimate \eqref{est_est}. Recall that $u_1$, $u_2$ are smooth solutions of \eqref{eq_PDE} and that we defined $v=u_1-u_2$ and $w=u_1+u_2$.
Rewriting \eqref{eq_PDE_difference}, we obtain that $v$ satisfies the equation
\begin{equation}\label{eq_PDE_difference_ii}
	\dt v
	+
	\dx D_x^\alpha v
	=
	\dx(vw)
	=
	\dx(vv)
	+
	2\dx(vu_2).
\end{equation}
Thus, as in the beginning of Section \ref{ss_for_solution}, it follows
\begin{equation}\label{eq_split_x_y}
	\norm{v}{\spec_T^s}^2
	-
	\norm{v_0}{\lebH^s}^2
	\lesssim
	\underset{=:X}{\underbrace{\sum_{k_1}
	2^{k_12s}
	\sup_{t}\int_0^{t}\int_\bbt P_{k_1}^2(\dx v)vv}}
	+
	\underset{=:Y}{\underbrace{\sum_{k_1}
	2^{k_12s}
	\sup_{t}\int_0^{t}\int_\bbt P_{k_1}^2(\dx v)vu_2}}.
\end{equation}
The advantage of splitting the term above on the right-hand side into $X$ and $Y$ is that $X$ can be estimated by proceeding as in Section \ref{ss_for_solution}, whereas the bound for $Y$ follows by modifying arguments used in Section \ref{ss_for_differences_ii}.

Let us begin with estimating $X$. As in Section \ref{ss_for_solution}, we frequency-localize the second and third factor in the integrand and restrict ourselves to the case $k_1\sim k_2\gg k_3$. Then, using
\begin{align*}
	&-\Omega(\xi_1,\xi_2,\xi_3)\hat{v}(\xi_1)\hat{v}(\xi_2)\hat{v}(\xi_3)
	\\=&
	\dt \left[\hat{v}(\xi_1)\hat{v}(\xi_2)\hat{v}(\xi_3)\right]
	+
	i\xi_1\widehat{vw}(\xi_1)\hat{v}(\xi_2)\hat{v}(\xi_3)
	+
	i\xi_2\hat{v}(\xi_1)\widehat{vw}(\xi_2)\hat{v}(\xi_3)
	+
	i\xi_3\hat{v}(\xi_1)\hat{u}(\xi_2)\widehat{vw}(\xi_3)
\end{align*}
instead of \eqref{eq_ibp_uuu} and noting that $w=v+2u_2$ holds, we deduce
\begin{align*}
	X
	\lesssim
	\sum_{n\geq k_1\gg k_3}
	I(k_1,k_3;v)
	+
	\sum_{n<k_1\gg k_3}
	I(k_1,k_3;v).
\end{align*}
Here, we slightly abused the notation introduced in \eqref{eq_split_i} -- the additional term $v$ indicates that we replace the factors $uuu$ in the spatial integral of \eqref{eq_split_i} by $vvv$.
Repeating the proof of Lemma \ref{lem_est_ert_ll}, it follows
\begin{align*}
	\sum_{n\geq k_1\gg k_3}
	\abs{I(k_1,k_3;v)}
	\lesssim
	T2^{2n}\norm{v}{\spfc_T^s}^3.
\end{align*}
Similar to Lemma \ref{lem_est_ert_hl}, we obtain
\begin{align*}
	\abs{I(k_1,k_3;v)}
	&\lesssim
	\abs{B(k_1,k_3;v)}
	\\&+
	\abs{I_1(k_1,k_3;v)}
	+
	\abs{I_2(k_1,k_3;v)}
	+
	\abs{I_{12}(k_1,k_3;v)}
	+
	\abs{I_{3}(k_1,k_3;v)}
	\\&+
	\abs{I_1(k_1,k_3;u_2)}
	+
	\abs{I_2(k_1,k_3;u_2)}
	+
	\abs{I_{12}(k_1,k_3;u_2)}
	+
	\abs{I_{3}(k_1,k_3;u_2)},
\end{align*}
where we again changed the notation. Here, each term in the first line has its integrand $uuu$ replaced by $vvv$, whereas each term in the second line has its integrand $uuuu$ replaced by $vvvv$. The terms in the third line have their integrand $uuuu$ replaced by $vu_2vv$, $vvu_2v$ or $vvvu_2$. An inspection of the proof of Lemma \ref{lem_est_ert_hl} shows that the bounds of $B$, $I_1$, $I_2$ and $I_3$ do not depend on the integrand, whereas the estimate for $I_{12}$ requires an integrand that is symmetric in its high-frequency variables. Hence, repeating the proof of Lemma \ref{lem_est_ert_hl}, we get
\begin{align*}
	\sum_{n<k_1\gg k_3}
	\abs{B(k_1,k_3;v)}
	&\lesssim
	2^{-nc}\norm{v}{\spfc^s_T}^3,
	\\
	\sum_{n<k_1\gg k_3}
	\abs{I_1(k_1,k_3;v)}
	+
	\abs{I_2(k_1,k_3;v)}
	+
	\abs{I_{12}(k_1,k_3;v)}
	+
	\abs{I_{3}(k_1,k_3;v)}
	&\lesssim
	T^d\norm{v}{\spfc^s_T}^4,
	\\
	\sum_{n<k_1\gg k_3}
	\abs{I_1(k_1,k_3;u_2)}
	+
	\abs{I_2(k_1,k_3;u_2)}
	+
	\abs{I_{3}(k_1,k_3;u_2)}
	&\lesssim
	T^d\norm{v}{\spfc^s_T}^3\norm{u_2}{\spfc^s_T}.
\end{align*}
In order to end the estimation of $X$, it remains to consider $I_{12}(k_1,k_3;u_2)$ given by
\begin{align*}
	I_{12}(k_1,k_3;u_2)
	&=
	2^{k_12r}\sup_{t}\sum_{\substack{k_2\\j\in[3]}}\int_0^{t}\sum_{\substack{\xi_{ab23}=0\\ \xi_a\not\sim\xi_b}} \left[\frac{\sigma_j}{\Omega}(\xi_{ab},\xi_2,\xi_3)+\frac{\sigma_j}{\Omega}(\xi_2,\xi_{ab},\xi_3)\right]\xi_{ab}\ftxh{v}(\xi_a)\ftxh{u}_2(\xi_b)\ftxh{v}(\xi_2)\ftxh{v}(\xi_3).
\end{align*}
\begin{lem}
	Let $s>3/2-\alpha$. Then, there exists $d>0$ such that we have
	\begin{align*}
		\sum_{n<k_1\gg k_3}\abs{I_{12}(k_1,k_3;u_2)}
		&\lesssim
		T^d\norm{v}{\spfc_T^s}^3\norm{u_2}{\spfc_T^s}
		+
		T^d\norm{v}{\spfc_T^{-1/2}}\norm{v}{\spfc_T^s}^2\norm{u_2}{\spfc_T^{s+2-\alpha}}.
	\end{align*}
\end{lem}
\begin{proof}
	In the case $k_a\gg k_b$ we have $k_a\sim k_2$ and the integrand $\ftxh{v}(\xi_a)\ftxh{u}_2(\xi_b)\ftxh{v}(\xi_2)\ftxh{v}(\xi_3)$ is symmetric in the high-frequency variables $\xi_a$ and $\xi_2$. As in Lemma \ref{lem_est_ert_hl}, it follows
	\begin{align*}
		\abs{I_{12}(k_1,k_3;u_2)}_{k_a\gg k_b}
		&\lesssim
		T^d\sum_{k_a, k_b, k_2} 2^{k_2(2s-\alpha)}2^{\max\{k_3,k_b\}}2^{k_3/2}2^{k_b/2}
		\norm{u_{2,k_b}}{\spfb^{k_b}_T}\prod_{j\in\{2,3,a\}}\norm{v_{k_j}}{\spfb^{k_j}_T}.
	\end{align*}
	In the remaining case $k_b\gg k_a$ we have $k_b\sim k_2$ and the integrand is no longer symmetric in the high-frequency variables $\xi_b$ and $\xi_2$. Hence, the argument given in Lemma \ref{lem_est_ert_hl} does not apply. Instead, we use the direct bound of the symbol
	\begin{align*}
		\abs{\frac{\sigma_j}{\Omega}(\xi_{ab},\xi_2,\xi_3)\xi_{ab}+\frac{\sigma_j}{\Omega}(\xi_2,\xi_{ab},\xi_3)\xi_{ab}}\lesssim 2^{k_1(1-\alpha)}
	\end{align*}
	as well as the fact that $u_2$ can be estimated in a high-regularity norm. An application of Lemma \ref{lem_i4} yields
	\begin{align*}
		\abs{I_{12}(k_1,k_3;u_2)}_{k_b\gg k_a}
		&\lesssim
		T^d\sum_{k_a, k_b, k_2} 2^{k_2(2s+1-\alpha)}2^{k_3/2}2^{k_a/2}
		\norm{u_{2,k_b}}{\spfb^{k_b}_T}\prod_{j\in\{2,3,a\}}\norm{v_{k_j}}{\spfb^{k_j}_T}.
	\end{align*}
	Recall that we assumed $k_1\sim k_2\gg k_3$. Thus, summation in $k_1$ and $k_3$ concludes the proof.
\end{proof}

Finally, we obtain
\begin{equation}\label{eq_X_bound}
	\abs{X}
	\lesssim
	(T2^{2n}+2^{-nc})\norm{v}{\spfc^s_T}^3
	+
	T^d\norm{v}{\spfc^s_T}^4
	+
	T^d\norm{v}{\spfc^s_T}^3\norm{u_2}{\spfc^s_T}
	+
	T^d\norm{v}{\spfc^{-1/2}_T}\norm{v}{\spfc^s_T}^2\norm{u_2}{\spfc^{s+2-\alpha}_T}.
\end{equation}

Let us continue by estimating $Y$. We will proceed similar to Section \ref{ss_for_differences_i}. However, as the regularity is different, we need to modify some steps, which is why we will provide a complete proof. As before, we have
\begin{align}\label{eq_Y_bound}
	\begin{split}
		Y
		&\lesssim
		\sum_{k_1, k_3}\underset{=:I\!I\!I(k_1,k_3)}{\underbrace{2^{k_12s} \int_0^T\int_\bbt P_{k_1}^2(\dx v)vP_{k_3}u_2}}
		\\&=
		\sum_{k_1,k_3\leq n}I\!I\!I(k_1,k_3)
		+
		\sum_{k_1\sim k_3> n}I\!I\!I(k_1,k_3)
		+
		\sum_{k_1\ll k_3> n}I\!I\!I(k_1,k_3)
		+
		\sum_{k_3\ll k_1> n}I\!I\!I(k_1,k_3).
	\end{split}
\end{align}

Let us estimate the low-frequency contribution first.
\begin{lem}\label{lem_hs_ll}
	Let $s>3/2-\alpha$. Then, we have
	\begin{align*}
		\sum_{k_1,k_3\leq n}\abs{I\!I\!I(k_1,k_3)}
		\lesssim
		T2^{2n}\norm{v}{\spfc^{s}_T}^2\norm{u_2}{\spfc^s_T}.
	\end{align*}	
\end{lem}
\begin{proof}
	Writing $v=\sum_{k_2}v_{k_2}$ and noticing that $k_2\lesssim n$ holds, the claim follows from the estimate
	\begin{align*}
		\abs{I\!I\!I(k_1,k_3)}
		\lesssim
		\sum_{k_2}2^{k_12s}\int_0^T\int_\bbt \abs{(\dx v)_{k_1}v_{k_2}u_{2,k_3}}
		\lesssim
		T\sum_{k_2} 2^{k_1(1+2s)}2^{k_3^*/2}\norm{v_{k_1}}{\lebL^\infty_T\lebL^2}\norm{v_{k_2}}{\lebL^\infty_T\lebL^2}\norm{u_{2,k_3}}{\lebL^\infty_T\lebL^2}
	\end{align*}
	and \eqref{eq_est_emb_spfc}.
\end{proof}
In order to estimate the remaining frequency interactions, we need another analogue of \eqref{eq_ibp_uuu}, which is given by
\begin{align}\label{eq_ibp_vvu2}
	\begin{split}
		&-\Omega(\xi_1,\xi_2)\ftxh{v}(\xi_1)\ftxh{v}(\xi_2)\ftxh{u}_2(\xi_3)
		\\=&\,
		\dt\left[\ftxh{v}(\xi_1)\ftxh{v}(\xi_2)\ftxh{u}_2(\xi_3)\right]
		+
		i\xi_1\widehat{vw}(\xi_1)\hat{v}(\xi_2)\hat{u}_2(\xi_3)
		+
		i\xi_2\hat{v}(\xi_1)\widehat{vw}(\xi_2)\hat{u}_2(\xi_3)
		+
		i\xi_3\hat{v}(\xi_1)\hat{v}(\xi_2)\widehat{u_2u_2}(\xi_3).
	\end{split}
\end{align}

\begin{lem}\label{lem_hs_hh}
	Let $s>3/2-\alpha$. Then, there exist $c,d>0$ such that we have
	\begin{align*}
		\sum_{k_1\sim k_3>n}\abs{I\!I\!I(k_1,k_3)}
		&\lesssim
		2^{-nc}\norm{v}{\spfc^{-1/2}_T}\norm{v}{\spfc^{s}_T}\norm{u_2}{\spfc^{s+2-\alpha}_T}
		\\
		&+
		T^d\norm{v}{\spfc^{-1/2}_T}\norm{v}{\spfc^{s}_T}\norm{u_2}{\spfc^{s+2-\alpha}_T}\left[\norm{w}{\spfc^s_T}+\norm{u_2}{\spfc^s_T}\right].
	\end{align*}
\end{lem}
\begin{proof}
	Applying \eqref{eq_ibp_vvu2} to each summand $I\!I\!I(k_1,k_3)$, we obtain
	\begin{align*}
		\abs{I\!I\!I(k_1,k_3)}
		\leq
		\abs{B\!B\!B(k_1,k_3)}
		+
		\abs{I\!I\!I_1(k_1,k_3)}
		+
		\abs{I\!I\!I_2(k_1,k_3)}
		+
		\abs{I\!I\!I_3(k_1,k_3)},
	\end{align*}
	where the terms on the right-hand side are given by
	\begin{align*}
		B\!B\!B(k_1,k_3)
		&:=
		2^{k_12s}
		\sup_{t}
		\sum_{k_2}
		\left[ \sum_{\xi_{123}=0}
		\frac{(-i\xi_1)\chi_{k_1}^2(\xi_1)\chi_{k_2}(\xi_2)\chi_{k_3}(\xi_3)}{\Omega(\xi_1,\xi_2,\xi_3)}\hat{v}(\xi_1)\hat{v}(\xi_2)\hat{u}_2(\xi_3)\right]_{0}^{t},
		\\
		I\!I\!I_1(k_1,k_3)
		&:=
		2^{k_12s}\sup_t\sum_{k_2}\int_0^t
		\sum_{\xi_{ab23}=0}
		\frac{(-i\xi_1)^2\chi_{k_1}^2(\xi_{ab})\chi_{k_2}(\xi_2)\chi_{k_3}(\xi_3)}{\Omega(\xi_{ab},\xi_2,\xi_3)}\ftxh{v}(\xi_a)\ftxh{w}(\xi_b)\ftxh{v}(\xi_2)\ftxh{u}_2(\xi_3),
		\\
		I\!I\!I_2(k_1,k_3)
		&:=
		2^{k_12s}\sup_t\sum_{k_2}\int_0^t
		\sum_{\xi_{1ab3}=0}
		\frac{(-i\xi_1)(-i\xi_2)\chi_{k_1}^2(\xi_1)\chi_{k_2}(\xi_{ab})\chi_{k_3}(\xi_3)}{\Omega(\xi_1,\xi_{ab},\xi_3)}\ftxh{v}(\xi_1)\ftxh{v}(\xi_a)\ftxh{w}(\xi_b)\ftxh{u}_2(\xi_3),
		\\
		I\!I\!I_3(k_1,k_3)
		&:=
		2^{k_12s}\sup_t\sum_{k_2}\int_0^t
		\sum_{\xi_{12ab}=0}
		\frac{(-i\xi_1)(-i\xi_3)\chi_{k_1}^2(\xi_1)\chi_{k_2}(\xi_2)\chi_{k_3}(\xi_{ab})}{\Omega(\xi_1,\xi_2,\xi_{ab})}\ftxh{v}(\xi_1)\ftxh{v}(\xi_2)\ftxh{u}_2(\xi_a)\ftxh{u}_2(\xi_b).
	\end{align*}
	Similar to the boundary term $B$ in Lemma \ref{lem_est_ert_hl}, the term $B\!B\!B$ can be estimated by
	\begin{align*}
		\abs{B\!B\!B(k_1,k_3)}
		\lesssim
		\sum_{k_2} 2^{k_1(2s+1-\alpha)}2^{-k_2/2}\norm{v_{k_1}}{\lebL^\infty_T\lebL^2}\norm{v_{k_2}}{\lebL^\infty_T\lebL^2}\norm{u_{2,k_3}}{\lebL^\infty_T\lebL^2}.
	\end{align*}	
	An application of Lemma \ref{lem_i4} in combination with the bounds ${\abs{\Omega^{-1}(\xi_{ab},\xi_2,\xi_3)(-i\xi_1)^2}\lesssim2^{k_1(2-\alpha)}2^{-k_2}}$, ${\abs{\Omega^{-1}(\xi_1,\xi_{ab},\xi_3)(-i\xi_1)(-i\xi_2)}\lesssim 2^{k_1(1-\alpha)}}$ and $\abs{\Omega^{-1}(\xi_1,\xi_2,\xi_{ab})(-i\xi_1)(-i\xi_3)}\lesssim 2^{k_1(2-\alpha)}2^{-k_2}$ leads to
	\begin{align*}
		\abs{I\!I\!I_1(k_1,k_3)}
		\lesssim
		T^d\!\!\!\!
		\sum_{k_a,k_b,k_2}\!\!\!\!
		\norm{v_{k_a}}{\spfb^{k_a}_T}\!
		\norm{w_{k_b}}{\spfb^{k_b}_T}\!
		\norm{v_{k_2}}{\spfb^{k_2}_T}\!
		\norm{u_{2,k_3}}{\spfb^{k_3}_T}\!
		\times\!
		\begin{cases}
			2^{k_1(2s+5/2-\alpha)}2^{-k_2/2}
			&\text{if }k_a\!\sim\! k_b\!\gtrsim\! k_1,
			\\
			2^{k_1(2s+2-\alpha)}2^{-k_2/2}2^{k_b/2}
			&\text{if }k_a\!\sim\! k_1\!\gg\! k_b,
			\\
			2^{k_1(2s+2-\alpha)}2^{-k_2/2}2^{k_a/2}
			&\text{if }k_b\!\sim\! k_1\!\gg\! k_a,
		\end{cases}
	\end{align*}
	and
	\begin{align*}
		\abs{I\!I\!I_2(k_1,k_3)}
		\lesssim
		T^d\!\!\!\!
		\sum_{k_2,k_a,k_b}\!\!\!\!
		\norm{v_{k_1}}{\spfb^{k_1}_T}\!
		\norm{v_{k_a}}{\spfb^{k_a}_T}\!
		\norm{w_{k_b}}{\spfb^{k_b}_T}\!
		\norm{u_{2,k_3}}{\spfb^{k_3}_T}\!
		\times\!
		\begin{cases}
			2^{k_1(2s+2-\alpha)}
			&\text{if }k_a\!\sim\! k_b\!\gtrsim\! k_1,
			\\
			2^{k_1(2s+1-\alpha)}2^{k_a/2}2^{k_b/2}
			&\text{else},
		\end{cases}
	\end{align*}
	as well as
	\begin{align*}
		\abs{I\!I\!I_3(k_1,k_3)}
		\lesssim
		T^d\!\!\!\!
		\sum_{k_2,k_a,k_b}\!\!\!\!
		\norm{v_{k_1}}{\spfb^{k_1}_T}\!
		\norm{v_{k_2}}{\spfb^{k_2}_T}\!
		\norm{u_{2,k_a}}{\spfb^{k_a}_T}\!
		\norm{u_{2,k_b}}{\spfb^{k_b}_T}\!
		\times\!
		\begin{cases}
			2^{k_1(2s+5/2-\alpha)}2^{-k_2/2}
			&\text{if }k_a\!\sim\! k_b\!\gtrsim\! k_1,
			\\
			2^{k_1(2s+2-\alpha)}2^{k_b/2}2^{-k_2/2}
			&\text{if }k_a\!\sim\! k_1\!\gg\! k_b,
			\\
			2^{k_1(2s+2-\alpha)}2^{k_a/2}2^{-k_2/2}
			&\text{if }k_b\!\sim\! k_1\!\gg\! k_a.
		\end{cases}
	\end{align*}
	After summation in $k_1$ and $k_3$, the proof is completed.
\end{proof}
\begin{lem}\label{lem_hs_lh}
	Let $s>3/2-\alpha$. Then, there exist $c,d>0$ such that we have
	\begin{align*}
		\sum_{k_1\ll k_3>n}\!\!\!
		\abs{I\!I\!I(k_1,k_3)}
		\lesssim
		2^{-nc}\!
		\norm{v}{\spfc^{-1/2}_T}\norm{v}{\spfc^{s}_T}\norm{u_2}{\spfc^{s+2-\alpha}_T}
		+\!
		T^d\norm{v}{\spfc^{-1/2}_T}\norm{v}{\spfc^{s}_T}\norm{u_2}{\spfc^{s+2-\alpha}_T}\left[\norm{w}{\spfc^s_T}+\norm{u_2}{\spfc^s_T}\right]\!.
	\end{align*}
\end{lem}
\begin{proof}
	Since we still have one high- and one low-frequency factor $v$ (see \eqref{eq_Y_bound}), we can repeat the proof of the previous lemma.
\end{proof}
\begin{lem}\label{lem_hs_hl}
	Let $s>3/2-\alpha$. Then, there exist $c,d>0$ such that we have
	\begin{align*}
		\sum_{k_3\ll k_1>n}\abs{I\!I\!I(k_1,k_3)}
		&\lesssim
		2^{-nc}\norm{v}{\spfc^s_T}^2\norm{u_2}{\spfc^s_T}
		+
		T^d\norm{v}{\spfc^s_T}^2\norm{u_2}{\spfc^s_T}^2
		+
		T^d\norm{v}{\spfc^s_T}^2\norm{w}{\spfc^s_T}\norm{u_2}{\spfc^s_T}
		\\
		&+
		T^d\norm{v}{\spfc^{-1/2}_T}\norm{v}{\spfc^s_T}\norm{u_2}{\spfc^s_T}\norm{u_2}{\spfc^{s+2-\alpha}_T}
		+
		T^d\norm{v}{\spfc^{-1/2}_T}\norm{v}{\spfc^s_T}\norm{w}{\spfc^s_T}\norm{u_2}{\spfc^{s+2-\alpha}_T}.
	\end{align*}
\end{lem}
\begin{proof}
	Recall that we have
	\begin{align*}
		I\!I\!I(k_1,k_3)
		=
		2^{k_12s}
		\sup_t\int_0^t\int_\bbt P_{k_1}^2(\dx v) vP_{k_3}u_2.
	\end{align*}
	The spatial integral can be written as follows:
	\begin{align*}
		\int_\bbt  P_{k_1}^2(\dx v)vP_{k_3}u_2
		&=\!
		-\!
		\int_\bbt\! P_{k_1}^2vvP_{k_3}\dx u_2
		-
		\int_\bbt\! P_{k_1}^2v\dx v P_{k_3}u_2
		\\&=\!
		-\!
		\underset{=:T_1}{\underbrace{\int_\bbt\! P_{k_1}^2v vP_{k_3}\dx u_2}}
		+
		\underset{=:T_2}{\underbrace{\frac{1}{2}\!\int_\bbt\! P_{k_1}v P_{k_1}vP_{k_3}\dx u_2}}
		-\!
		\underset{=:T_3}{\underbrace{\int_\bbt\! P_{k_1}v\left[P_{k_1}(\dx vP_{k_3}u_2)\!-\!P_{k_1}\dx vP_{k_3}u_2\right]}}.
	\end{align*}
	Together with
	\begin{align*}
		I\!I\!I^j(k_1,k_3)
		:=
		2^{-k_1}\sup_t\int_0^t T^j,
		\qquad
		j\in[3],
	\end{align*}
	we conclude
	\begin{align*}
		\abs{I\!I\!I(k_1,k_3)}
		\leq
		\abs{I\!I\!I^1(k_1,k_3)}
		+
		\abs{I\!I\!I^2(k_1,k_3)}
		+
		\abs{I\!I\!I^3(k_1,k_3)}.
	\end{align*}
	Next, we apply \eqref{eq_ibp_vvu2} to each $I\!I\!I^j(k_1,k_3)$, $j\in[2]$, and obtain 
	\begin{align*}
		\abs{I\!I\!I^j(k_1,k_3)}
		\leq
		\abs{B\!B\!B^j(k_1,k_3)}+\abs{I\!I\!I^j_1(k_1,k_3)}+\abs{I\!I\!I^j_2(k_1,k_3)}+\abs{I\!I\!I^j_3(k_1,k_3)},
	\end{align*}
	where the terms on the right-hand side for $j=1$ are given by
	\begin{align*}
		B\!B\!B^1(k_1,k_3)
		&:=
		2^{k_12s}
		\sup_{t}\sum_{k_2}
		\left[ 	\sum_{\xi_{123}=0}\frac{-i\xi_3\chi_{k_1}^2(\xi_1)\chi_{k_2}(\xi_2)\chi_{k_3}(\xi_3)}{\Omega(\xi_1,\xi_2,\xi_3)}\hat{v}(\xi_1)\hat{v}(\xi_2)\hat{u}_2(\xi_3)\right]_{0}^{t},
		\\
		I\!I\!I^1_1(k_1,k_3)
		&:=
		2^{k_12s}\sup_t\sum_{k_2}\int_0^t
		\sum_{\xi_{ab23}=0}
		\frac{(-i\xi_{ab})(-i\xi_3)\chi_{k_1}^2(\xi_{ab})\chi_{k_2}(\xi_2)\chi_{k_3}(\xi_3)}{\Omega(\xi_{ab},\xi_2,\xi_3)}\hat{v}(\xi_a)\hat{w}(\xi_b)\hat{v}(\xi_2)\hat{u}_2(\xi_3),
		\\
		I\!I\!I^1_2(k_1,k_3)
		&:=
		2^{k_12s}\sup_t\sum_{k_2}\int_0^t
		\sum_{\xi_{1ab3}=0}
		\frac{(-i\xi_{ab})(-i\xi_3)\chi_{k_1}^2(\xi_1)\chi_{k_2}(\xi_{ab})\chi_{k_3}(\xi_3)}{\Omega(\xi_1,\xi_{ab},\xi_3)}\hat{v}(\xi_1)\hat{v}(\xi_a)\hat{w}(\xi_b)\hat{u}_2(\xi_3),
		\\
		I\!I\!I^1_3(k_1,k_3)
		&:=
		2^{k_12s}\sup_t\sum_{k_2}\int_0^t
		\sum_{\xi_{12ab}=0}
		\frac{(-i\xi_{ab})^2\chi_{k_1}^2(\xi_1)\chi_{k_2}(\xi_2)\chi_{k_3}(\xi_{ab})}{\Omega(\xi_1,\xi_2,\xi_{ab})}\hat{v}(\xi_1)\hat{v}(\xi_2)\ftxh{u}_2(\xi_{a})\ftxh{u}_2(\xi_{b}).
	\end{align*}
	As before, the terms for $j=2$ follow from obvious modifications.
	We proceed by bounding the terms $I\!I\!I^1$, $I\!I\!I^2$ and $I\!I\!I^3$ separately.\\
	
	\textbf{Estimating $I\!I\!I^1$.} Note that $k_1\sim k_2$ holds.
	As in Lemma \ref{lem_est_ert_hl}, the boundary term admits the estimate
	\begin{align*}
		\abs{B\!B\!B^1(k_1,k_3)}
		\lesssim
		\sum_{k_2}2^{k_2(2s-\alpha)}2^{k_3/2}\norm{v_{k_1}}{\lebL^\infty_T\lebL^2}\norm{v_{k_2}}{\lebL^\infty_T\lebL^2}\norm{u_{2,k_3}}{\lebL^\infty_T\lebL^2}.
	\end{align*}
	Let us continue with $I\!I\!I_3^1$. An application of Lemma \ref{lem_i4} together with ${\abs{\Omega^{-1}(\xi_{ab},\xi_2,\xi_3)(-i\xi_{ab})^2}\lesssim 2^{-k_1\alpha}2^{k_3}}$ yields
	\begin{align*}
		\abs{I\!I\!I_3^1(k_1,k_3)}
		\leq
		T^d
		\sum_{k_2}
		\norm{v_{k_1}}{\spfb^{k_1}_T}
		\norm{v_{k_2}}{\spfb^{k_2}_T}
		\norm{u_{2,k_a}}{\spfb^{k_a}_T}
		\norm{u_{2,k_b}}{\spfb^{k_b}_T}
		\times
		\begin{cases}
			2^{k_1(2s+2-\alpha)}
			&\text{if }k_a\sim k_b\gtrsim k_1,
			\\
			2^{k_1(2s+3/2-\alpha)}2^{k_b/2}
			&\text{if }k_a\sim k_1\gg k_b,
			\\
			2^{k_1(2s+3/2-\alpha)}2^{k_a/2}
			&\text{if }k_b\sim k_1\gg k_a.
		\end{cases}
	\end{align*}
	Similarly, using $\abs{\Omega^{-1}(\xi_{ab},\xi_2,\xi_3)(-i\xi_{ab})(-i\xi_3)}\lesssim 2^{k_1(1-\alpha)}$, we bound $I\!I\!I^1_1(k_1,k_3)$ in the case ${k_a\sim k_b\gtrsim k_1}$ by
	\begin{align*}
		\abs{I\!I\!I^1_1(k_1,k_3)}_{k_a\sim k_b\gtrsim k_1}
		\lesssim
		T^d
		\sum_{k_2, k_a, k_b}
		2^{k_1(2s+3/2-\alpha)}2^{k_3/2}
		\norm{v_{k_a}}{\spfb^{k_a}_T}
		\norm{w_{k_b}}{\spfb^{k_b}_T}
		\norm{v_{k_2}}{\spfb^{k_2}_T}
		\norm{u_{2,k_3}}{\spfb^{k_3}_T}.
	\end{align*}
	To treat the case $k_a\sim k_1\gg k_b$, we proceed as in Lemma \ref{lem_h12_HL}. The modulus of $I\!I\!I^1_1(k_1,k_3)$ is bounded by
	\begin{align*}
		2^{k_12s}\sum_{k_a\sim k_2\gg k_b}\abs{\int_0^t\sum_{\xi_{ab23}=0}m(\xi_a,\xi_b,\xi_2,\xi_3)\underset{=:r(\xi_a,\xi_b,\xi_2,\xi_3)}{\underbrace{\hat{v}(\xi_a)\hat{w}(\xi_b)\hat{v}(\xi_2)\hat{u}_2(\xi_3)}}},
	\end{align*}
	where $m$ is defined in \eqref{eq_def_m}. Using the symmetry of $r$ in its first and third variable, we write
	\begin{align*}
		2\sum_{\xi_{ab23}=0} m(\xi_a,\xi_b,\xi_2,\xi_3)r(\xi_a,\xi_b,\xi_2,\xi_3)
		=
		\sum_{\xi_{ab23}=0} \left[m(\xi_a,\xi_b,\xi_2,\xi_3)+m(\xi_2,\xi_b,\xi_a,\xi_3)\right]r(\xi_a,\xi_b,\xi_2,\xi_3).
	\end{align*}
	Moreover, recall that 
	\begin{align*}
		\abs{m(\xi_a,\xi_b,\xi_2,\xi_3)+m(\xi_2,\xi_b,\xi_a,\xi_3)}
		\lesssim
		\sum_{j\in[5]}\abs{m_j}
		\lesssim
		2^{-k_1\alpha}2^{\max\{k_b,k_3\}}
	\end{align*}
	holds. Combining the last observations with Lemma \ref{lem_i4}, we arrive at
	\begin{align*}
		\abs{I\!I\!I^1_1(k_1,k_3)}_{k_a\sim k_1\gg k_b}
		\lesssim
		T^d
		\sum_{k_2, k_a, k_b}
		2^{k_2(2s-\alpha)}2^{\max\{k_3,k_b\}}2^{k_3/2}2^{k_b/2}
		\norm{v_{k_a}}{\spfb^{k_a}_T}
		\norm{w_{k_b}}{\spfb^{k_b}_T}
		\norm{v_{k_2}}{\spfb^{k_2}_T}
		\norm{u_{2,k_3}}{\spfb^{k_3}_T}.
	\end{align*}
	In the case $k_b\sim k_1\gg k_a$, $r$ is not symmetric in its second and third variable. Hence, we only obtain
	\begin{align}\label{eq_split_m}
		\begin{split}
			2\sum_{\xi_{ab23}=0} m(\xi_a,\xi_b,\xi_2,\xi_3)r(\xi_a,\xi_b,\xi_2,\xi_3)
			&=
			\sum_{\xi_{ab23}=0} 	\left[m(\xi_a,\xi_b,\xi_2,\xi_3)+m(\xi_a,\xi_2,\xi_b,\xi_3)\right]r(\xi_a,\xi_b,\xi_2,\xi_3)
			\\&+
			\sum_{\xi_{ab23}=0}m(\xi_a,\xi_2,\xi_b,\xi_3) \left[r(\xi_a,\xi_2,\xi_b,\xi_3)-r(\xi_a,\xi_b,\xi_2,\xi_3)\right].
		\end{split}
	\end{align}
	Similar to Lemma \ref{lem_h12_HL}, we can write $\left[m(\xi_a,\xi_b,\xi_2,\xi_3)+m(\xi_a,\xi_2,\xi_b,\xi_3)\right]=\sum_{j\in[5]} m'_j$ for some appropriately chosen summands $m'_j$ with modulus bounded by $\smash{2^{\max\{k_3,k_a\}}2^{-k_2\alpha}}$. This allows us to estimate the first sum on the right-hand side of \eqref{eq_split_m} as before. For the second sum, observe that
	\begin{align*}
		\left[r(\xi_a,\xi_2,\xi_b,\xi_3)-r(\xi_a,\xi_b,\xi_2,\xi_3)\right]
		&=
		(-i\xi_3)\chi_{k_3}(\xi_3)\chi_{k_b}(\xi_b)\hat{v}(\xi_a)\hat{w}(\xi_3)\left[\hat{w}(\xi_2)\hat{v}(\xi_b)-\hat{w}(\xi_b)\hat{v}(\xi_2)\right]
		\\&=
		2(-i\xi_3)\chi_{k_3}(\xi_3)\chi_{k_b}(\xi_b)\hat{v}(\xi_a)\hat{w}(\xi_3)\left[\hat{u}_2(\xi_2)\hat{v}(\xi_b)-\hat{u}_2(\xi_b)\hat{v}(\xi_2)\right]
	\end{align*}
	holds, which is advantageous since a factor $u_2$ appears localized to a high-frequency variable. Using the bound $\abs{m(\xi_a,\xi_2,\xi_b,\xi_3)}\lesssim 2^{k_1(1-\alpha)}$, Lemma \ref{lem_i4} implies
	\begin{align*}
		\abs{I\!I\!I^1_1(k_1,k_3)}_{k_b\sim k_1\gg k_a}
		&\lesssim
		T^d
		\sum_{k_2,k_a,k_b}
		2^{k_2(2s-\alpha)}2^{k_3^*3/2}2^{k_4^*/2}
		\norm{v_{k_a}}{\spfb^{k_a}_T}
		\norm{w_{k_b}}{\spfb^{k_b}_T}
		\norm{v_{k_2}}{\spfb^{k_2}_T}
		\norm{u_{2,k_3}}{\spfb^{k_3}_T}
		\\
		&+
		T^d
		\sum_{k_2,k_a,k_b}
		2^{k_2(2s+1-\alpha)}2^{k_a/2}2^{k_3/2}
		\norm{v_{k_a}}{\spfb^{k_a}_T}\norm{u_{2,k_3}}{\spfb^{k_3}_T}
		\\
		&\qquad\times
		\left[
		\norm{u_{2,k_b}}{\spfb^{k_b}_T}
		\norm{v_{k_2}}{\spfb^{k_2}_T}
		+
		\norm{v_{k_b}}{\spfb^{k_b}_T}
		\norm{u_{2,k_2}}{\spfb^{k_2}_T}
		\right].
	\end{align*}
	We estimate $I\!I\!I^1_2$ by the same arguments.\\

	\textbf{Estimating $I\!I\!I^2$.}
	As in Lemma \ref{lem_h12_HL}, we can repeat the arguments used for the estimation of $I\!I\!I_1$ by replacing every $\chi_{k_1}^2$ and $\chi_{k_2}$ by $\chi_{k_1}$.\\
	
	\textbf{Estimating $I\!I\!I^3$.}
	Recall that we have
	\begin{align*}
		I\!I\!I^3(k_1,k_3)
		&=
		2^{k_12s}\int_0^T
		\int_\bbt P_{k_1}v\left[P_{k_1}(\dx vP_{k_3}u_2)-P_{k_1}\dx vP_{k_3}u_2\right]
		\\&=
		2^{k_12s}\int_0^T
		\sum_{\xi_{123}=0}\underset{=:\nu(\xi_1,\xi_2,\xi_3)}{\underbrace{\xi_2\chi_{k_1}(\xi_1)\left[\chi_{k_1}(\xi_{23})-\chi_{k_1}(\xi_2)\right]\chi_{k_3}(\xi_3)}}\hat{v}(\xi_1)\hat{v}(\xi_2)\hat{u}_2(\xi_3).
	\end{align*}
	Together with \eqref{eq_ibp_vvu2}, it follows
	\begin{align*}
		\abs{I\!I\!I^3(k_1,k_3)}
		\leq
		\abs{B\!B\!B^3(k_1,k_3)}
		+
		\abs{I\!I\!I^3_1(k_1,k_3)}
		+
		\abs{I\!I\!I^3_2(k_1,k_3)}
		+
		\abs{I\!I\!I^{3,*}_{12}(k_1,k_3)}
		+
		\abs{I\!I\!I^{3,\#}_3(k_1,k_3)},
	\end{align*}
	where the terms on the right-hand side are given by
	\begin{align*}
		B\!B\!B^3(k_1,k_3)
		&:=
		2^{k_12s}
		\sup_t\left[\sum_{\xi_{123}=0}\frac{\nu}{\Omega}(\xi_1,\xi_2,\xi_3)\hat{v}(\xi_1)\hat{v}(\xi_2)\hat{u}_2(\xi_3)\right]_0^t,
		\\
		I\!I\!I^{3,\times}_{12}(k_1,k_3)
		&:=
		2^{k_12s}\sup_t\int_0^t\sum_{\substack{\xi_{ab23}=0\\ \xi_a\sim\xi_b}}
		\left[\xi_{ab}\frac{\nu}{\Omega}(\xi_{ab},\xi_2,\xi_3)+\xi_{ab}\frac{\nu}{\Omega}(\xi_2,\xi_{ab},\xi_3)\right]
		\hat{v}(\xi_a)\hat{w}(\xi_b)\hat{v}(\xi_2)\hat{u}_2(\xi_3),
		\\
		I\!I\!I^{3,*}_{12}(k_1,k_3)
		&:=
		2^{k_12s}\sup_t\int_0^t\sum_{\substack{\xi_{ab23}=0\\ \xi_a\sim\xi_2\gg\xi_b}}
		\left[\xi_{ab}\frac{\nu}{\Omega}(\xi_{ab},\xi_2,\xi_3)+\xi_{2b}\frac{\nu}{\Omega}(\xi_a,\xi_{2b},\xi_3)\right]
		\hat{v}(\xi_a)\hat{w}(\xi_b)\hat{v}(\xi_2)\hat{u}_2(\xi_3),
		\\
		I\!I\!I^{3,\#}_{12}(k_1,k_3)
		&:=
		2^{k_12s}\sup_t\int_0^t\sum_{\substack{\xi_{ab23}=0\\ \xi_b\sim\xi_2\gg\xi_a}}
		\left[\xi_{ab}\frac{\nu}{\Omega}(\xi_{ab},\xi_2,\xi_3)+\xi_{ab}\frac{\nu}{\Omega}(\xi_2,\xi_{ab},\xi_3)\right]
		\hat{v}(\xi_a)\hat{w}(\xi_b)\hat{v}(\xi_2)\hat{u}_2(\xi_3),
		\\
		I\!I\!I^{3}_3(k_1,k_3)
		&:=
		2^{k_12s}\sup_t\int_0^t
		\sum_{\xi_{12ab}=0}
		(-i\xi_{ab})\frac{\nu}{\Omega}(\xi_1,\xi_2,\xi_{ab})\hat{v}(\xi_1)\hat{v}(\xi_2)\ftxh{u}_2(\xi_{a})\ftxh{u}_2(\xi_{b}).
	\end{align*}
	Note that the variable $\xi_2$ is already of size $2^{k_1}$ due to the localization in $\nu(\xi_1,\xi_2,\xi_3)$. As before, we still localize each $\xi_2$ to dyadic frequency ranges by inserting $1=\sum_{k_2}\chi_{k_2}(\xi_2)$ to improve notation. We also localize $\xi_a$ and $\xi_b$.
	
	Similar to Lemma \ref{lem_est_ert_hl}, the boundary term $B\!B\!B^3$ can be estimated by
	\begin{align*}
		\abs{B\!B\!B^3(k_1,k_3)}
		\lesssim
		\sum_{k_2}2^{k_2(2s-\alpha)}2^{k_3/2}\norm{v_{k_1}}{\lebL^\infty_T\lebL^2}\norm{v_{k_2}}{\lebL^\infty_T\lebL^2}\norm{u_{2,k_3}}{\lebL^\infty_T\lebL^2}.
	\end{align*}
	To bound $I\!I\!I_3^3$, we use $\abs{\xi_{ab}\frac{\nu}{\Omega}(\xi_1,\xi_2,\xi_{ab})}\lesssim 2^{-k_1\alpha}2^{k_3}$ and apply Lemma \ref{lem_i4} leading to
	\begin{align*}
		\abs{I\!I\!I_3^3(k_1,k_3)}
		\lesssim
		T^d\!\!\!
		\sum_{k_2, k_a, k_b}\!\!
		\norm{v_{k_1}}{\spfb^{k_1}_T}\!
		\norm{v_{k_2}}{\spfb^{k_2}_T}\!
		\norm{u_{2,k_a}}{\spfb^{k_a}_T}\!
		\norm{u_{2,k_b}}{\spfb^{k_b}_T}\!
		\times\!
		\begin{cases}
			2^{k_1(2s-\alpha)}2^{k_3^*/2}2^{k_4^*/2}
			&\text{if }k_a\!\sim\! k_b\!\gtrsim\! k_3,
			\\
			2^{k_1(2s-\alpha)}2^{k_a3/2}2^{k_b/2}
			&\text{if }k_3\!\sim\! k_a\!\gg\! k_b,
			\\
			2^{k_1(2s-\alpha)}2^{k_b3/2}2^{k_a/2}
			&\text{if }k_3\!\sim\! k_b\!\gg\! k_a.
		\end{cases}
	\end{align*}
	Next, we estimate $I\!I\!I_{12}^{3,\times}$ with the help of the bound $\abs{\xi_{ab}\frac{\nu}{\Omega}(\xi_{ab},\xi_2,\xi_3)+\xi_{ab}\frac{\nu}{\Omega}(\xi_{ab},\xi_2,\xi_3)}\lesssim 2^{k_1(1-\alpha)}$. We conclude
	\begin{align*}
		\abs{I\!I\!I_{12}^{3,\times}(k_1,k_3)}
		\lesssim
		T^d
		\sum_{k_2,k_a,k_b}
		2^{k_1(2s+2-\alpha)}
		\norm{v_{k_a}}{\spfb^{k_a}_T}\norm{w_{k_b}}{\spfb^{k_b}_T}\norm{v_{k_2}}{\spfb^{k_2}_T}\norm{u_{2,k_3}}{\spfb^{k_3}_T}.
	\end{align*}
	To handle $I\!I\!I_{12}^{3,*}$, we recall from the proof of Lemma \ref{lem_h12_HL} that
	\begin{align*}
		\abs{\xi_{ab}\frac{\nu}{\Omega}(\xi_{ab},\xi_2,\xi_3)+\xi_{2b}\frac{\nu}{\Omega}(\xi_a,\xi_{2b},\xi_3)}
		\lesssim
		2^{-k_2\alpha}2^{\max\{k_3,k_b\}}
	\end{align*}
	holds.
	Hence, an application of Lemma \ref{lem_i4} yields
	\begin{align*}
		\abs{I\!I\!I_{12}^{3,*}(k_1,k_3)}
		\lesssim
		T^d
		\sum_{k_a,k_b,k_2}
		2^{k_2(2s-\alpha)}2^{\max\{k_3,k_b\}}2^{k_3/2}2^{k_b/2}
		\norm{v_{k_a}}{\spfb^{k_a}_T}
		\norm{w_{k_b}}{\spfb^{k_b}_T}
		\norm{v_{k_2}}{\spfb^{k_2}_T}
		\norm{u_{2,k_3}}{\spfb^{k_3}_T}.
	\end{align*}
	Now, let us proceed with $I\!I\!I_{12}^{3,\#}(k_1,k_3)$. We write
	\begin{align*}
		I\!I\!I_{12}^{3,\#}(k_1,k_3)
		&=
		2^{k_12s}\sup_t\int_0^t\sum_{\substack{\xi_{ab23}=0\\ \xi_b\sim\xi_2\gg\xi_a}}
		\xi_{ab}\frac{\nu}{\Omega}(\xi_{ab},\xi_2,\xi_3)\left[\hat{v}(\xi_2)\hat{w}(\xi_b)
		-
		\hat{v}(\xi_b)\hat{w}(\xi_2)\right]\hat{v}(\xi_a)\hat{u}_2(\xi_3)
		&=:G
		\\&+
		2^{k_12s}\sup_t\int_0^t\sum_{\substack{\xi_{ab23}=0\\ \xi_b\sim\xi_2\gg\xi_a}}
		\left[\xi_{ab}\frac{\nu}{\Omega}(\xi_{ab},\xi_2,\xi_3)
		+
		\xi_{a2}\frac{\nu}{\Omega}(\xi_{b},\xi_{2a},\xi_3)\right]\hat{v}(\xi_a)\hat{v}(\xi_b)\hat{w}(\xi_2)\hat{u}_2(\xi_3)
		&=:H.
	\end{align*}
	Above, we insert the equation $\hat{v}(\xi_2)\hat{w}(\xi_b)
	-
	\hat{v}(\xi_b)\hat{w}(\xi_2)=2\hat{v}(\xi_b)\hat{u}_2(\xi_2)-2\hat{v}(\xi_2)\hat{u}_2(\xi_b)$, using the bound ${\abs{\xi_{ab}\frac{\nu}{\Omega}(\xi_{ab},\xi_2,\xi_3)}\lesssim2^{k_1(1-\alpha)}}$ and applying Lemma \ref{lem_i4}, we observe
	\begin{align*}
		\abs{G}
		&\lesssim
		T^d
		\sum_{k_2,k_a,k_b}
		2^{k_1(2s+1-\alpha)}2^{k_a/2}2^{k_3/2}
		\norm{v_{k_a}}{\spfb^{k_a}_T}
		\left[\norm{u_{2,k_b}}{\spfb^{k_b}_T}\norm{v_{k_2}}{\spfb^{k_2}_T}
		+
		\norm{v_{k_b}}{\spfb^{k_b}_T}\norm{u_{2,k_2}}{\spfb^{k_2}_T}
		\right]
		\norm{u_{2,k_3}}{\spfb^{k_3}_T}.
	\end{align*}
	To bound $H$, we note that we have
	\begin{align*}
		\xi_{ab}\frac{\nu}{\Omega}(\xi_{ab},\xi_2,\xi_3)
		+
		\xi_{a2}\frac{\nu}{\Omega}(\xi_{b},\xi_{2a},\xi_3)
		&=
		\left[\xi_{ab}+\xi_{a2}\right]\frac{\nu}{\Omega}(\xi_{ab},\xi_2,\xi_3)
		&=:A'_1
		\\&+
		\xi_{a2}\left[\nu(\xi_b,\xi_{2a},\xi_3)-\nu(\xi_{ab},\xi_2,\xi_3)\right]\Omega^{-1}(\xi_{ab},\xi_2,\xi_3)
		&=:A'_2
		\\&+
		\xi_{a2}\nu(\xi_b,\xi_{2a},\xi_3)\left[\Omega^{-1}(\xi_b,\xi_{2a},\xi_3)-\Omega^{-1}(\xi_{ab},\xi_2,\xi_3)\right]
		&=:A'_3.
	\end{align*}
	The modulus of each $A_i'$ can be bounded by $2^{-\alpha k_2}2^{\max\{k_3,k_b\}}$. Hence, together with Lemma \ref{lem_i4}, we conclude
	\begin{align*}
		\abs{H}
		&\lesssim
		T^d
		\sum_{k_2,k_a,k_b}
		2^{k_1(2s-\alpha)}2^{\max\{k_a,k_3\}}2^{k_a/2}2^{k_3/2}
		\norm{v_{k_a}}{\spfb^{k_a}_T}
		\norm{w_{k_b}}{\spfb^{k_b}_T}
		\norm{v_{k_2}}{\spfb^{k_2}_T}
		\norm{u_{2,k_3}}{\spfb^{k_3}_T}.	
	\end{align*}
	After summation of all obtained bounds in $k_1$ and $k_3$, the proof is finished.
\end{proof}

Observe that estimate \eqref{est_est} follows from \eqref{eq_split_x_y}, \eqref{eq_X_bound}, \eqref{eq_Y_bound} and Lemmata \ref{lem_hs_ll}, \ref{lem_hs_hh}, \ref{lem_hs_lh} and \ref{lem_hs_hl}.

\section{Proof of the main theorem}\label{s_proof}

In this section we will use the results of the previous sections in order to prove Theorem \ref{thm_main}. The proof is divided into four steps:
First, we recall a result guaranteeing the existence and uniqueness of smooth solutions of \eqref{eq_PDE} as well as the continuity of the data-to-solution-map. Second, we collect the estimates obtained in the previous sections. This leads to an a priori estimate in $\lebH^s(\bbt)$ for smooth solutions of \eqref{eq_PDE} with mean zero as well as to two estimates for the difference of smooth solutions with mean zero -- one estimate in $\lebH^{-1/2}(\bbt)$ and another in $\lebH^s(\bbt)$.
In the third step, we approximate an initial datum $u_0\in\lebH^s_0(\bbt)$ by smooth initial data $(u_{k,0})_{k\in\bbn}$ and show that the sequence of smooth solutions $(S^\infty_T(u_{k,0}))_{k\in\bbn}$ converges in $C([0,T];\lebH^s(\bbt))$. Denoting the limit by $S^s_T(u_0)$, this yields a continuous extension of $S^\infty_T$ to $\lebH^s_0(\bbt)$.
In the last step, we upgrade this to an extension of $S^\infty_T$ to $\lebH^s(\bbt)$ by using the conservation of the mean along the flow of equation \eqref{eq_PDE}.\\

\textbf{Step 1.}
The following statement is a direct consequence of Theorem $6$ and $7$ proved in \cite{Kat1975}:
\begin{pro}
	Let $R>0$ and $\sigma>3/2$. Then, there exists a positive time $T=T(R)>0$ such that for all ${u_0\in B_R(0)\subset\lebH^s(\bbt)}$ there is a unique solution $S^\sigma_T(u_0)$ of \eqref{eq_PDE} such that
	\begin{align*}
		S^\sigma_T(u_0)\in C([0,T];\lebH^\sigma(\bbt))
	\end{align*} 
	holds. Moreover, the map $S^\sigma_T:B_R(0)\rightarrow C([0,T];\lebH^\sigma(\bbt))$ is continuous.
\end{pro}
In particular, we obtain the existence, uniqueness and continuity of the map
\begin{equation}\label{eq_dts}
	S^\infty_{T}:B_R(0)\subset \lebH^\infty(\bbt)\rightarrow C([0,T];\lebH^\infty(\bbt)).
\end{equation}

\textbf{Step 2:} 
Fix $\alpha\in(0,1)$, $T\in(0,1]$ and $r\geq s$. Let $u_0$, $u_{1,0}$ and $u_{2,0}$ be $\lebH^\infty_0(\bbt)$-functions and denote the corresponding solutions of \eqref{eq_PDE} by $u=S^\infty_T(u_0)$, $u_1=S^\infty_T(u_{1,0})$ and $u_2=S^\infty_T(u_{2,0})$. Set $v=u_1-u_2$ and $w=u_1+u_2$. Then, from Proposition \ref{pro_est_spf}, Lemma \ref{lem_est_spn} and Lemma \ref{lem_est_spe}, we obtain the following set of estimates
\begin{equation}\label{eq_btr_hr}
	\begin{cases}
		\norm{u}{\spfc^r_T}
		&\lesssim
		\norm{u}{\spec^r_T}+\norm{\dx(u^2)}{\spnc^r_T},
		\\
		\norm{\dx(u^2)}{\spnc^r_T}
		&\lesssim
		T^d\norm{u}{\spfc^r_T}^2,
		\\
		\norm{u}{\spec^r_T}^2
		&\lesssim
		\norm{u_0}{\lebH^r}^2
		+
		T2^{2n}\norm{u}{\spfc^r_T}^2\norm{u}{\spfc^s_T}
		+
		2^{-nc}\norm{u}{\spfc^r_T}^2\norm{u}{\spfc^s_T}
		+
		T^{d}\norm{u}{\spfc^r_T}^2\norm{u}{\spfc^s_T}^2.
	\end{cases}
\end{equation}
Moreover, we obtain two sets of estimates for differences of solutions. More precisely, we have
\begin{equation}\label{eq_btr_hz}
	\begin{cases}
		\norm{v}{\spfc^{-1/2}_T}
		&\lesssim
		\norm{v}{\spec^{-1/2}_T}+\norm{\dx(vw)}{\spnc^{-1/2}_T},
		\\
		\norm{\dx (vw)}{\spnc^{-1/2}_T}
		&\lesssim
		T^d\norm{v}{\spfc^{-1/2}_T}\norm{w}{\spfc^s_T},
		\\
		\norm{v}{\spec^{-1/2}_T}^2
		&\lesssim
		\norm{v_0}{\lebH^{-1/2}}^2
		+
		(T2^{2n}+2^{-nc})\norm{v}{\spfc^{-1/2}_T}^2\norm{w}{\spfc^s_T}
		\\&+\,
		T^d\norm{v}{\spfc^{-1/2}_T}^2
		\left(
		\norm{w}{\spfc^s_T}^2+\norm{u_1}{\spfc^s_T}^2+\norm{u_2}{\spfc^s_T}^2\right),
	\end{cases}
\end{equation}
as well as
\begin{equation}\label{eq_btr_hs}
	\begin{cases}
		\norm{v}{\spfc^s_T}
		&\lesssim
		\norm{v}{\spec^s_T}+\norm{\dx(vw)}{\spnc^s_T},
		\\
		\norm{\dx (vw)}{\spnc^s_T}
		&\lesssim
		T^d\norm{v}{\spfc^s_T}\norm{w}{\spfc^s_T},
		\\
		\norm{v}{\spec^s_T}^2
		&\lesssim
		\norm{v_0}{\lebH^s}^2
		+
		2^{-nc}\norm{v}{\spfc^{-1/2}_T}\norm{v}{\spfc^s_T}\norm{u_2}{\spfc^{s+2-\alpha}_T}
		\\&
		+\,
		(T2^{2n}+2^{-nc})\left(\norm{v}{\spfc^s_T}^3
		+
		\norm{v}{\spfc^s_T}^2\norm{u_2}{\spfc^s_T}\right)
		\\&+\,
		T^d\big(\norm{v}{\spfc^s_T}^4
		+
		\norm{v}{\spfc^s_T}^3\norm{u_2}{\spfc^s_T}
		+
		\norm{v}{\spfc^s_T}^2\norm{u_2}{\spfc^s_T}\norm{w}{\spfc^s_T}
		+
		\norm{v}{\spfc^s_T}^2\norm{u_2}{\spfc^s_T}^2
		\\&
		+
		\norm{v}{\spfc^{-1/2}_T}\norm{v}{\spfc^s_T}\norm{u_2}{\spfc^s_T}\norm{u_2}{\spfc^{s+2-\alpha}_T}
		+
		\norm{v}{\spfc^{-1/2}_T}\norm{v}{\spfc^s_T}^2\norm{u_2}{\spfc^{s+2-\alpha}_T}\big).
	\end{cases}
\end{equation}
According to Proposition \ref{pro_est_spf}, all quantities above are finite, which allows to bootstrap each set of estimates. Doing so, we obtain a common time $T=T(\norm{u_0}{\lebH^s})>0$, for which we have an a priori estimate in $\lebH^r_0(\bbt)$ given by
\begin{equation}\label{est_apriori_hr}
	\norm{u}{\lebL^\infty_T\lebH^r}
	\lesssim
	\norm{u_0}{\lebH^r}
\end{equation}
as well as a difference estimate in $\lebH^{-1/2}_0(\bbt)$ of the form
\begin{equation}\label{est_diff_low}
	\norm{v}{\lebL^\infty_T\lebH^{-1/2}}
	\lesssim
	\norm{v_0}{\lebH^{-1/2}}
\end{equation}
and a difference estimate in $\lebH^s_0(\bbt)$ given by
\begin{equation}\label{est_diff_hs}
	\norm{v}{\lebL^\infty\lebH^s}
	\lesssim
	\norm{v_0}{\lebH^s}
	+
	\left[\norm{v}{\lebL^\infty\lebH^{-1/2}}
	\norm{v}{\lebL^\infty\lebH^s}
	\norm{u_2}{\lebL^\infty\lebH^{s+2-\alpha}}\right]^{1/2}.
\end{equation}
Now, we specify $u_{1,0}=u_0$ and $u_{2,0}=P_{\leq n}u_0$ for $n>0$. In that case, we can improve \eqref{est_diff_hs}. From the inequality \eqref{est_diff_low}, we obtain
\begin{align*}
	\norm{u_1-u_2}{\lebL^\infty\lebH^{-1/2}}
	&\lesssim
	\norm{u_0-P_{\leq n}u_0}{\lebH^{-1/2}}
	=
	\norm{P_{>n}u_0}{\lebH^{-1/2}}
	\lesssim
	2^{n(-s-1/2)}\norm{P_{>n}u_0}{\lebH^s},
\end{align*}
whereas \eqref{est_apriori_hr} yields
\begin{align*}
	\norm{u_2}{\lebL^\infty\lebH^{s+2-\alpha}}
	&\lesssim
	\norm{P_{\leq n}u_0}{\lebH^{s+2-\alpha}}
	\lesssim
	2^{n(2-\alpha)}\norm{P_{\leq n}u_0}{\lebH^s}.
\end{align*}
Combining these estimates with \eqref{est_diff_hs}, we arrive at
\begin{align*}
	\norm{u_1-u_2}{\lebL^\infty\lebH^s}
	\lesssim
	\norm{P_{>n}u_{0}}{\lebH^s}
	+
	\left[2^{n(3/2-\alpha-s)}\norm{P_{> n}u_0}{\lebH^s}
	\norm{u_1-u_2}{\lebL^\infty\lebH^s}
	\norm{P_{\leq n}u_0}{\lebH^s}\right]^{1/2}.
\end{align*}
For $n$ sufficiently large (depending on the implicit constant in the inequality above, the $\lebH^s(\bbt)$-norm of $u_0$ and on $3/2-\alpha-s$), we conclude
\begin{align*}
	\norm{u_1-u_2}{\lebL^\infty\lebH^s}
	\lesssim
	\norm{P_{>n}u_{0}}{\lebH^s}.
\end{align*}
Recalling that we chose $u_1=S^\infty_T(u_0)$ and $u_2=S^\infty_T(P_{\leq n}u_0)$, we arrive at the estimate
\begin{equation}\label{est_diff_solution}
	\norm{S^\infty_T(u_0)-S^\infty_T(P_{\leq n}u_0)}{\lebL^\infty\lebH^s}
	\lesssim
	\norm{P_{> n}u_0}{\lebH^s}
\end{equation}
holding for any $u_0\in\lebH^\infty_0(\bbt)$, every $s>3/2-\alpha$ and all sufficiently large $n$.\\

\textbf{Step 3:} In the third step, we show that $S^\infty_T$ extends continuously to $B_R(0)\subset\lebH^s_0(\bbt)$. Fix $u_0\in\lebH^s_0(\bbt)$ and let $(u_{k,0})_{k\in\bbn}$ be a sequence of smooth functions with mean zero converging to $u_0$ in $\lebH^s(\bbt)$.
For $n,k,l\in\bbn$, we have
\begin{equation}\label{est_Sinf}
	\begin{split}
	\norm{S^\infty_T(u_{k,0})-S^\infty_T(u_{l,0})}{\lebL^\infty\lebH^s}
	&\leq
	\norm{S^\infty_T(u_{k,0})-S^\infty_T(P_{\leq n}u_{k,0})}{\lebL^\infty\lebH^s}
	\\
	&+
	\norm{S^\infty_T(u_{l,0})-S^\infty_T(P_{\leq n}u_{l,0})}{\lebL^\infty\lebH^s}
	\\
	&+
	\norm{S^\infty_T(P_{\leq n}u_{k,0})-S^\infty_T(P_{\leq n}u_{l,0})}{\lebL^\infty\lebH^s}.
	\end{split}
\end{equation}
We will prove that the right-hand side converges to zero as $k$ and $l$ tend to infinity.

Fix $\epsilon>0$. There exists $k_0$ such that for all $k>k_0$ and all $n\in\bbn$ we have
\begin{align*}
	\norm{P_{>n}(u_{k,0}-u_0)}{\lebH^s}
	\leq
	\norm{u_{k,0}-u_0}{\lebH^s}
	\leq
	\epsilon.
\end{align*}
Moreover, we find $n_0$ such that for all $n>n_0$ the inequality $\norm{P_{>n}u_{0}}{\lebH^s}<\epsilon$ holds. Hence, choosing $n>n_0$, we get $\norm{P_{>n}u_{k,0}}{\lebH^s}<2\epsilon$. Applying \eqref{est_diff_solution}, we conclude that for all $k>k_0$ and all $n>n_0$ we have
\begin{equation}\label{est_Sinfty_k}
	\norm{S^\infty_T(u_{k,0})-S^\infty_T(P_{\leq n}u_{k,0})}{\lebL^\infty\lebH^s}
	\lesssim
	2\epsilon.
\end{equation}
Due to
\begin{align*}
	\norm{P_{\leq n}u_{k,0}-P_{\leq n}u_0}{\lebL^\infty\lebH^r}
	\leq
	\begin{cases}
		\norm{P_{\leq n}(u_{k,0}-u_0)}{\lebL^\infty\lebH^s}
		&\text{ if }
		r\leq s,
		\\
		2^{n(r-s)}\norm{P_{\leq n}(u_{k,0}-u_0)}{\lebL^\infty\lebH^s}
		&\text{ if }
		r>s,
	\end{cases}	
\end{align*}
convergence of $(u_{k,0})_{k\in\bbn}$ to $u_0$ in $\lebH^s(\bbt)$ implies convergence of $(P_{\leq n}u_{k,0})_{k\in\bbn}$ to $P_{\leq n}u_{k,0}$ in $\lebH^r(\bbt)$ for any $r\in\bbr$. In particular, this yields convergence of $(P_{\leq n}u_{k,0})_{k\in\bbn}$ to $P_{\leq n}u_{0}$ in $\lebH^\infty(\bbt)$. Thus, the continuity of the map $S^\infty_T$ implies
\begin{align*}
	\norm{S^\infty_T(P_{\leq n}u_{k,0})-S^\infty_T(P_{\leq n}u_0)}{\lebL^\infty\lebH^s}
	\lesssim
	\epsilon
\end{align*}
for sufficiently large $k$ leading to
\begin{equation}\label{est_Sinfty_kl}
	\norm{S^\infty_T(P_{\leq n}u_{k,0})-S^\infty_T(P_{\leq n}u_{l,0})}{\lebL^\infty\lebH^s}
	\lesssim
	2\epsilon.
\end{equation}
From \eqref{est_Sinf}, \eqref{est_Sinfty_k} and \eqref{est_Sinfty_kl}, we deduce
\begin{align*}
	\norm{S^\infty_T(u_{k,0})-S^\infty_T(u_{l,0})}{\lebL^\infty\lebH^s}
	\lesssim
	6\epsilon
\end{align*}
for all sufficiently large $k$ and $l$. Thus, $(S^\infty_T(u_{k,0}))_{k\in\bbn}$ is a Cauchy sequence in $C([0,T];\lebH^s(\bbt))$, which implies the existence of a unique limit denoted by $S^s_T(u_0)\in C([0,T];\lebH^s(\bbt))$. Hence, the map $S^s_T$ extends continuously to $\lebH^s_0(\bbt)$.\\

\textbf{Step 4:} We recall that the mean is conserved along the flow of \eqref{eq_PDE}. In particular, if $u_0\in\lebH^\infty(\bbt)$ and $c\in\bbr$, then the solution with initial datum $u_0+c$ can be written as
\begin{align*}
	S^\infty_T(u_0+c)(t,x)
	=
	S^\infty_T(u_0)(t,x+2ct)
	+
	c.
\end{align*}
Fix $u_0\in\lebH^s(\bbt)$ and let $(u_{k,0})_{k\in\bbn}$ be a sequence converging to $u_0$ in $\lebH^s(\bbt)$. Using the previous equation and the continuous extension of $S^\infty_T$ to $\lebH^s_0(\bbt)$ from Step 3, it is easy to conclude that $(S^\infty_T(u_{k,0}))_{k\in\bbn}$ is a Cauchy sequence in $C([0,T];\lebH^s(\bbt))$. Then, denoting the unique limit of $(S^\infty_T(u_{k,0}))_{k\in\bbn}$ by $S^s_T(u_0)$, we conclude that the map $S^\infty_T$ extends continuously to $\lebH^s(\bbt)$.

\subsection{Improved a priori estimate}\label{ss_apr}
As indicated in the introduction, the a priori estimate \eqref{est_apriori_hr} can be improved for $\alpha\in(1/2,1)$ by some minor modifications of the proofs given in the previous sections. We omit the exact details and just point out which estimates need to be replaced.

From now on, let the parameter $\vari$ in the definitions of $\spfc^s_T$, $\spnc^s_T$ and $\spec^s_T$ be slightly larger than $2-\alpha$. We will show a linear, a nonlinear and an energy estimate as in \eqref{eq_btr_hr}.
The linear estimate \eqref{eq_est_spf} holds for all $r\in\bbr$ and any $\vari>0$, see Proposition \ref{pro_est_spf}.
Next, we claim that the nonlinear estimate holds for all $r>0$. Indeed, if we repeat the proof of Lemma \ref{lem_est_spn}, the case $k_1\sim k_2\gg k_3$ can be improved using $\vari>2-\alpha$. For all remaining cases, we replace each application of \eqref{case_tri_generic} by an application of
\begin{align*}
	\overset{3}{\underset{i=1}{\bigast}} f_i(0,0)
	&\lesssim
	(1+2^{l_3^*-\alpha k_1^*})^{1/2}2^{l_1^*/2}
	\prod_{j=1}^3\norm{f_j}{\lebL^2_\tau\lebl^2_\xi}.
\end{align*}
This estimate holds under the assumptions of Lemma \ref{lem_con}, see Lemma 3.3 in \cite{Sch2020} for a proof.

It remains to improve the energy estimate. Similar to above, we obtain
\begin{align*}
	\overset{4}{\underset{i=1}{\bigast}} f_i(0,0)
	&\lesssim
	(1+2^{l_3^*-\alpha k_1^*})^{1/2}2^{(l_1^*+l_2^*)/2}2^{k_4^*/2}
	\prod_{j=1}^4\norm{f_j}{\lebL^2_\tau\lebl^2_\xi}.
\end{align*}
With this estimate, we can repeat the proof of Lemma \ref{lem_i4} leading to
\begin{align*}
	\abs{S_{\varphi}(u_1,u_2,u_3,u_4)(T)
	}
	\lesssim
	T^d\sum_{k_1,k_2,k_3,k_4}\sup_{\substack{\xi_i\in\supp\chi_{k_i}\\i\in[4]}}\abs{\varphi(\xi_1,\xi_2,\xi_3,\xi_4)}2^{k_1^*(1-\alpha)+}2^{k_4^*/2}
	\prod_{j=1}^4\norm{P_{k_j}u_j}{\spfb^{k_j}_T}.
\end{align*}
In the range $\alpha\in(1/2,1)$, we can repeat the proofs given in Section \ref{ss_for_solution} with the estimate above. As a consequence, we obtain the energy estimate for all $r>\max\{2(1-\alpha),5/4-\alpha\}$.

Bootstrapping these three estimates, we get an improved a priori estimate for smooth solutions $u$ of \eqref{eq_PDE} with mean zero, i.e. \eqref{est_apriori_hr} holds for all $r>\max\{3/2-\alpha,2(1-\alpha),5/4-\alpha\}$.

\section*{Acknowledgements}
I want to thank Sebastian Herr and Robert Schippa for helpful discussions. Funded by the Deutsche Forschungsgemeinschaft (DFG, German Research
Foundation) – IRTG 2235 – Project number 282638148.


\begin{thebibliography}{99}
	
	\bibitem{BS1975}
	J. L. Bona and R. Smith,
	The initial-value problem for the Korteweg-de Vries equation,
	\emph{Philosophical Transactions of the Royal Society of London. Series A. Mathematical and Physical Sciences}
	{\bf 278(1287)} (1975), 555--601.
	
	\bibitem{Bou1993b}
	J. Bourgain,
	Fourier transform restriction phenomena for certain lattice subsets and applications to nonlinear evolution equations. II. The KdV-equation,
	\emph{Geometric and Functional Analysis}
	{\bf 3(3)} (1993), 209--262.
	
	\bibitem{CKS2003a}
	J. Colliander, M. Keel, G. Staffilani, H. Takaoka, T. Tao,
	Sharp global well-posedness for KdV and modified KdV on $\bbr$ and $\bbt$,
	\emph{Journal of the American Mathematical Society}
	{\bf 16(3)} (2003), 705--749.
	
	\bibitem{GKT2022}
	P. Gérard, T. Kappeler, P. Topalov,
	Sharp well-posedness results of the Benjamin-Ono equation in $\lebH^s(\bbt,\bbr)$ and qualitative properties of its solutions,
	\href{https://arxiv.org/abs/2004.04857}{arxiv: 2004.04857} (2020), to appear in \textit{Acta Mathematica}.
	
	\bibitem{Guo2012}
	Z. Guo,
	Local well-posedness for dispersion generalized Benjamin-Ono equations in Sobolev spaces,
	\emph{Journal of Differential Equations}
	{\bf 252(3)} (2012), 2053--2084.
	
	\bibitem{GO2018}
	Z. Guo and T. Oh,
	Non-existence of solutions for the periodic cubic NLS below $L^2$,
	\emph{International Mathematics Research Notices}
	{\bf 6} (2018), 1656--1729.
	
	\bibitem{Her2008}
	S. Herr.
	A note on bilinear estimates and regularity of flow maps for nonlinear dispersive equations,
	\emph{Proceedings of the American Mathematical Society}
	{\bf 136(8)} (2008), 2881--2886.
	
	\bibitem{HIKK2010}
	S. Herr, A. D. Ionescu, C. E. Kenig, H. Koch,
	A para-differential renormalization technique for nonlinear dispersive equations,
	\emph{Communications in Partial Differential Equations}
	{\bf 35(10)} (2010), 1827--1875.
	
	\bibitem{Hur2018}
	V. M. Hur,
	Norm inflation for equations of KdV type with fractional dispersion,
	\emph{Differential and Integral Equations. An International Journal for Theory \& Applications}
	{\bf 31(11-12)} (2018), 833--850.
	
	\bibitem{IKT2008}
	A.D. Ionescu, C.E. Kenig, D. Tataru,
	Global well-posedness of the KP-I initial-value problem in the energy space,
	\emph{Inventiones mathematicae}
	{\bf 173(2)} (2008), 265--304.
	
	\bibitem{IK2007}
	A. D. Ionescu and C. E. Kenig,
	Global well-posedness of the Benjamin-Ono equation
	in low-regularity spaces,
	\emph{Journal of the American Mathematical Society}
	{\bf 20(3)} (2007), 753--798.
	
	\bibitem{KT2006a}
	T. Kappeler and P. Topalov,
	Global wellposedness of KdV in $H^{-1}(\bbt,\bbr)$,
	\emph{Duke Mathematical Journal}
	{\bf 135(2)} (2006), 327--360.
	
	\bibitem{Kat1975}
	T. Kato,
	Quasi-linear equations of evolution, with applications to partial differential equations,
	Spectral Theory and Differential Equations. Proceedings of the Symposium held at Dundee, Scotland, July 1-19, 1974,
	\emph{Lecture Notes in Mathematics}
	{\bf 448} (1975), 25-70.
		
	\bibitem{KPV1996}
	C. E. Kenig, G. Ponce, L. Vega,
	A bilinear estimate with applications to the
	KdV equation,
	\emph{Journal of the American Mathematical Society}
	{\bf 9(2)} (1996), 573--603.
	
	\bibitem{KLV2023}
	R. Killip, T. Laurens, M. Vi{\c{s}}an,
	Sharp well-posedness for the Benjamin-Ono
	equation,
	\href{https://arxiv.org/abs/2304.00124}{arXiv: 2304.00124}
	(2023).
	
	\bibitem{KV2019}
	R. Killip and M. Vi{\c{s}}an,
	KdV is well-posed in $H^{-1}$,
	\emph{Annals of Mathematics. Second Series}
	{\bf 190(1)} (2019), 249--305.
	
	\bibitem{KT2005}
	H. Koch and N. Tzvetkov,
	Nonlinear wave interactions for the Benjamin-Ono equation,
	\emph{International Mathematics Research Notices}
	{\bf 30} (2005), 1833-1847.
	
	\bibitem{LPS2014}
	F. Linares, D. Pilod, J.-C. Saut,
	Dispersive Perturbations of Burgers and Hyperbolic Equations I: Local Theory,
	\emph{SIAM Journal on Mathematical Analysis}
	{\bf 46(2)} (2014), 1505--1537.
	
	\bibitem{Mol2008}
	L. Molinet,
	Global well-posedness in $L^2$ for the periodic Benjamin-Ono equation,
	\emph{American Journal of Mathematics}
	{\bf 130(3)} (2008), 635--683.
	
	\bibitem{MST2001}
	L. Molinet, J. C. Saut, N. Tzvetkov,
	Ill-posedness issues for the Benjamin-Ono and related equations,
	\emph{SIAM Journal on Mathematical Analysis}
	{\bf 33(4)} (2001), 982--988.
	
	\bibitem{MV2014}
	L. Molinet and S. Vento,
	Improvement of the energy method for strongly nonresonant dispersive equations and applications,
	\emph{Analysis \& PDE}
	{\bf 8(6)} (2015), 1455--1495.
	
	\bibitem{MPV2018}
	L. Molinet, D. Pilod, S. Vento,
	On well-posedness for some dispersive perturbations of Burger's equation,
	\emph{Annales de l'Institut Henri Poincaré C, Analyse non linéaire}
	{\bf 35(7)} (2018), 1719--1756.
	
	\bibitem{Sch2020}
	R. Schippa,
	Local and global well-posedness for dispersion generalized
	Benjamin-Ono equations on the circle,
	\emph{Nonlinear Analysis},
	{\bf 196(3)} (2020), 111777.
	
	\bibitem{Tao2004}
	T. Tao,
	Global well-posedness of the Benjamin-Ono equation in {$H^1(\bbr)$},
	\emph{Journal of Hyperbolic Differential Equations},
	{\bf 1(1)} (2004), 27--49.
	
	\bibitem{Whi1974}
	G. B. Whitham,
	Linear and nonlinear Waves,
	John Wiley and Sons, New York, (1999). 
\end{thebibliography}
\end{document}